\documentclass{amsart}
\usepackage{amscd}
\usepackage{amsmath,amssymb,amsfonts}
\usepackage{xypic}
\usepackage[mathcal]{euscript}
\usepackage[cmtip,all,2cell]{xy}
\usepackage{url}
\usepackage{latexsym}
\usepackage[all]{xy}
\usepackage{amsthm}
\usepackage[latin1]{inputenc}
\usepackage{multicol} 
\usepackage{geometry}
\usepackage{hyperref}

\UseAllTwocells

\newtheorem{theorem}{Theorem}[section]
\newtheorem{lemma}[theorem]{Lemma}
\newtheorem{proposition}[theorem]{Proposition}

\newtheorem{corollary}[theorem]{Corollary}

\theoremstyle{definition}
\newtheorem{definition}[theorem]{Definition}
\newtheorem{example}[theorem]{Example}
\newtheorem{counterexample}[theorem]{Counterexample}
\newtheorem{remark}[theorem]{Remark}

\newtheorem{convention}[theorem]{Convention}
\newtheorem{warning}[theorem]{Warning}

\newcommand{\nach}{\longrightarrow}
\newcommand{\auf}{\longmapsto}

\newcommand{\mc}{\mathcal{C}}
\newcommand{\md}{\mathcal{D}}

\newcommand{\mm}{\mathcal{M}}
\newcommand{\nn}{\mathcal{N}}

\newcommand{\la}{\langle}
\newcommand{\ra}{\rangle}
\newcommand{\Fin}{\mathcal{F}in}

\DeclareMathOperator{\id}{id}

\DeclareMathOperator{\nat}{\mathsf{nat}}

\DeclareMathOperator{\Ho}{\mathsf{Ho}}

\DeclareMathOperator{\hhom}{\mathsf{hom}}
\DeclareMathOperator{\Hom}{\mathsf{Hom}}

\DeclareMathOperator{\im}{Im}

\DeclareMathOperator{\sSet}{\mathsf{Set}_\Delta}

\DeclareMathOperator{\colim}{colim}

\DeclareMathOperator{\HoRKan}{HoRKan}
\DeclareMathOperator{\HoLKan}{HoLKan}
\DeclareMathOperator{\RKan}{RKan}

\DeclareMathOperator{\op}{op}
\DeclareMathOperator{\co}{co}
\DeclareMathOperator{\Cat}{\mathsf{Cat}}
\DeclareMathOperator{\CAT}{\mathsf{CAT}}

\DeclareMathOperator{\Dia}{\mathsf{Dia}}
\DeclareMathOperator{\Holim}{Holim}
\DeclareMathOperator{\Hocolim}{Hocolim}
\DeclareMathOperator{\pr}{pr}
\DeclareMathOperator{\DD}{\mathbb{D}}
\DeclareMathOperator{\EE}{\mathbb{E}}

\DeclareMathOperator{\dia}{\mathsf{dia}}
\DeclareMathOperator{\ipush}{i_{_\ulcorner}}
\DeclareMathOperator{\ipull}{i_\lrcorner}
\DeclareMathOperator{\push}{\ulcorner}
\DeclareMathOperator{\pull}{\lrcorner}
\DeclareMathOperator{\cone}{\mathsf{Cone}}
\DeclareMathOperator{\fibre}{\mathsf{Fiber}}

\DeclareMathOperator{\strict}{\mathsf{strict}}

\DeclareMathOperator{\PDer}{\mathsf{PDer}}
\DeclareMathOperator{\Der}{\mathsf{Der}}

\DeclareMathOperator{\ModQ}{\mathsf{ModQ}}

\DeclareMathOperator{\PPr}{\mathsf{Pr}}

\DeclareMathOperator{\cyl}{\mathsf{cyl}}

\entrymodifiers={+!!<0pt,\fontdimen22\textfont2>}

\begin{document}

\title{Derivators, pointed derivators, and stable derivators}
\author{Moritz Groth}
\date{\today} 
\address{Moritz Groth, Radboud University, Nijmegen, Netherlands, email: m.groth@math.ru.nl}

\maketitle

\begin{abstract}
We develop some aspects of the theory of derivators, pointed derivators, and stable derivators. As a main result, we show that the values of a stable derivator can be canonically endowed with the structure of a triangulated category. Moreover, the functors belonging to the stable derivator can be turned into exact functors with respect to these triangulated structures. Along the way, we give a simplification of the axioms of a pointed derivator and a reformulation of the base change axiom in terms of Grothendieck (op)fibration. Furthermore, we have a new proof that a combinatorial model category has an underlying derivator.
\end{abstract}

\setcounter{section}{-1}

\tableofcontents

\section{Introduction and plan}

The theory of stable derivators as initiated by Heller \cite{heller,heller_stable} and Grothendieck \cite{grothendieck} and studied, at least in similar settings, among others, by Franke \cite{franke}, Keller \cite{keller_universal} and Maltsiniotis~\cite{maltsiniotis2}\nocite{maltsiniotis1}, can be motivated by saying that it provides an enhancement of triangulated categories. Triangulated categories suffer the well-known defect that the cone construction is not functorial. A consequence of this non-functoriality of the cone construction is the fact that there is no good theory of homotopy (co)limits for triangulated categories. One can still define these notions, at least in some situations where the functors are defined on categories which are freely generated by a graph. This is the case e.g.\ for the cone construction itself, the homotopy pushout, and the homotopy colimit of a sequence of morphisms. But in all these situations, the `universal objects' are only unique up to \emph{non}-canonical isomorphism.  The slogan used to describe this situation is the following one: diagrams in a triangulated category do not carry sufficient information to define their homotopy (co)limits in a \emph{canonical} way.

But in the typical situations, as in the case of the derived category of an abelian category or in the case of the homotopy category of a stable model resp.\ $\infty -$category, the `model in the background' allows for such constructions in a functorial manner. So, the passage from the model to the derived resp.\ homotopy category truncates the available information too strongly. To be more specific, let~$\mathcal{A}$ be an abelian category such that the derived categories which occur in the following discussion exist. Moreover, let us denote by $C(\mathcal{A})$ the category of chain complexes in $\mathcal{A}$. As usual, let $[1]$ be the ordinal $0\leq 1$ considered as a category $(0\nach 1).$ Hence, for an arbitrary category $\mathcal{C},$ the functor category $\mathcal{C}^{[1]}$ of functors from $[1]$ to $\mathcal{C}$ is the arrow category of $\mathcal{C}.$ With this notation, the cone functor at the level of abelian categories is a functor $\mathsf{C}\colon C(\mathcal{A}^{[1]})\cong C(\mathcal{A})^{[1]}\nach C(\mathcal{A}).$ But to give a \emph{construction} of the cone functor in terms of \emph{homotopical} algebra only, one has to consider more general diagrams. For this purpose,  let $f\colon X\nach Y$ be a morphism of chain complexes in $\mathcal{A}.$ Then the cone $\mathsf{C}f$ of $f$ is the \emph{homotopy} pushout of the following diagram: 
$$\xymatrix{
X\ar[d]\ar[r]^f& Y\\
0&
}$$
At the level of derived categories, the cone construction is again functorial \emph{when considered as a functor} $D(\mathcal{A}^{[1]})\nach D(\mathcal{A}).$ The important point is that one forms the arrow categories \emph{before} passage to the derived categories. Said differently, at the level of derived categories, we have, in general, $D(\mathcal{A}^{[1]})\ncong D(\mathcal{A})^{[1]}.$ Moreover, as we have mentioned, to actually give a construction of this functor one needs apparently also the derived category of diagrams in $\mathcal{A}$ of the above shape and a homotopy pushout functor. More systematically, one should not only consider the derived category of an abelian category but also the derived categories of diagram categories and restriction and homotopy Kan extension functors between them. This is the basic idea behind the notion of a derivator.

But the theory of derivators is more than `only an enhancement of triangulated categories'. In fact, it gives us an alternative axiomatic approach to an abstract homotopy theory (cf.\ Remark~\ref{rem_equivalence}). As in the theory of model categories and $\infty-$categories, there is a certain hierarchy of such structures: the unpointed situation, the pointed situation, and the stable situation. In the classical situation of topology, this hierarchy corresponds to the passage from spaces to pointed spaces and then to spectra. In classical homological algebra, the passage from the derived category of non-negatively graded chain complexes to the unbounded derived category can be seen as a second example for passing from the pointed to the stable situation. In the theory of derivators this threefold hierarchy of structures is also present, and the corresponding notions are then derivators, pointed derivators, and stable derivators. Franke has introduced in \cite{franke} a theory of systems of triangulated diagram categories which is similar to the notion of a stable derivator. The fact that the theory of derivators admits the mentioned threefold hierarchy of structures is one main advantage over the approach of Franke.

The theory described in this paper is not completely new. In particular, it owes a lot to Maltsiniotis who exposed and expanded the foundations of the theory originating with Grothendieck. The first two chapters can be considered as a review of these foundations, although our exposition deviates somewhat from existing ones. In particular, we make systematic use of the `base change formalism' from the very beginning, resulting in a streamlined development of the theory. The more original part of the paper lies in the remaining two chapters in which we use a simplified notion of pointed derivators as we discuss below.

We give a complete and self-contained proof that the values of a stable derivator can be \emph{canonically} endowed with the structure of a triangulated category. Similarly, we show that the functors which are part of the derivator can be \emph{canonically} turned into exact functors with respect to these structures. This is in a sense the main work and will occupy the bulk of this paper.  Moreover, we build on ideas of Franke \cite{franke} from his theory of systems of triangulated diagram categories and adapt them to this alternative set of axioms. Similar ideas are used in Lurie's \cite{HA} on the theory of stable $\infty-$categories.

Along the way we give a simplification of the axioms of a pointed derivator. The usual definition of a pointed derivator, here called a strongly pointed derivator, is formulated using the notion of cosieves and sieves. One usually demands that the homotopy left Kan extension functor~$i_!$ along a cosieve $i$ has itself a left adjoint~$i^?,$ and similarly that the homotopy right Kan extension functor~$j_\ast$ along a sieve~$j$ has a right adjoint~$j^!$. Motivated by algebraic geometry, these additional adjoints are then called exceptional resp.\ coexceptional inverse image functors. We show that this definition can be simplified. It suffices to ask that the underlying category of the derivator is pointed, i.e., has a zero object. This definition is more easily motivated, more intuitive for topologists, and, of course, simpler to check in examples. We give a direct proof of the equivalence of these two notions in Section \ref{sec_pointed}. A second proof of this is given in the stable setting using the fact that recollements of triangulated categories are overdetermined (cf.\ Subsection \ref{subsection_recollement}).

The author is aware of the fact that there will be a written up version of a proof of the existence of these canonical triangulated structures in a future paper by Maltsiniotis. In fact, Maltsiniotis presented an alternative, unpublished variant of Franke's theorem in a seminar in Paris in 2001. He showed that this notion of stable derivators is equivalent to a variant thereof (as used in the thesis of Ayoub \cite{ayoub1,ayoub2}) where the triangulations are part of the notion. Nevertheless, we give this independent account. Moreover, the construction of the suspension functor in \cite{cisinskineeman} and the axioms in \cite{maltsiniotis2} indicate that that proof will use the (co)exceptional inverse image functors. But one point here is to show that these functors are not needed for these purposes.

We now turn to a short description of the content of the paper. In Section \ref{sec_derivators}, we give the central definitions and deduce some immediate consequences of the axioms. The existence of certain very special (co)limits can be explained using the so-called (partial) underlying diagram functors. We develop some aspects of the `base change calculus' (Subsection \ref{subsection_BC}) which is the main tool in most of the proofs in this paper. Using that calculus we are able to characterize derivators by saying that they satisfy base change for Grothendieck (op)fibrations. This in turn is the key ingredient to establish the theoretically important class of examples, that for a derivator $\DD$ the prederivator~$\DD^M$ (cf.\ Example \ref{example_prod}) is also a derivator (Theorem \ref{thm_dertensored}). As a further class of examples, we give a simple, i.e., completely formal, proof that combinatorial model categories have underlying derivators. 

In Section \ref{sec_2cat}, we introduce morphisms and natural transformations in the context of derivators which leads to the $2$-category $\Der$ of derivators. We then turn to homotopy-colimit preserving morphisms and establish some basic facts about them. In particular, again using the fact that derivators satisfy base change for Grothendieck (op)fibrations we show that homotopy Kan extensions in a derivator of the form $\DD^M$ are calculated pointwise (Proposition \ref{prop_kanpointwise}) which will be of some importance in Section \ref{sec_stable}. Moreover, we study in some detail the notion of an adjunction between derivators.

In Section \ref{sec_pointed}, we consider pointed derivators and give the typical examples. We prove that our `weaker' definition of a pointed derivator is equivalent to the `stronger' one using the (co)exceptional inverse image functors (Corollary \ref{cor_pointed}). Moreover, in the pointed context homotopy right Kan extensions along sieves give `extension by zero functors' and dually for cosieves (Proposition \ref{prop_extzero}). We briefly talk about (co)Cartesian squares in a derivator and deduce some properties about them. An important example of this kind of results is the composition and cancellation property of (co)Cartesian squares (Proposition \ref{prop_cancel}). Another one is a `detection result' for (co)Cartesian squares in larger diagrams (Proposition \ref{prop_detect}) which is due to Franke \cite{franke}. We close the section by a discussion of the important suspension, loop, cone, and fiber functors. 

In the final section, we stick to stable derivators for which by definition the classes of coCartesian and Cartesian squares coincide. Some nice consequences of this are that the suspension and the loop morphisms define inverse equivalences, that biCartesian squares satisfy the 2-out-of-3 property, and that we are working in the additive context (Proposition \ref{prop_preadditive} and Corollary \ref{cor_additive}). The main aim of the section is to establish the canonical triangulated structures on the values of a stable derivator (Theorem \ref{thm_triang}). These are preserved by exact morphisms of stable derivators (Proposition \ref{prop_triang}) and, in particular, by the functors belonging to the stable derivator itself (Corollary \ref{cor_canonicallyexact}). In the last subsection, we remark that, given a stable derivator and a (co)sieve, we obtain a recollement of triangulated categories. This reproves, in the stable case, that pointed derivators are `strongly pointed'.

There are three more remarks in order before we begin with the paper. First, we do not develop the general theory of derivators for its own sake and also not in its broadest generality. In this paper, we only develop as much of the general theory as is needed to give complete, self-contained proofs of the mentioned results. Nevertheless, this paper may serve as an introduction to many central ideas in the theory of derivators and no prior knowledge is assumed.

The second remark concerns duality. Many of the statements in this paper have dual statements which also hold true by the dual proof (the reason for this is Example \ref{example_dual}). In most cases, we will not make these statements explicit and we will hardly ever give a proof of both statements. Nevertheless, we allow ourselves to refer to a statement also in cases where, strictly speaking, the dual statement is needed. 

The last remark concerns the terminology employed here. In the existing literature on derivators, the term `triangulated derivator' is used instead of `stable derivator'. We preferred to use this different terminology for two reasons: First, the terminology `triangulated derivator' (introduced by Maltsiniotis in \cite{maltsiniotis2}) is a bit misleading in that no triangulations are part of the initial data. One main point of this paper is to give a proof that these triangulations can be canonically constructed. Thus, from the perspective of the typical distinction between \emph{structures} and \emph{properties} the author does not like the former terminology too much. Second, in the related theories of model categories and $\infty-$categories, corresponding notions exist and are called \emph{stable} model categories and \emph{stable} $\infty-$categories respectively. So, the terminology stable derivator reminds us of the related theories.\footnote{This research was supported by the Deutsche Forschungsgemeinschaft within the graduate program `Homotopy and Cohomology' (GRK 1150)}
\hspace{2mm}

\section{Derivators}\label{sec_derivators}

\subsection{Basic definitions}

As we mentioned in the introduction, the basic idea behind a derivator is to consider simultaneously derived or homotopy categories of diagram categories of different shapes. So, the most basic notion in this business is the following one.

\begin{definition}
A \emph{prederivator} $\DD$ is a strict $2$-functor $\DD\colon\Cat ^{\op} \nach \CAT$.
\end{definition}
\noindent
Here, $\Cat$ denotes the $2$-category of small categories, $\Cat^{\op}$ is obtained from $\Cat$ by reversing the direction of the functors, while $\CAT$ denotes the `$2$-category' of not necessarily small categories. There are the usual set-theoretical problems with the notion of the `$2$-category' $\CAT$ in that this will not be a category enriched over $\Cat$. Since we will never need this non-fact in this paper, we use slogans as the `$2$-category $\CAT$' as a convenient parlance and think instead of a prederivator as a function $\DD$ as we describe it now. Given a prederivator $\DD$ and a functor $u\colon J\nach K$, an application of $\DD$ to $u$ gives us two categories $\DD(J),\;\DD(K),$ and a functor 
$$\DD(u)=u^\ast\colon\DD(K)\nach\DD(J).$$
Similarly, given two functors $u,\: v\colon J\nach K$ and a natural transformation $\alpha\colon u\nach v,$ we obtain an induced natural transformation $\alpha^\ast$ as depicted in the next diagram:
$$
\xymatrix{J \rtwocell^u_v {\alpha}& K& & \DD(K) \rtwocell^{u^\ast}_{v^\ast}{\;\;\alpha^\ast}& \DD(J)}
$$
 \noindent
This datum is compatible with compositions and identities in a strict sense, i.e.,  we have equalities of the respective expressions and not only coherent natural isomorphisms between them. For the relevant basic $2$-categorical notions, which were introduced by Ehresmann in \cite{ehresmann}, we refer to~\cite{kelly_2cat} or to \cite[Chapter 7]{borceux1}, but nothing deep from that theory is needed here.\\
The following examples give an idea of how such prederivators arise. Among these probably the second, third, and fourth one are the examples to have in mind in later sections.

\begin{example}
Every category $\mc$ gives rise to the \emph{prederivator represented by} $\mc$:
$$y(\mc)=\mc\colon J\auf \mc^J$$
Here, $\mc^J$ denotes the functor category of functors from $J$ to $\mc.$ 
\end{example}

\noindent
Anticipating the fact that we have a $2$-category $\PDer$ of prederivators (cf.\ Section \ref{sec_2cat}) we want to mention that this example extends to a ($2$-categorical) Yoneda embedding $y\colon\CAT\nach\PDer.$
In this and the companion papers (\cite{groth_monder, groth_enriched}) we introduce many notions for derivators which are analogs of well-known notions from category theory. Then it will be important to see that these notions are extensions of the classical ones in that both notions coincide on the represented (pre)derivators.

\begin{example}
Let $\mathcal{A}$ be a sufficiently nice abelian category, i.e., such that we can form the derived categories occurring in this example without running into set theoretical problems. Recall that, by definition, the derived category $D(\mathcal{A})$ is the localization of the category of chain complexes at the class of quasi-isomorphisms. For a category~$J$, the functor category $\mathcal{A}^J$ is again an abelian category. In the associated category of chain complexes $C(\mathcal{A}^J)\cong C(\mathcal{A})^J$, quasi-isomorphisms are defined pointwise, so that restriction of diagram functors induce functors on the level of derived categories. Thus, we have the \emph{prederivator} $\DD_{\mathcal{A}}$ \emph{associated to an abelian category} $\mathcal{A}$:
$$\DD_\mathcal{A}\colon J\auf \DD_\mathcal{A}(J)=D(\mathcal{A}^J)$$
\end{example}

The next example assumes some knowledge of model categories. The original reference is \cite{quillen} while a well written, leisurely introduction to the theory can be found in \cite{DwyerSpalinski}. Much more material is treated in the monographs \cite{hovey} and \cite{hirschhorn}. 

\begin{convention}
In this paper model categories are assumed to have limits and colimits of all \emph{small} (as opposed to only finite) diagrams. Furthermore, we do \emph{not} take functorial factorizations as part of the notion of a model category. First, this would be an additional structure on the model categories which is anyhow not respected by the morphisms, i.e., by Quillen functors. Second, this assumption would be a bit in conflict with the philosophy of higher category theory. The category of -say- cofibrant replacements of a given object in a model category is contractible so that any choice is equally good and there is no essential difference once one passes to homotopy categories. We refer to \cite[Part 2]{hirschhorn} for many results along these lines.
\end{convention}

\begin{example}\label{ex_model}
Let $\mathcal{M}$ be a cofibrantly generated model category. Recall that one of the good things about cofibrantly generated model categories is that diagram categories $\mathcal{M}^J$ can be endowed with the so-called \emph{projective model structure}. In more detail, let us call a natural transformation of $\mathcal{M}-$valued functors a \emph{projective fibration} if all components are fibrations, and similarly a \emph{projective weak equivalence}  if all components are weak equivalences in $\mathcal{M}$. A \emph{projective cofibration} is a map which has the left-lifting-property with respect to all maps which are simultaneously projective fibrations and projective weak equivalences. With these definitions, $\mathcal{M}^J$ is again a model category and we can thus consider the associated homotopy category. Recall that the canonical functor $\gamma\colon\mathcal{M}\nach\Ho(\mathcal{M})$ from $\mm$ to its homotopy category is a 2-localization. This means, that $\gamma$ induces for every category~$\mc$ an \emph{isomorphism of categories}
$$\gamma^\ast\colon\mc^{\Ho(\mathcal{M})}\nach\mc^{(\mathcal{M},W)}$$ 
\noindent
where the right-hand-side denotes the full subcategory of $\mc^{\mathcal{M}}$ spanned by the functors which send weak equivalences to isomorphisms. Moreover, since projective weak equivalences are defined as levelwise weak equivalences, these are preserved by restriction of diagram functors. By the universal property of the localization functors the restriction of diagram functors descend uniquely to the homotopy categories. Thus, given such a cofibrantly generated model category $\mathcal{M}$, we can form the \emph{prederivator} $\DD_{\mathcal{M}}$ \emph{associated to} $\mathcal{M}$ if we set 
$$\DD_\mathcal{M}\colon J\auf \DD_\mathcal{M}(J)=\Ho(\mathcal{M}^J).$$
\end{example}

A similar example can be given using the theory of $\infty-$categories (aka.\ quasi-categories, weak Kan complexes), i.e., of simplicial sets satisfying the inner horn extension property. These were originally introduced by Boardman and Vogt in their work \cite{BoardVogt} on homotopy invariant algebraic structures. Detailed accounts of this theory are given in the tomes due to Joyal \cite{joyal1,joyal2,joyal3,joyal4} and Lurie \cite{HTT,HA}. A short exposition of many of the central ideas and also of the philosophy of this theory can be found in \cite{groth_infinity}.

\begin{example}
Let $\mathcal{C}$ be an $\infty-$category and let $K\in\sSet$ be a simplicial set. Then one can show that the simplicial mapping space $\mathcal{C}^K_\bullet=\hhom_{\sSet}(\Delta^\bullet \times K, \mathcal{C})$ is again an $\infty-$category (as opposed to a more general simplicial set). This follows from the fact that the Joyal model structure~(\cite{joyal1}) on the category of simplicial sets is Cartesian. We can hence vary the simplicial set $K$ and consider the associated homotopy categories $\Ho(\mathcal{C}^K)$. Using the nerve functor $N$ which gives us a fully faithful embedding of the category $\Cat$ in the category $\sSet$ of simplicial sets, we thus obtain the \emph{prederivator} $\DD_{\mathcal{C}}$ \emph{associated to the} $\infty-$\emph{category} $\mathcal{C}$:
$$\DD_\mathcal{C}\colon J\auf \DD_{\mathcal{C}}(J)=\Ho \big( \mathcal{C}^{N(J)}\big)$$
The functoriality of this construction follows from Theorem 5.14 of \cite{joyal3}.
\end{example}

The last example which we are about to mention now does not seem to be too interesting in its own right. But as we will see later it largely reduces the amount of work in many proofs (cf.\ Theorem~\ref{thm_dertensored}).

\begin{example}\label{example_prod}
Let $\DD$ be a prederivator and let $M$ be a fixed category. Then the assignment 
$$\DD^M\colon \Cat^{\op}\nach \CAT\colon J\auf \DD^M(J)=\DD(M\times J)$$
is again a prederivator. Similarly, given a functor $u\colon L\nach M$ we obtain a morphism of $2$-functors $u^\ast\colon \DD^M\nach \DD^L.$ There is a notion of morphisms of prederivators (cf.\ Section \ref{sec_2cat} and more specifically Example \ref{ex_morphisms}) and it is easy to see that the pairing $(M,\DD)\auf \DD^M$ is actually functorial in both variables. Moreover, we have coherent isomorphisms $(\DD^L)^M\cong\DD^{L\times M}$ and $\DD^e\cong \DD$ for the terminal category~$e$.
\end{example}

\begin{remark}
{\rm i)} In some situations, in particular under certain finiteness conditions, one does not wish to consider diagrams of arbitrary shapes but only of a certain kind (e.g.\ finite, finite-dimensional, posets). There is a notion of a \emph{diagram category} $\Dia$ which is a 2-subcategory of $\Cat$ having certain closure properties. Correspondingly, there is then the associated notion of a \emph{prederivator of type} $\Dia$. We preferred to not give these definitions at the very beginning since we wanted to start immediately with the development of the theory. Once the main results are established we check which properties have been used and come back to this point (cf.\ the discussion before Definition \ref{def_diagram}). So, the reader is invited to replace `a (pre)derivator' by `a (pre)derivator \emph{of type} $\Dia$' throughout this paper. An example of the usefulness of this more flexible notion is given by Keller in \cite{keller_exact} where he shows that there is a stable derivator associated to an exact category in the sense of Quillen \cite{quillen_ktheory} if one restricts to finite directed diagrams.\\
{\rm ii)} There is an additional remark concerning the definition of a prederivator. In our setup a pre-derivator is a $2$-functor $\DD\colon\Cat^{\op}\nach\CAT$ as opposed to a more general pseudo-functor (which is for example used in \cite{franke}). More specifically, we insisted on the fact that $\DD$ preserves identities and compositions in a strict sense and not only up to coherent natural isomorphisms. Since all examples showing up in nature have this stronger functoriality we are fine with this notion. However, from the perspective of `homotopical invariance of structures', a definition based on pseudo-functors would be better: let $\DD$ be a prederivator and let us be given a category $\mathcal{E}_J$ for each small category~$J.$ Let us moreover assume that we are given equivalences of categories $\DD(J)\nach\mathcal{E}_J.$ Then, in general, we cannot use the equivalences to obtain a prederivator $\EE$ with $\EE(J)=\mathcal{E}_J$ such that the equivalences of categories assemble to an equivalence of prederivators. This would only be the case if the equivalences are, in fact, isomorphisms which --by the basic philosophy of category theory-- is a too strong notion. Nevertheless, for the sake of a simplification of the exposition we preferred to stick to $2$-functors but want to mention that everything we do here can also be done with pseudo-functors.
\end{remark}

Let now $\DD$ be a prederivator and let $u\colon J\nach K$ be a functor. Motivated by the above examples we call the induced functor $\DD(u)=u^\ast\colon \DD(K)\nach \DD(J)$ a \emph{restriction of diagram functor} or \emph{precomposition functor}. As a special case of this, let $J=e$ be the terminal category, i.e., the category with one object and identity morphism only. For an object $k$ of $K$, we denote by $k\colon e\nach K$ the unique functor sending the unique object of $e$ to $k$. Given a prederivator $\DD$, we obtain, in particular, for each object $k\in K$ an associated functor $k^\ast\colon \DD(K)\nach \DD(e)$ which takes values in the \emph{underlying category} $\DD(e)$. Let us call such a functor an \emph{evaluation functor}. For a morphism $f\colon X\nach Y$ in~$\DD(K)$ let us write $f_k\colon X_k\nach Y_k$ for its image under $k^\ast.$

\begin{definition}
Let $\DD$ be a prederivator and let $u\colon J\nach K$ be a functor.\\
i) The prederivator $\DD$ admits \emph{homotopy left Kan extensions along} $u$ if the induced functor $u^\ast$ has a left adjoint:
$$(u_!=\HoLKan_u,u^\ast)\colon \DD(J)\rightharpoonup \DD(K)$$
The prederivator $\DD$ admits \emph{homotopy colimits of shape} $J$ if the functor $p_J^\ast$ induced by $p_J\colon J\nach e$ has a left adjoint:
$$({p_J}_!=\Hocolim _J,p_J^\ast)\colon \DD(J)\rightharpoonup\DD(e)$$
ii) The prederivator $\DD$ admits \emph{homotopy right Kan extensions along} $u$ if the induced functor $u^\ast$ has a right adjoint:
$$(u^\ast,u_\ast=\HoRKan_u)\colon \DD(K)\rightharpoonup \DD(J)$$
The prederivator $\DD$ admits \emph{homotopy limits of shape} $J$ if the functor $p_J^\ast$ induced by $p_J\colon J\nach e$ has a right adjoint:
$$(p_J^\ast,{p_J}_\ast=\Holim _J)\colon \DD(e)\rightharpoonup\DD(J)$$
\end{definition}

Recall from classical category theory, that under cocompleteness assumptions left Kan extensions can be calculated pointwise by certain colimits, and similarly that under completeness assumptions right Kan extensions can be calculated pointwise by certain limits \cite[p.\ 237]{maclane}. More precisely, consider $u\colon J\nach K$ and $F\colon J\nach \mc$ where $\mc$ is a complete category:
$$
\xymatrix{
J\ar[r]^-F \ar[d]_-u & \mc  \\
K\ar@{-->}[ru]_-{\RKan _u(F)} &
}
$$
Then the right Kan extension $\RKan_u (F)$ of $F$ along $u$ exists and can be described using \emph{Kan's formula} \cite{kan_adjoint} as 
$$\RKan _u (F)_k\cong \lim_{J_{k/}}\pr^\ast (F)= \lim_{J_{k/}}F\circ \pr,\qquad k\in K.$$
In the above formula, we have used the following notation. Let $u\colon J\nach K$ be a functor and let~$k$ be an object of $K$. Then one can form the \emph{slice category} $J_{k/}$ \emph{of objects} $u$\emph{-under k}. An object in this category is a pair $(j,f)$ consisting of an object $j\in J$ together with a morphism $f\colon  k\nach u(j)$ in~$K$. Given two such objects $(j_1,f_1)$ and $(j_2,f_2),$ a morphism $g\colon (j_1,f_1)\nach (j_2,f_2)$ is a morphism $g\colon j_1\nach j_2$ in $J$ such that the obvious triangle in~$K$ commutes. Dually, one can form the \emph{slice category} $J_{/k}$ \emph{of objects} $u$\emph{-over} $k$. In both cases, there are canonical functors 
$$\pr\colon J_{k/}\nach J\qquad \mbox{and} \qquad \pr\colon J_{/k}\nach J$$
forgetting the morphism component. We will not distinguish these projection morphisms notationally but it will always be clear from the context which projection morphism we are considering.  A dual formula holds for left Kan extension in the case of a cocomplete target category $\mc$ and will not be made explicit.

The corresponding property for \emph{homotopy} Kan extensions holds in the case of model categories (cf.\ Subsection \ref{subsec_ex}) and will be demanded axiomatically for a derivator. In order to be able to formulate this axiom, we have to talk about base change morphisms. For this purpose, let $\DD$ be a prederivator and consider a natural transformation of functors $\alpha\colon w\circ u \nach u'\circ v$. By an application of $\DD$, we thus have the following two squares on the left:
$$
\xymatrix{
J \ar[r]^-v\ar[d]_-u& J' \ar[d]^-{u'} && \DD(J) & \DD(J') \ar[l]_-{v^\ast}&
\DD(K)&\DD(J)\ar[l]_-{u_\ast} & \DD(J') \ar[l]_-{v^\ast}&\\
K \ar[r]_-w & K'\xtwocell[-1,-1]{}\omit&&\DD(K) \ar[u]^-{u^\ast} & \DD(K')\ar[l]^-{w^\ast}\ar[u]_-{{u'}^\ast}\xtwocell[-1,-1]{}\omit&
&\DD(K) \ar@/^1.0pc/[lu]^-=\ar[u]^-{u^\ast}\xtwocell[-1,-1]{}\omit & \DD(K')\ar[l]^-{w^\ast}\ar[u]_-{{u'}^\ast}\xtwocell[-1,-1]{}\omit& \DD(J')\ar@/_1.0pc/[lu]_-= \ar[l]^-{u'_\ast}\xtwocell[-1,-1]{}\omit
}
$$
\noindent
Let us assume that $\DD$ admits homotopy right Kan extensions along $u$ and $u'.$ We denote any chosen adjoints and the corresponding adjunction morphisms by
$$(u^\ast,u_\ast),\qquad \eta \colon\id\nach u_\ast\circ u^\ast,\qquad\mbox{and}\qquad \epsilon \colon u^\ast\circ u_\ast\nach \id$$
\noindent
in the case of $u$ and similarly in the case of $u'.$ This can be used to extend our square to the upper right diagram in which the additional natural transformation are given by the respective adjunction morphisms. We can thus define the natural transformation $\alpha_\ast$ by pasting this diagram to a single natural transformation. Spelling this out, $\alpha_\ast$ is given by the following composition of natural transformations:
$$\xymatrix{
      w^\ast\circ u'_\ast \ar[rr]^{\alpha_\ast}\ar[d]_{\eta} & & u_\ast\circ v^\ast  &\DD(J)\ar[d]_{u_\ast} & \DD(J') \ar[l]_{v^\ast}\ar[d]^{u'_\ast}\xtwocell[1,-1]{}\omit\\
  u_\ast\circ u^\ast\circ w^\ast \circ u'_\ast  \ar[rr]_{\alpha^\ast} & & u_\ast\circ v^\ast\circ u'^\ast \circ u'_\ast \ar[u]_{\epsilon'}&\DD(K)  & \DD(K')\ar[l]^{w^\ast}
}
$$
\noindent
This natural transformation $\alpha_\ast$ is called `the' \emph{Beck-Chevalley transformed 2-cell associated to} $\alpha$. Since this construction is very important in this paper let us make explicit the dual construction. So, let us consider a natural transformation $\alpha \colon u'\circ v\nach w\circ u$ as in:
$$
\xymatrix{
J \xtwocell[1,1]{}\omit \ar[r]^v\ar[d]_u& J' \ar[d]^{u'} && \DD(J) \xtwocell[1,1]{}\omit& \DD(J') \ar[l]_{v^\ast}&
\DD(K)\xtwocell[1,1]{}\omit&\DD(J)\xtwocell[1,1]{}\omit\ar[l]_-{u_!} & \DD(J')\xtwocell[1,1]{}\omit \ar[l]_-{v^\ast}& \\
K \ar[r]_w & K' && \DD(K) \ar[u]^{u^\ast} & \DD(K') \ar[u]_{{u'}^\ast}\ar[l]^{w^\ast} &
&\DD(K) \ar@/^1.0pc/[lu]^-=\ar[u]^-{u^\ast} & \DD(K')\ar[l]^-{w^\ast}\ar[u]_-{{u'}^\ast}& \DD(J')\ar@/_1.0pc/[lu]_-= \ar[l]^-{u'_!}
}
$$
Under the assumption that the prederivator admits homotopy left Kan extensions along $u$ and $u'$ we can proceed as above and define the natural transformation $\alpha_!$ by pasting, i.e., as follows:
$$\xymatrix{
      u_! \circ v^\ast \ar[rr]^{\alpha_!}\ar[d]_{\eta'} & & w^\ast\circ u'_!  &\DD(J)\ar[d]_{u_!} & \DD(J') \ar[l]_{v^\ast}\ar[d]^{u'_!}\\
  u_!\circ v^\ast\circ u'^\ast \circ u'_!  \ar[rr]_{\alpha ^\ast} & & u_!\circ u^\ast\circ w^\ast \circ u'_! \ar[u]_{\epsilon}&\DD(K) \xtwocell[-1,1]{}\omit & \DD(K')\ar[l]^{w^\ast}
}
$$
\noindent
This natural transformation $\alpha_!$ is again called `the' \emph{Beck-Chevalley transformed 2-cell associated to}~$\alpha$. 

In both cases, we constructed a new $2$-cell by choosing certain adjoint functors and then composing with the adjunction morphisms. It is immediate that the result depends on these choices only up to natural isomorphism. This technique will be developed a bit more systematically in Subsection \ref{subsection_BC} but see also \cite{groth_scderivator} where some aspects from classical category theory are treated from this perspective.

At the very moment, we are interested in the following situation. Let $u\colon J\nach K$ be a functor and $k\in K$ an object. Identifying $k$ again with the corresponding functor $k:e\nach K$, we have the following two natural transformations $\alpha$ in the context of the slice constructions:
$$\xymatrix{
   J_{k/}\ar[r]^\pr\ar[d]_{p_{J_{k/}}} & J \ar[d]^u& & J_{/k}\ar[r]^\pr\ar[d]_{p_{J_{/k}}}\xtwocell[1,1]{}\omit& J\ar[d]^u\\
      e \ar[r]_k& K \xtwocell[-1,-1]{}\omit& & e\ar[r]_k & K
}
$$
The components of $\alpha$ at $(j,f\colon k\nach u(j))$ resp.\ $(j,f\colon u(j)\nach k)$ are $f$ in both cases. Assuming $\DD$ to be a prederivator admitting the necessary homotopy Kan extensions, we thus obtain the following Beck-Chevalley transformed $2$-cells $\alpha_\ast$ and $\alpha_!:$
$$
\xymatrix{
\DD(J_{k/})\ar[d]_{\Holim_{J_{k/}}} & \DD(J)\ar[l]_-{\pr ^\ast}\ar[d]^{u_\ast}\xtwocell[1,-1]{}\omit&&\DD(J_{/k})\ar[d]_{\Hocolim_{J_{/k}}} & \DD(J)\ar[l]_-{\pr ^\ast}\ar[d]^{u_!}\\
\DD(e)& \DD(K)\ar[l]^-{k^\ast}&&\DD(e)\xtwocell[-1,1]{}\omit & \DD(K)\ar[l]^-{k^\ast}
}
$$
\noindent
Asking these natural transformations to be isomorphisms is a convenient way to axiomatize Kan's formulas. With these preparations we can give the central definition of a derivator.

\begin{definition}\label{def_derivator}
A prederivator $\DD$ is called a \emph{derivator} if it satisfies the following axioms:\\
(Der1) For two categories $J_1$ and $J_2$, the functor $\DD(J_1\sqcup J_2)\nach \DD(J_1)\times\DD(J_2)$ induced by the inclusions is an equivalence of categories. Moreover, the category $\DD(\emptyset)$ is not the empty category.\\
(Der2) A morphism $f\colon X\nach Y$ in $\DD(J)$ is an isomorphism if and only if $f_j\colon X_j\nach Y_j$ is an isomorphism in $\DD(e)$ for every object $j\in J.$\\
(Der3) For every functor $u\colon J\nach K$, there are homotopy left and right Kan extensions along $u$:
$$(u_!,u^\ast)\colon\DD(J)\rightharpoonup\DD(K)\qquad\mbox{and}\qquad(u^\ast,u_\ast)\colon\DD(K)\rightharpoonup\DD(J).$$
\noindent
(Der4) For every functor $u\colon J\nach K$ and every $k\in K$, the morphisms
$$\Hocolim_{J_{/k}}\pr^\ast(X)\stackrel{\alpha_!}{\nach} u_!(X)_k\qquad \mbox{and}\qquad u_\ast(X)_k\stackrel{\alpha_\ast}{\nach}\Holim_{J_{k/}}\pr^\ast (X)$$
are isomorphisms for all $X\in\DD(J)$.
\end{definition}

A few remarks on the axioms are in order. The first axiom says of course that a diagram on a disjoint union is completely determined by its restrictions to the direct summands. The second part of the first axiom is included in order to exclude the `empty derivator' as an example. But it will also imply the existence of initial and final objects (cf.\ Proposition \ref{prop_(co)limits}). The second axiom can be motivated from the examples as follows. A natural transformation is an isomorphism if and only if it is pointwise an isomorphism. Similarly, in the context of an abelian category, there is the easy fact that a morphism of chain complexes in a functor category is a quasi-isomorphism if and only if it is a quasi-isomorphism at each object. Moreover, in the context of model categories, whatever model structure one establishes on a diagram category with values in a model category, one certainly wants the class of weak equivalences to be defined pointwise. Finally, the corresponding result for $\infty-$categories is established by Joyal as Theorem 5.14 in \cite{joyal3}. The last two axioms of course encode a  `homotopical bicompleteness property' together with Kan's formulas. One could easily develop a more general theory of prederivators which are only homotopy (co)complete or even only have a certain class of homotopy (co)limits. 

\begin{example}\label{example_bicompletecat}
Let $\mathcal{C}$ be a category. The represented prederivator $y(\mathcal{C})\colon J\auf \mathcal{C}^J$ is a derivator if and only if $\mc$ is bicomplete. Thus, the $2$-category of bicomplete categories is embedded into the $2$-category of derivators.
\end{example}

The idea is of course that the derivator encodes additional structure on its values. One nice feature of this approach is that this structure does not have to be chosen but its existence can be deduced from the axioms. Note that all axioms are of the form that they demand a \emph{property}; the only actual \emph{structure} is the given prederivator. This is similar to the situation of additive categories where the enrichment in abelian groups can uniquely be deduced from the fact that the underlying category has certain exactness properties. We will come back to this point later in the context of stable derivators (cf.\ Remark \ref{rem_advantages}). 

As a first example of this `higher structure' we give the following example. We will pursue this more systematically from Subsection \ref{subsection_BC} on. Let $J$ be a category and consider the coproduct $J\sqcup J$ together with the codiagonal and the inclusion functors:

$$
\xymatrix{
J \ar[rd]_-{\id_J}\ar[r]^-{i_1}& J\sqcup J\ar[d]^-{\nabla_J} & J\ar[dl]^-{\id_J}\ar[l]_-{i_2}\\
& J &
}
$$

\begin{proposition}\label{prop_(co)limits}
Let $\DD$ be a derivator and let $J$ be a category.\\
{\rm i)} The value of $\DD$ at the empty category $\emptyset$ is trivial, i.e., $\DD(\emptyset)$ is equivalent to $e.$\\
{\rm ii)} The category $\DD(J)$ admits an initial object $\emptyset$ and a terminal object $\ast$.\\
{\rm iii)} The category $\DD(J)$ admits finite coproducts and finite products.
\end{proposition}
\begin{proof}
i) Considering the disjoint union $\emptyset=\emptyset \sqcup \emptyset$, (Der1) implies that we have an equivalence given by the diagonal functor $\DD(\emptyset)\nach \DD(\emptyset)\times\DD(\emptyset)$. Thus, all morphism sets are either empty sets or singletons. The first case would deduce a contradiction to the fact that the diagonal is a bijection on path components. Thus, $\DD(\emptyset)$ is trivial and we will hence denote any object of $\DD(\emptyset)$ by $0$.\\
ii) Consider the unique empty functor $\emptyset_J:\emptyset \nach J$ and apply (Der3) in order to obtain left resp.\ right adjoints 
$${\emptyset_J}_!\colon\DD(\emptyset)\nach \DD(J),\qquad {\emptyset_J}_\ast\colon \DD(\emptyset)\nach \DD(J).$$
Since a left (right) adjoint preserves initial (final) objects, the image of any object $0$ under ${\emptyset_J}_!$ (${\emptyset_J}_\ast$) is an initial (terminal) object in $\DD(J).$ Let us denote any such image by $\emptyset$ respectively $\ast.$\\
iii) By (Der1), we have an equivalence of categories $(i_1^\ast,i_2^\ast):\DD(J\sqcup J)\stackrel{\simeq}{\nach}\DD(J)\times\DD(J)$. Choose an inverse equivalence $k$ and consider the following diagram:
$$
\xymatrix{
\DD(J)\times\DD(J)\ar@<0.8ex>[r]^-k&\DD(J\sqcup J)\ar@<0.8ex>[r]^-{{\nabla_J}_!}\ar@<0.8ex>[l]^-{(i_1^\ast,i_2^\ast)}& \DD(J)\colon\Delta_{\DD(J)}\ar@<0.8ex>[l]^-{\nabla_J^\ast}
}
$$
\noindent
Since the right adjoint is the diagonal functor, ${\nabla_J}_!\circ k$ gives a coproduct. Similarly, ${\nabla_J}_\ast\circ k$ will define a product functor on $\DD(J).$
\end{proof}

We want to emphasize that, in general, the values of a derivator only have very few \emph{categorical} (co)limits. In order to relate this to the homotopical variants and also for later purposes, let us introduce the underlying diagram functors and their partial variants. We saw already that an object~$m\in M$ induces an evaluation functor $m^\ast\colon\DD(M)\nach\DD(e).$ Similarly, a morphism $\alpha\colon m_1\nach m_2$ in $M$ can be considered as a natural transformation of the corresponding classifying functors and thus gives rise to
$$\xymatrix{
e \rtwocell^{m_1}_{m_2}{\alpha} & M, & &\DD(M)\rtwocell^{m_1^\ast}_{m_2^\ast}{\;\;\alpha^\ast}&\DD(e).
}$$
\noindent
Under the categorical exponential law which can be written in a suggestive form as
$$\Big({\DD(e)^{\DD(M)}}\Big)^M\cong \DD(e)^{M\times\DD(M)}\cong \Big(\DD(e)^M\Big)^{\DD(M)},$$
we hence obtain an \emph{underlying diagram functor}
$$\dia_M\colon\DD(M)\nach \DD(e)^M.$$
Similarly, given a product $M\times J$ of two categories and $m\in M$, we can consider the corresponding functor
$$m\times\id_J\colon J\cong e\times J\nach M\times J.$$
Following the same arguments as above, we obtain a \emph{partial underlying diagram functor}
$$\dia _{M,J}\colon \DD(M\times J)\nach \DD(J)^M.$$
\noindent
Thus, the natural isomorphism $M\cong M\times e$ induces an identification of $\dia_M$ and $\dia_{M,e}.$ Now, the functor $p_M\colon M\nach e$ gives rise to the following diagram
$$
\xymatrix{
\DD(M)\ar[rr]^{\dia_M}\ar@/_1.3pc/[d]_{\Hocolim_M}\ar@/^1.3pc/[d]^{\Holim _M}&& \DD(e)^M\ar@/_1.3pc/@{.>}[d]\ar@/^1.3pc/@{.>}[d]\\
\DD(e)\ar[u]^{p_M^\ast}\ar[rr]_\id&&\DD(e)\ar[u]_{\Delta_M}
}
$$
which commutes in the sense that we have $\dia_M\circ\: p_M^\ast =  \Delta_M\colon\DD(e)\nach\DD(e)^M.$ If the underlying diagram functor $\dia_M$ happens to be an equivalence for a certain category $M$, then also $\Delta_M$ has adjoints on both sides, i.e., the category $\DD(e)$ has then (co)limits of shape $M.$ Similar remarks apply to the case of the partial underlying diagram functor $\dia_{M,J}$ where we would then deduce a conclusion about the category $\DD(J).$ Now, axiom (Der1) implies that the partial underlying diagram functors 
$$\dia_{\emptyset,J}\colon\DD(\emptyset)\nach\DD(J)^\emptyset=e\qquad\mbox{and}\qquad
\dia_{e\sqcup e, J}\colon\DD(J\sqcup J)\nach\DD(J)\times\DD(J)$$
are equivalences. This explains why we were able to deduce Proposition \ref{prop_(co)limits} from the axioms but, in general, do not have other categorical (co)limits.

Although, in general, we do not want to assume that also other partial underlying diagram functors are equivalences, the following definition is very important.
This definition again emphasizes the importance of the distinction between the categories $\DD(K)$ and $\DD(e)^K.$

\begin{definition}
A derivator $\DD$ is called \emph{strong} if the partial underlying diagram functor 
$$\dia_{[1],J}\colon\DD([1]\times J)\nach\DD(J)^{[1]}$$
is full and essentially surjective for each category $J.$
\end{definition}

\begin{remark}
The strongness property of a derivator is a bit harder to motivate. It can be checked that the derivators associated to model categories are strong. Moreover, since the partial underlying diagram functors are isomorphisms for represented derivators, these are certainly also strong. Thus, the strongness property is satisfied by the naturally occurring derivators. 

In this paper, the strongness will play a key role in the construction of the triangulated structures on the values of a stable derivator. The point is that the strongness property allows one to lift morphisms in the underlying category $\DD(e)$ to objects in the category $\DD([1])$ where we can apply certain constructions to it. Similarly, it allows us to lift morphisms in $\DD(e)^{[1]}$ to morphisms in $\DD([1])$ or even to objects in $\DD([1]\times[1]).$ 

But it is not only the case for the stable context that the strongness property is convenient. Already in the context of pointed derivators it is very helpful. This property allows the construction of fiber and cofiber sequences associated to a morphism in the underlying category of a strong, pointed derivator. Similarly, one might expect that in later developments of the theory this property will also be useful in the unpointed context. Nevertheless, we follow Maltsiniotis in not including the strongness as an axiom of the basic notion of a derivator. 

Moreover, it might be helpful to consider variants of the definition. Given a family $\mathcal{F}$ of small categories, we define a derivator $\DD$ to be $\mathcal{F}$-\emph{strong} if the partial underlying diagram functors $\dia_{M,J}$ are full and essentially surjective for all $M\in\mathcal{F}$ and all categories $J.$ Heller considered in \cite{heller} the case where $\mathcal{F}$ consists of all finite, free categories.
\end{remark}

Let us quickly recall the dualization process for derivators. As the author was confused for a while about the different dualizations for $2$-categories we will give some details. The point is that given a 2-category~$\mc$ we obtain a new 2-category $\mc^{\op}$ be inverting the direction of the 1-morphisms and we get a further 2-category $\mc^{\co}$ by inverting the direction of the 2-morphisms. Moreover, these operations can be combined so that given a $2$-category using the various dualizations we obtain $4$ different $2$-categories (more generally, an $n$-category has $2^n$ different dualizations).

Let us explain these dualizations more conceptually, i.e., from the perspective of enriched category theory. First, we can consider `the enrichment level'  $\Cat$ as a \emph{symmetric} monoidal category. The formation of opposite categories can be performed in the context of enriched categories as soon as the enrichment level is symmetric monoidal (cf.\ to \cite[Section 6.2]{borceux2}). Thus, we can form the dual of a $2$-category $\mc$ as a category enriched over $\Cat.$ The result of this dualization is the $2$-category $\mc^{\op}$ in which the $1$-morphisms have changed direction. Alternatively, since the Cartesian monoidal structure on the $1$-category $\Cat$ behaves well with dualization, there is a second way of dualizing a general $2$-category. More precisely, we can consider the dualization of small categories as a monoidal functor $(-)^{\op}\colon\Cat\nach\Cat$ with respect to the Cartesian structures. Since any monoidal functor induces a base change functor at the level of enriched categories (cf.\ to \cite[Section 6.4]{borceux2}), we obtain a $2$-category $\mc^{\co}.$ This $2$-category is obtained from $\mc$ by inverting the direction of $2$-morphisms. Finally, applying both dualizations to $\mc$ we obtain the $2$-category $\mc^{\co,\op}=\mc^{\op,\co}.$ Remarking that the dualization 2-functor of categories $J\mapsto J^{\op}$ inverts the direction of the natural transformations but keeps the direction of the functors we thus can make the following definition.

\begin{definition}
Let $\DD$ be a prederivator, then we define the dual prederivator $\DD^{\op}$ by the following diagram:
$$
\xymatrix{
\Cat^{\op}\ar[r]^{\DD^{\op}}\ar[d]_{(-)^{\op}}& \CAT\\
\Cat^{\op,\co}\ar[r]_{\DD}&\CAT^{\co}\ar[u]_{(-)^{\op}}
}
$$
\end{definition}

\begin{example}\label{example_dual}
A prederivator $\DD$ is a derivator if and only if its dual $\DD^{\op}$ is a derivator.
\end{example}

This result implies that in many general statements about derivators and morphisms between derivators we only have to prove claims about --say-- homotopy left Kan extensions while the corresponding claim for homotopy right Kan extensions follows by duality.

\subsection{Homotopy exact squares and some properties of homotopy Kan extensions}\label{subsection_BC}

In this subsection we want to establish some lemmas about the formation of Beck-Chevalley transformed natural transformations (as it is used in the definition of a derivator). A very convenient fact is the nice behavior of this formalism with respect to pasting. Since the Beck-Chevalley transformation itself already uses pasting of natural transformations let us quickly recall the latter. For this purpose, let us consider the following two diagrams in $\CAT:$
$$
\xymatrix{
\mc_1\xtwocell[1,1]{}\omit&\mc_2\xtwocell[1,1]{}\omit\ar[l]_{v_1}&\mc_3\ar[l]_{v_2}&&\mc_1&\mc_2\ar[l]_{v_1}&\mc_3\ar[l]_{v_2}\\
\md_1\ar[u]^{u_1}&\md_2\ar[u]_{u_2}\ar[l]^{w_1}&\md_3\ar[l]^{w_2}\ar[u]_{u_3}
&&\md_1\ar[u]^{u_1}&\md_2\ar[u]_{u_2}\ar[l]^{w_1}\xtwocell[-1,-1]{}\omit&\md_3\ar[u]_{u_3}\ar[l]^{w_2}
\xtwocell[-1,-1]{}\omit
}
$$
If we call in both diagrams the natural transformations $\alpha_1$ and $\alpha_2$ respectively then we can form the following composite natural transformations:
$$
\alpha_1\odot\alpha_2=\alpha_1w_2\cdot v_1\alpha_2\qquad\mbox{and}\qquad
\alpha_2\odot\alpha_1= v_1\alpha_2\cdot\alpha_1w_2
$$
By definition, these natural transformations are obtained by \emph{pasting} of the respective diagrams. The chosen notation reminds us of the order the transformations show up in the composition. This procedure can also be applied to larger diagrams if all natural transformations `point in the same direction'.

To give an example one can use these pasting diagrams to depict the triangular identities for an adjunction $(L,R,\eta,\epsilon)\colon\mc\rightharpoonup\md$ as follows:
$$
\xymatrix{
\md\xtwocell[1,1]{}\omit&\mc\xtwocell[1,1]{}\omit\ar[l]_-L & \ar@{}[dr]|{=}& \md &&
\mc&\md\ar[l]_-R &\ar@{}[dr]|{=}&\mc\\
&\md \ar@/^1.0pc/[lu]^-=\ar[u]^R & \mc\ar[l]^-L \ar@/_1.0pc/[lu]_-= & \mc\ar[u]_L&& 
&\mc \ar@/^1.0pc/[lu]^-=\ar[u]^L\xtwocell[-1,-1]{}\omit& \md\ar@/_1.0pc/[lu]_-= \ar[l]^-R\xtwocell[-1,-1]{}\omit&\md\ar[u]_R
}
$$
Here, the unlabeled natural transformations are the adjunction morphisms and we simplified the notation of an identity transformation by only displaying the corresponding functor.

Let us now turn to the formalism of Beck-Chevalley transformed natural transformations. For this purpose, let us consider two natural transformations $\alpha_1$ and $\alpha_2$ in $\CAT$ as indicated in:
$$
\xymatrix{
\mc_1\xtwocell[1,1]{}\omit&\mc_2\ar[l]_v&&\mc_1&\mc_2\ar[l]_v\\
\md_1\ar[u]^{u_1}&\md_2\ar[u]_{u_2}\ar[l]^w&&\md_1\ar[u]^{u_1}&\md_2\ar[u]_{u_2}\ar[l]^w
\xtwocell[-1,-1]{}\omit
}
$$
Under the assumption that the vertical functors have left adjoints ${u_i}_!$ we obtain \emph{Beck-Chevalley (BC) transformed 2-cells} ${\alpha_1}_!$ by pasting as depicted in the following diagram on the left. Similarly, the existence of right adjoints~${u_i}_\ast$ allows us to construct \emph{Beck-Chevalley (BC) transformed 2-cells}~${\alpha_2}_\ast$ by pasting as shown in the diagram on the right:
$$
\xymatrix{
\md_1\xtwocell[1,1]{}\omit&\mc_1\xtwocell[1,1]{}\omit\ar[l]_-{{u_1}_!} & \mc_2\xtwocell[1,1]{}\omit \ar[l]_-v& &
\md_1&\mc_1\ar[l]_-{{u_1}_\ast} & \mc_2 \ar[l]_-v&\\
&\md_1 \ar@/^1.0pc/[lu]^-=\ar[u]^-{u_1} & \md_2\ar[l]^-w\ar[u]_-{u_2}& \mc_2\ar@/_1.0pc/[lu]_-= \ar[l]^-{{u_2}_!}& 
&\md_1 \ar@/^1.0pc/[lu]^-=\ar[u]^-{u_1}\xtwocell[-1,-1]{}\omit & \md_2\ar[l]^-w\ar[u]_-{u_2}\xtwocell[-1,-1]{}\omit& \mc_2\ar@/_1.0pc/[lu]_-= \ar[l]^-{{u_2}_\ast}\xtwocell[-1,-1]{}\omit
}
$$
As an result of these pastings we obtain natural transformations ${\alpha_1}_!$ and ${\alpha_2}_\ast$ as follows:
$$
\xymatrix{
\mc_1\ar[d]_{{u_1}_!}&\mc_2\ar[l]_v\ar[d]^{{u_2}_!}&&\mc_1\ar[d]_{{u_1}_\ast}&\mc_2\ar[l]_v\ar[d]^{{u_2}_\ast}
\xtwocell[1,-1]{}\omit\\
\md_1\xtwocell[-1,1]{}\omit&\md_2\ar[l]^w&&\md_1&\md_2\ar[l]^w
}
$$
\noindent
Again, the resulting natural transformations depend on some choices but only up to natural isomorphism. Using this terminology, the natural transformations showing up in axiom (Der4) of Definition \ref{def_derivator} are given by the Beck-Chevalley transformed $2$-cells associated to the following ones respectively:
$$
\xymatrix{
\DD(J_{/k})\xtwocell[1,1]{}\omit & \DD(J) \ar[l]_-{\pr^\ast} && \DD(J_{k/}) & \DD(J) \ar[l]_{\pr^\ast}\\
\DD(e) \ar[u]^-{p^\ast} & \DD(K)\ar[l]^-{k^\ast}\ar[u]_-{u^\ast} && \DD(e) \ar[u]^{p^\ast} & \DD(K) \ar[u]_{u^\ast}\ar[l]^{k^\ast}\xtwocell[-1,-1]{}\omit
}
$$

There is a certain ambiguity in the definition of the Beck-Chevalley transformation. If --say-- in the case of $\alpha_1\colon v\circ u_2\nach u_1\circ w$ the horizontal functors admit right adjoints we obtain a different transformed 2-cell. However, this gives rise to a conjugate natural transformation (see Lemma~\ref{lemma_BCtwoisos}).

We will commit the following abuse of notation. Let us assume we are given a derivator $\DD$ and let us assume we have $\alpha_1=\DD(\beta_1)$ and $\alpha_2=\DD(\beta_2)$ for corresponding natural transformations $\beta_i$ in $\Cat.$ Then, for simplicity, we will denote the Beck-Chevalley transformed $2$-cell associated to $\alpha_1$ and $\alpha_2$ by ${\beta_1}_!$ and ${\beta_2}_\ast$ respectively. Thus, we agree on the following short-hand-notation:
$$\beta_!=(\beta^\ast)_!\qquad\mbox{\and}\qquad\beta_\ast=(\beta^\ast)_\ast$$
Moreover, let us say that the transformations $\beta_!$ and $\beta_\ast$ are obtained by `base change'.

We give some important classes of examples for this base change formalism to indicate its broad applicability beyond the purposes in this paper.

\begin{example}
The following natural transformations respectively formations can be obtained by the base change formalism:\\
{\rm i)} adjunction units and adjunction counits\\
{\rm ii)} the formation of conjugate transformations (cf.\ \cite{maclane})\\
{\rm iii)} the natural maps expressing that a functor preserves certain (co)limits or Kan extensions\\
{\rm iv)} in the context of triangulated categories, the exact structure on a functor adjoint to an exact functor\\
{\rm v)} in the context of a triangulated category with a t-structure $(\mathcal{T},\tau_{\geq 0},\tau_ {\leq 0})$ the natural isomorphisms $\tau_{\leq m}\circ\tau_{\geq n}\nach\tau_{\geq n}\circ\tau_{\leq m}$ for $m,\: n\in\mathbb{Z}$\\
{\rm vi)} in the context of monoidal categories, the monoidal structure on a functor right adjoint to a comonoidal functor and dually
\end{example}

Let us now collect a few key technical lemmas on this formalism. The proofs of these results are quite formal and we will only include the easiest one here (cf.\ however to \cite{groth_scderivator}) to advertise the convenience of drawing these pasting diagrams.
A key fact for the remainder of this paper is the good behavior of Beck-Chevalley transformed $2$-cells with respect to pasting. Thus, let us again consider two pasting situations in $\CAT$
$$
\xymatrix{
\mc_1\xtwocell[1,1]{}\omit&\mc_2\xtwocell[1,1]{}\omit\ar[l]_{v_1}&\mc_3\ar[l]_{v_2}&&\mc_1&\mc_2\ar[l]_{v_1}&\mc_3\ar[l]_{v_2}\\
\md_1\ar[u]^{u_1}&\md_2\ar[u]_{u_2}\ar[l]^{w_1}&\md_3\ar[l]^{w_2}\ar[u]_{u_3}
&&\md_1\ar[u]^{u_1}&\md_2\ar[u]_{u_2}\ar[l]^{w_1}\xtwocell[-1,-1]{}\omit&\md_3\ar[u]_{u_3}\ar[l]^{w_2}
\xtwocell[-1,-1]{}\omit
}
$$
with natural transformations $\alpha_1$ and $\alpha_2.$ 

\begin{lemma}\label{lemma_BCpasting}
Let us consider the above diagrams in $\CAT$ and let us assume that the vertical functors have left adjoints resp.\ right adjoints. Then there are the following relations among the Beck-Chevalley transformed $2$-cells:
$$
(\alpha_1\odot\alpha_2)_!={\alpha_2}_!\odot{\alpha_1}_!\qquad\mbox{respectively}\qquad (\alpha_2\odot\alpha_1)_\ast={\alpha_1}_\ast\odot{\alpha_2}_\ast
$$
\end{lemma}
\begin{proof}
We give a proof of the first equation. Unraveling definitions this boils down to depicting the transformation ${\alpha_2}_!\odot{\alpha_1}_!$ as
$$
\xymatrix{
\md_1\xtwocell[1,1]{}\omit&\mc_1\xtwocell[1,1]{}\omit\ar[l]_-{{u_1}_!} & \mc_2\xtwocell[1,1]{}\omit \ar[l]_-{v_1}& &&\\
&\md_1 \ar@/^1.0pc/[lu]^-=\ar[u]^-{u_1} & \md_2\xtwocell[1,1]{}\omit\ar[l]^-{w_1}\ar[u]_-{u_2}& \mc_2\xtwocell[1,1]{}\omit\ar@/_1.0pc/[lu]_-= \ar[l]^-{{u_2}_!}&\mc_3\xtwocell[1,1]{}\omit\ar[l]^-{v_2}&\\
&&&\md_2 \ar@/^1.0pc/[lu]^-=\ar[u]^-{u_2} & \md_3\ar[l]^-{w_2}\ar[u]_-{u_3}& \mc_3\ar@/_1.0pc/[lu]_-= \ar[l]^-{{u_3}_!}
}
$$
and then using a triangular identity in order to obtain $(\alpha_1\odot\alpha_2)_!.$
\end{proof}

Thus, base change is compatible with horizontal pasting and there is a similar such result for vertical pasting. Note, that we only have a compatibility with respect to pasting and not `a functoriality'. In particular, it can (and will) be the case that we start with a commutative square but that the associated transformed $2$-cells are even not isomorphisms. Nevertheless, this compatibility with respect to pasting combined with the $2$-out-of-$3$-property for isomorphisms will be a key ingredient in many proofs of this paper. 

It is useful to remark that the assignments $\alpha\mapsto\alpha_!$ and $\alpha\mapsto\alpha_\ast$ are inverse to each other in a certain precise sense. So, let us again consider a natural transformation~$\alpha$ as depicted below. We can then iterate the Beck-Chevalley transformation as indicated in the following diagrams. By doing so we first obtain $\alpha_!$ and then $(\alpha_!)_\ast:$
$$
\xymatrix{
\mc_1\xtwocell[1,1]{}\omit&\mc_2\ar[l]_v\ar@{}[rrd]|{\mapsto}&& \mc_1\ar[d]_{{u_1}_!}&\mc_2\ar[l]_v\ar[d]^{{u_2}_!}\ar@{}[rrd]|{\mapsto}&& \mc_1\xtwocell[1,1]{}\omit&\mc_2\ar[l]_v\\
\md_1\ar[u]^{u_1}&\md_2\ar[u]_{u_2}\ar[l]^w&& \md_1\xtwocell[-1,1]{}\omit&\md_2\ar[l]^w&& \md_1\ar[u]^{u_1}&\md_2\ar[u]_{u_2}\ar[l]^w
}
$$
\noindent
In this situation we have the following lemma which also has its obvious dual form. The proof is left to the reader but can also be found in \cite{groth_scderivator}.

\begin{lemma}\label{lemma_BCinverse}
In the above situation we have the equality $\alpha=(\alpha_!)_\ast\colon v\circ u_2\nach u_1\circ w.$
\end{lemma}

These two lemmas can for example be combined to deduce that for two conjugate natural transformations we have that one of them is an isomorphism if and only if the other is. We want to collect a last lemma about this formalism which proves to be helpful in the discussion of adjunctions (or adjunctions of two variables as in \cite{groth_monder}) and equivalences of derivators. 

\begin{lemma}\label{lemma_BCtwoisos}
Let $\alpha\colon v\circ u_2\nach u_1\circ w$ be a natural transformation in $\CAT$ such that the  functors~$u_i$ have left adjoints ${u_i}_!$ for $i=1,\:2$ and such that the functors $v$ and~$w$ have right adjoints $v_\ast$ and $w_\ast$ respectively. The natural transformations $\alpha_!\colon {u_1}_!\circ v\nach w\circ {u_2}_!$ and $\alpha_\ast\colon u_2\circ w_\ast\nach v_\ast\circ u_1$ are conjugate. In particular, $\alpha_!$ is an isomorphism if and only if $\alpha_\ast$ is an isomorphism
\end{lemma}

More details on the above examples and also the proofs of the last two lemmas can be found in \cite{groth_scderivator}. Let us now apply this formalism in the context of a derivator.

\begin{definition}
Let $\DD$ be a derivator and let us consider a natural transformation $\alpha$ as indicated in the following square in $\Cat:$
$$
\xymatrix{
J_1\ar[r]^v\ar[d]_{u_1}\xtwocell[1,1]{}\omit &J_2\ar[d]^{u_2}\\
K_1\ar[r]_w&K_2
}
$$
The square is $\DD$-\emph{exact} if the base change $\alpha_!\colon{u_1}_!\circ v^\ast\nach w^\ast\circ {u_2}_!$ (or, by Lemma \ref{lemma_BCtwoisos}, equivalently $\alpha_\ast\colon u_2^\ast\circ w_\ast\nach v_\ast\circ u_1^\ast$) is a natural isomorphism. The square is called \emph{homotopy exact} if it is~$\DD$-exact for all derivators $\DD.$
\end{definition}

We will also apply the terminology of $\DD$-exact squares in the context of a prederivator $\DD$ admitting the necessary homotopy Kan extensions. For a derivator $\DD$ it follows immediately from Lemma \ref{lemma_BCpasting} that $\DD$-exact squares are stable under horizontal and vertical pasting. 

\begin{warning}
We want to include a warning on a certain risk of ambiguity if the natural transformation~$\alpha$ under consideration happens to be an isomorphism. In that case it can (and will) happen that the Beck-Chevalley transformation of~$\alpha$ is an isomorphism without this being the case for ~$\alpha^{-1}$ (cf.\ for example to Subsection \ref{subsection_cocont}). In particular, this can happen for commutative squares. Thus, in case there is a risk of ambiguity we will always give a direction to natural isomorphisms and even to identity transformations (cf.\ for example to Proposition \ref{prop_BCfibration}). 
\end{warning}

We will next illustrate the notion of homotopy exact squares by giving some examples which are central to the development of the theory of derivators (for a more systematic discussion we refer to \cite{maltsiniotis_exact}). Using the $2$-functoriality of prederivators, the following is immediate.

\begin{lemma}\label{lemma_predadj}
Let $\DD$ be a prederivator and let $(L,R)\colon J\rightharpoonup K$ be an adjunction. Then we obtain an adjunction 
$$(R^\ast,L^\ast)\colon \DD(J)\rightharpoonup \DD(K).$$
Moreover, if $L$ (resp.\ $R$) is fully faithfully, then so is $R^\ast$ (resp.\ $L^\ast$).
\end{lemma}
\begin{proof}
Every $2$-functor of the variance of a prederivator sends an adjunction $(L,R,\eta,\epsilon)$ to an adjunction $(R^\ast,L^\ast,\eta^\ast,\epsilon^\ast).$
\end{proof}
\noindent
In the statement of this lemma we allowed ourselves the following abuse of notation. Strictly speaking an adjunction is not determined by the two functors~$L$ and~$R$ but one also has to specify either of the following: the natural isomorphism~$\phi$ of the morphism sets, the unit~$\eta$, or the counit~$\epsilon$. In order to simplify the notation we nevertheless allowed ourselves (and also will do so in the remainder of the paper) to write $(L,R)$ instead of $(L,R,\phi),\:(L,R,\eta),$ or $(L,R,\epsilon).$

A related result using the notion of homotopy exact squares can be formulated as follows. This result expresses the cofinality of right adjoints.

\begin{proposition}\label{prop_cofinality}
For a right adjoint functor $R\colon J\nach K$ the following square is homotopy exact:
$$
\xymatrix{
J\ar[r]^R\ar[d]_{p_J}\xtwocell[1,1]{}\omit&K\ar[d]^{p_K}\\
e\ar[r]_{\id}&e
}
$$
\end{proposition}
\begin{proof}
We have to show that the natural transformation ${p_J}_!R^\ast\nach {p_K}_!$ is an isomorphism for an arbitrary derivator. But it can be checked that this natural transformation is the conjugate transformation of the identity $\id\colon p_K^\ast\nach L^\ast p_J^\ast=(p_JL)^\ast=p_K^\ast$ concluding the proof.
\end{proof}

Thus, for a derivator $\DD$, a right adjoint functor $R\colon J\nach K,$ and an object $X\in\DD(K)$ we have a canonical isomorphism
$$\Hocolim_JR^\ast(X)\stackrel{\cong}{\nach}\Hocolim_KX.$$
For later reference, let us spell out the important special case where the right adjoint $R=t\colon e\nach K$ just specifies a terminal object in $K.$ The second part of the lemma follows immediately by passing to the conjugate of the natural transformation showing up in the first part.

\begin{lemma}\label{lemma_terminal}
Let $\DD$ be a derivator and let $K$ be a category admitting a terminal object $t.$ \\
{\rm i)} For $X\in \DD(K)$ we have a natural isomorphism $X_t\stackrel{\cong}{\nach}\Hocolim _KX.$\\
{\rm ii)} We have a canonical isomorphism of functors $p_K^\ast\stackrel{\cong}{\nach}t_\ast.$ The essential image of  $t_\ast$ consists of precisely those objects for which all structure maps in the underlying diagram are isomorphisms.
\end{lemma}

Here is another important result about homotopy Kan extensions.

\begin{proposition}\label{prop_honestext}
Let $u\colon J\nach K$ be a fully faithful functor, then the following square is homotopy exact:
$$
\xymatrix{
J\ar[r]^{\id}\ar[d]_{\id}&J\ar[d]^u\\
J\ar[r]_u&K
}
$$
Thus, the adjunction morphisms $\eta\colon \id \nach u^\ast u_!$ and $\epsilon\colon u^\ast u_\ast\nach \id$ are isomorphisms, i.e., homotopy Kan extension functors along fully faithful functors are fully faithful.
\end{proposition}
\begin{proof}
Since isomorphisms can be detected pointwise we can reduce our task to showing that the following pasting is homotopy exact for all $j\in J:$
$$
\xymatrix{
J_{/j}\ar[r]^{\pr}\ar[d]_p\xtwocell[1,1]{}\omit & J\ar[r]^{\id}\ar[d]_{\id}\xtwocell[1,1]{}\omit &J\ar[d]^u\\
e\ar[r]_j&J\ar[r]_u&K
}
$$
But, the fully faithfulness of $u$ implies that we have an isomorphism $J_{/u(j)}\nach J_{/j}$ so that it suffices (e.g.\ by Proposition \ref{prop_cofinality}) to show that the next pasting is homotopy exact:
$$
\xymatrix{
J_{/u(j)}\ar[r]\ar[d]_p\xtwocell[1,1]{}\omit&J_{/j}\ar[r]^{\pr}\ar[d]_p\xtwocell[1,1]{}\omit & J\ar[r]^{\id}\ar[d]_{\id}\xtwocell[1,1]{}\omit &J\ar[d]^u\\
e\ar[r]&e\ar[r]_j&J\ar[r]_u&K
}
$$
But this is guaranteed by axiom (Der4).
\end{proof}

Since we now know that, for fully faithful $u\colon J\nach K$, the homotopy Kan extension functors $u_!,u_\ast\colon \DD(J)\nach\DD(K)$ are fully faithful, we would like to obtain a characterization of the objects in the essential images. The point of the next lemma is that one only has to control the adjunction morphisms at arguments $k\in K-u(J).$

\begin{lemma}\label{lemma_essim}
Let $\DD$ be a derivator, $u\colon J\nach K$ a fully faithful functor, and $X\in \DD(K).$\\
{\rm i)} $X$ lies in the essential image of $u_!$ if and only if the adjunction counit $\epsilon\colon u_!u^\ast\nach \id$ induces an isomorphism $\epsilon _k \colon u_!u^\ast (X)_k\nach X_k$ for all $k\in K-u(J).$\\
{\rm ii)} $X$ lies in the essential image of $u_\ast$ if and only if the adjunction unit $\eta \colon \id \nach u_\ast u^\ast$ induces an isomorphism $\eta _k\colon X_k\nach  u_\ast u^\ast (X)_k$ for all $k\in K-u(J).$
\end{lemma}
\begin{proof}
We give a proof of ii), so let us consider the adjunction $(u^\ast,u_\ast)\colon\DD(K)\rightharpoonup\DD(J).$ By Proposition \ref{prop_honestext}, $u_\ast$ is fully faithful. Thus, $X\in \DD(K)$ lies in the essential image of $u_\ast$ if and only if the adjunction unit $\eta\colon X\nach u_\ast u^\ast X$ is an isomorphism. Since isomorphisms can be tested pointwise, this is the case if and only if we have an isomorphism $\eta _k\colon X_k\nach u_\ast u^\ast (X)_k$ for all $k \in K.$ For the converse direction, one of the triangular identities for our adjunction reads as $\id=\epsilon u^\ast\cdot u^\ast\eta$. Thus, with $\epsilon$ also $u^\ast\eta$ is an isomorphism so that it suffices to check at points which do not lie in the image.
\end{proof}

There are two important classes of fully faithful functors where the essential image of homotopy Kan extensions can be characterized more easily. So let us give their definition.

\begin{definition}
Let $u\colon J\nach K$ be a fully faithful functor which is injective on objects.\\
i) The functor $u$ is called a \emph{cosieve} if whenever we have a morphism $u(j)\nach k$ in $K$ then $k$ lies in the image of $u.$\\
ii) The functor $u$ is called a \emph{sieve} if whenever we have a morphism $k\nach u(j)$ in $K$ then $k$ lies in the image of $u.$
\end{definition}

The following proposition and a variant for the case of pointed derivators (cf.\ Proposition \ref {prop_extzero}) will be frequently used throughout this paper.

\begin{proposition}\label{prop_extzero_unpointed}
Let $\DD$ be a derivator.\\
{\rm i)} Let $u\colon J\nach K$ be a cosieve, then the homotopy left Kan extension $u_!$ is fully faithful and $X\in \DD(K)$ lies in the essential image of $u_!$ if and only if $X_k\cong \emptyset$ for all $k\in K-u(J).$\\
{\rm ii)} Let $u\colon J\nach K$ be a sieve, then the homotopy right Kan extension $u_\ast$ is fully faithful and $X\in \DD(K)$ lies in the essential image of $u_\ast$ if and only if $X_k\cong \ast$ for all $k\in K-u(J).$
\end{proposition}
\begin{proof}
We give a proof of i). The statement about the fully faithfulness of $u_!$ follows from the fully faithfulness of the cosieve and Proposition \ref{prop_honestext}. To describe the essential image we use the criterion of Lemma \ref{lemma_essim}. But for $k\in K-u(J)$ we have
$$u_!u^\ast(X)_k\cong\Hocolim_{J_{/k}}\pr^\ast u^\ast(X)= \Hocolim_\emptyset\pr^\ast u^\ast (X)=\emptyset.$$
In this sequence, the isomorphism is given by (Der4), the first equality follows from the definition of a cosieve, and the second equality follows from the description of initial objects. Thus $\epsilon _k\colon u_!u^\ast (X)_k\nach X_k$ is an isomorphism for all $k\in K-u(J)$ if and only if $X_k\cong \emptyset$ for all $k \in K-u(J).$
\end{proof}

\subsection{Examples}\label{subsec_ex}
The first aim of this subsection consists of establishing a class of examples which is very important for theoretical issues. Namely, we want to show that with $\DD$ also the prederivator~$\DD^M$ is a derivator for an arbitrary small category $M.$ We will then show that combinatorial model categories have underlying derivators.

For the first point, the hardest part will be to show that $\DD^M$ again satisfies axiom (Der4). In order to achieve this we include a short detour and establish some reformulations of this axiom which also are of independent interest and will again be used further below. 

Let us begin by considering the following pullback diagram in $\Cat:$
$$
\xymatrix{
J_1\ar[r]^{v}\ar[d]_{u_1}&J_2\ar[d]^{u_2}\\
K_1\ar[r]_{w}&K_2\xtwocell[-1,-1]{}\omit
}
$$ 
For the notion of Grothendieck (op)fibrations we refer to \cite[Section 8.1]{borceux2} or \cite{vistoli}.

\begin{proposition}\label{prop_BCfibration}
Using the above notation, a pullback diagram is homotopy exact, if $u_2$ is a Grothendieck fibration or if $w$ is a Grothendieck opfibration.
\end{proposition}
\begin{proof}
We give the proof in the case where $u_2$ is a Grothendieck fibration. For a derivator $\DD$ we thus have to show that the canonical map $\id_\ast\colon w^\ast {u_2}_\ast\nach{u_1}_\ast v^\ast$ is a natural isomorphism. Since isomorphisms can be tested pointwise, (Der4) implies that it suffices to show that the following pasting is a homotopy exact square for all $k_1\in K_1:$
$$
\xymatrix{
(J_1)_{k_1/}\ar[r]^{\pr}\ar[d]_p&J_1\ar[r]^{v}\ar[d]_{u_1}&J_2\ar[d]^{u_2}\\
e\ar[r]_{k_1}&K_1\ar[r]_{w}\xtwocell[-1,-1]{}\omit&K_2\xtwocell[-1,-1]{}\omit
}
$$
Since our diagram in $\Cat$ is a pullback diagram, we deduce that with $u_2$ also $u_1$ is a Grothendieck fibration. Thus, the canonical functor
$$c\colon (J_1)_{k_1}\nach (J_1)_{k_1/}\colon\qquad j_1\mapsto (j_1, k_1\stackrel{\id}{\nach} u_1(j_1))$$
is a left adjoint functor (\cite{quillen_ktheory}). Here, we denote by $(J_1)_{k_1}$ the fiber of $u_1$ over $k_1,$ i.e., the subcategory of $J_1$ consisting of all objects sent to $k_1$ and all morphisms sent to $\id_{k_1}.$ Now, Lemma~\ref{lemma_BCpasting} and Proposition \ref{prop_cofinality} imply that the above pasting is homotopy exact if and only if this is the case for the pasting in the following left diagram:
$$
\xymatrix{
(J_1)_{k_1}\ar[r]^c\ar[d]_p&(J_1)_{k_1/}\ar[r]^{\pr}\ar[d]_p&J_1\ar[r]^{v}\ar[d]_{u_1}&J_2\ar[d]^{u_2}&
(J_1)_{k_1}\ar[r]^w\ar[d]_p&(J_2)_{w(k_1)}\ar[r]^c\ar[d]_p&(J_2)_{w(k_1)/}\ar[r]^{\pr}\ar[d]_p&J_2\ar[d]^{u_2}\\
e\ar[r]&e\ar[r]_{k_1}\xtwocell[-1,-1]{}\omit&K_1\ar[r]_{w}\xtwocell[-1,-1]{}\omit&K_2\xtwocell[-1,-1]{}\omit&
e\ar[r]&e\ar[r]\xtwocell[-1,-1]{}\omit&e\ar[r]_{w(k_1)}\xtwocell[-1,-1]{}\omit&K_2\xtwocell[-1,-1]{}\omit
}
$$
It is easy to check that the above two pastings define the same natural transformation. Thus, by exactly the same arguments again it suffices to show that the square
$$
\xymatrix{
(J_1)_{k_1}\ar[r]^w\ar[d]_p&(J_2)_{w(k_1)}\ar[d]^p\\
e\ar[r]&e\xtwocell[-1,-1]{}\omit
}
$$
is homotopy exact. But, since we started with a pullback diagram, $w$ restricted in this way is an isomorphism of categories so that our claim follows (e.g.\ again by Proposition \ref{prop_cofinality}).
\end{proof}

We will refer to this proposition by saying that a derivator satisfies base change for Grothendieck (op)fibrations. This proposition allows us to establish the next theorem. 

\begin{theorem}\label{thm_dertensored}
Let $\DD$ be a derivator and let $M$ be a small category. Then the prederivator $\DD^M\colon \Cat^{\op}\nach\CAT\colon K\mapsto \DD(M\times K)$ is a derivator.
\end{theorem}
\begin{proof}
The axioms (Der1)-(Der3) are immediate so we only have to establish axiom (Der4) for $\DD^M.$ By duality, it suffices to give the proof for the case of homotopy right Kan extensions. In other words, we have to show that the square
$$
\xymatrix{
M\times J_{k/}\ar[r]\ar[d]&M\times J\ar[d]\\
M\times e\ar[r]&M\times K\xtwocell[-1,-1]{}\omit
}
$$
is $\DD$-exact. But this $2$-cell can be obtained as the pasting of the following diagram
$$
\xymatrix{
M\times J_{k/}\ar[r]\ar[d]&M\times J\ar[d]\\
M\times K_{k/}\ar[r]\ar[d]&M\times K\ar[d]\xtwocell[-1,-1]{}\omit\\
M\times e\ar[r]&M\times K\xtwocell[-1,-1]{}\omit
}
$$
in which the upper square is a pullback diagram such that the bottom horizontal arrow is a Grothendieck opfibration. Thus, by Proposition \ref{prop_BCfibration} it suffices to show that the lower square given by the natural transformation~$\alpha_2$ is $\DD$-exact. We claim that it would suffice to show that the pasting obtained by the following left diagram is $\DD$-exact for every $m\in M:$
$$
\xymatrix{
M_{m/}\times K_{k/}\ar[r]\ar[d]&M\times K_{k/}\ar[r]\ar[d]&M\times K\ar[d]&&(M\times K)_{(m,k)/}\ar[r]\ar[d]&M\times K\ar[d]\\
e\ar[r]_m&M\times e\ar[r]\xtwocell[-1,-1]{}\omit&M\times K\xtwocell[-1,-1]{}\omit&&e\ar[r]_-{(m,k)}&M\times K\xtwocell[-1,-1]{}\omit
}
$$ 
Using this claim it then suffices to observe that this pasting is naturally isomorphic to the square on the right-hand-side which is $\DD$-exact by Kan's formula. So, it remains to establish our claim for which purpose we call the natural transformation on the left~$\alpha_1.$ The associated diagram obtained by base change looks like
$$
\xymatrix{
\DD(M_{m/}\times K_{k/})\ar[d] & \DD(M\times K_{k/})\ar[l]\ar[d]\xtwocell[1,-1]{}\omit& \DD(M\times K)\ar[l]_-{\pr^\ast}\ar[d]\xtwocell[1,-1]{}\omit\\
\DD(e)& \DD(M\times e)\ar[l]^-{m^\ast}&\DD(M\times K)\ar[l]
}
$$
in which the 2-cells are given by ${\alpha_1}_\ast$ and ${\alpha_2}_\ast$ respectively. The compatibility of base change with respect to pasting thus gives us the following equations:
$$
(\alpha_1\odot\alpha_2)_\ast\qquad =\qquad {\alpha_2}_\ast\odot{\alpha_1}_\ast\qquad =\qquad {\alpha_2}_\ast \pr^\ast \cdot\; m^\ast {\alpha_1}_\ast
$$
Now, the canonical isomorphism $M_{m/}\times K_{k/}\cong (M\times K_{k/})_{m/}$ and axiom (Der4) imply that ${\alpha_2}_\ast$ is an isomorphisms. Using the fact that isomorphisms are detected pointwise we can now conclude as follows: ${\alpha_1}_\ast$ is an isomorphism if and only if $m^\ast {\alpha_1}_\ast$ is an isomorphism for all $m\in M$ which is the case if and only if $(\alpha_1\odot\alpha_2)_\ast$ is an isomorphism for all $m\in M.$ Thus it was indeed enough to show that the above pasting on the left is $\DD$-exact.
\end{proof}

Thus, whenever we want to establish a general result about the values $\DD(M)$ of a derivator $\DD$ we may assume that we are considering the underlying category of a derivator since we can always pass from $\DD$ to $\DD^M.$

Let us note that the conclusion of Proposition \ref{prop_BCfibration} is actually equivalent to (Der4). Moreover, there is a further reformulation using a `symmetric variant of Kan's formulas'. More specifically, let us consider the following square in $\Cat:$
$$
\xymatrix{
(u_1/u_2)\ar[r]^{\pr_1}\ar[d]_{\pr_2}\xtwocell[1,1]{}\omit &J_1\ar[d]^{u_1}\\
J_2\ar[r]_{u_2}&K_2
}
$$
Here, the category $(u_1/u_2)$ is the \emph{comma category} where an object is a triple 
$$(j_1,j_2,\alpha\colon u_1(j_1)\nach u_2(j_2)),\qquad j_1\in J_1,\quad j_2\in J_2$$
and the functors $\pr_i$ are the obvious projection functors. The arrow component of such an object defines the natural transformation depicted in the diagram. If we specialize to $J_1=e$ or $J_2=e$ we get back the diagrams showing up in the pointwise calculation of Kan extensions.

\begin{proposition}\label{prop_3BC}
Let $\DD$ be a prederivator which satisfies the axioms \textnormal{(Der1)-(Der3)}. Then the following three statements are equivalent:\\
{\rm i)} The prederivator $\DD$ is a derivator, i.e., it also satisfies \textnormal{(Der4)}.\\
{\rm ii)} The prederivator $\DD$ satisfies base change for Grothendieck (op)fibrations.\\
{\rm iii)} The prederivator $\DD$ satisfies base change for comma categories, i.e., the squares associated to comma categories are $\DD$-exact.
\end{proposition}
\begin{proof}
By Proposition \ref{prop_BCfibration} we already know that ii) is implied by i). Let us show that also the converse holds, i.e., we want to show that the square
$$\xymatrix{
   J_{k/}\ar[r]^\pr\ar[d]_{p_{J_{k/}}} & J \ar[d]^u\\
      e \ar[r]_k& K \xtwocell[-1,-1]{}\omit
}
$$
is $\DD$-exact if we assume ii). But the reasoning in this case is a simplified version of the arguments in the proof of Theorem \ref{thm_dertensored}. So, it remains to show that i) and iii) are equivalent. One direction is immediate by specializing comma categories to slice categories so we only have to prove that (Der4) implies iii). Using similar reduction arguments as in the last proof (including the behavior of base change with respect to pasting and the fact that isomorphisms are detected pointwise) it suffices to show that the following pasting is $\DD$-exact for all objects $j_2\in J_2:$
$$
\xymatrix{
(u_1/u_2)_{/j_2}\ar[r]^{\pr}\ar[d]_p\xtwocell[1,1]{}\omit & (u_1/u_2)\ar[r]^{\pr_1}\ar[d]_{\pr_2}\xtwocell[1,1]{}\omit &J_1\ar[d]^{u_1}\\
e\ar[r]_{j_2}&J_2\ar[r]_{u_2}&K
}
$$
Now, there is a canonical functor $R\colon {J_1}_{/u_2(j_2)}\nach (u_1/u_2)_{/j_2}$ which is defined by:
$$(j_1,u_1(j_1)\nach u_2(j_2))\qquad\mapsto\qquad\big((j_1,u_1(j_1)\nach u_2(j_2),j_2),j_2\stackrel{\id}{\nach}j_2\big)$$
This functor can be checked to define a right adjoint so that by Proposition \ref{prop_cofinality} it suffices to show that the pasting in the following diagram is $\DD$-exact:
$$
\xymatrix{
{J_1}_{/u_2(j_2)}\ar[r]^R\ar[d]_p\xtwocell[1,1]{}\omit&(u_1/u_2)_{/j_2}\ar[r]^{\pr}\ar[d]_p\xtwocell[1,1]{}\omit & (u_1/u_2)\ar[r]^{\pr_1}\ar[d]_{\pr_2}\xtwocell[1,1]{}\omit &J_1\ar[d]^{u_1}\\
e\ar[r]&e\ar[r]_{j_2}&J_2\ar[r]_{u_2}&K
}
$$
But this pasting is precisely the square used to calculate homotopy Kan extensions along $u_1$ so that we can conclude by (Der4).
\end{proof}

Let us now turn to the second important class of examples of derivators, namely the ones associated to nice model categories. This is included not only for the sake of completeness but also because our proof differs from the one given in \cite{cisinski}. Our proof is completely self-dual and is simpler in that it does not make use of the explicit description of the generating (acyclic) projective cofibrations of a diagram category associated to a cofibrantly generated model category. We restrict attention to the following situation.

\begin{definition}
A model category $\mathcal{M}$ is called \emph{combinatorial} if it is cofibrantly generated and if the underlying category is presentable.
\end{definition}

This class of model categories was introduced by Smith and is studied e.g.\ in \cite{HTT,rosicky_comb,beke,dugger_combinatorial}. For background on cofibrantly generated model categories we refer to \cite{hovey}. The theory of presentable categories was initiated by Gabriel and Ulmer in \cite{gabrielulmer}. Further references to this theory are \cite{borceux2,adamekrosicky}. One basic idea of the presentability assumption is the following one. The presentability imposes beyond the bicompleteness a certain `smallness condition' on a category which has at least two important consequences. The first one is that the usual set-theoretic problems occurring when one considers functor categories disappear at least if one restricts attention to colimit-preserving functors. But this is anyhow the adapted class of morphism to be studied in this context. Moreover, in the world of presentable categories one can focus more on conceptual ideas than on technical points of certain arguments: a functor between presentable categories is a left adjoint if and only if it is colimit-preserving. The usual `solution set condition' of Freyd's adjoint functor theorem is automatically fulfilled in this context. For more comments in this direction see Subsection 2.6 in \cite{groth_infinity}, where these ideas are discussed in the context of presentable $\infty-$categories. Important examples of presentable categories are the categories of sets, simplicial sets, all presheaf categories (and, more generally, all Grothendieck toposes), algebraic categories as well as the Grothendieck abelian categories and the category $\Cat$. A non-example is the category of topological spaces although this can be repaired if one sticks to the `really convenient category' (Smith) of $\Delta$-generated spaces (\cite{rosicky_convenient}). The slogan is that `presentable categories are small enough so that certain set-theoretical problems disappear but are still large enough to include many important examples'.

Anyhow, all we need from the theory of combinatorial model categories is the validity of the next theorem so that we could also work axiomatically with the conclusion of this theorem. The statement about the projective model structures is a consequence of the lifting theorem of cofibrantly generated model structures along a left adjoint while the statement about the injective model structure was only proved more recently. Both results are for example established in \cite[Proposition A.2.8.2]{HTT}.

\begin{theorem}
Let $\mathcal{M}$ be a combinatorial model category and let $J$ be a small category. The category $\mathcal{M}^J$ can be endowed with the projective and with the injective model structure.
\end{theorem}

Recall that the \emph{projective} model structure is determined by the fact that the weak equivalences and the fibrations are defined levelwise. In the \emph{injective} model structure this is the case for the weak equivalences and the cofibrations. We will denote the functor categories $\mathcal{M}^J$ endowed with the corresponding model structures by $\mathcal{M}^J_{proj}$ resp. $\mathcal{M}^J_{inj}.$ In the special case where the combinatorial model category we start with is the category of simplicial sets endowed with the homotopy-theoretic Kan model structure, the projective model structure on a diagram category is the Bousfield-Kan structure of \cite{bousfieldkan} while the injective model structure is the Heller structure of \cite{heller}. One point of these model structures is that certain adjunctions are now Quillen adjunctions for trivial reasons.

\begin{lemma}
Let $\mathcal{M}$ be a combinatorial model category and let $u\colon J\nach K$ be a functor. Then we have the following Quillen adjunctions
$$(u_!,u^\ast)\colon \mathcal{M}^J_{proj}\nach \mathcal{M}^K_{proj}\qquad \mbox{and}\qquad (u^\ast,u_\ast)\colon \mathcal{M}^K_{inj}\nach \mathcal{M}^J_{inj}.$$
\end{lemma}

We now have almost everything at our disposal needed to establish the following result.

\begin{proposition}\label{prop_comb}
Let $\mathcal{M}$ be a combinatorial model category. Then the assignment
$$\DD_{\mathcal{M}}\colon \Cat^{\op}\nach\CAT\colon J\auf \Ho(\mathcal{M}^J)$$
defines a strong derivator.
\end{proposition}

\begin{proof}
The first axiom (Der1) is immediate. (Der2) holds in this case since the weak equivalences are precisely the morphisms which are inverted by the formation of homotopy categories and since the weak equivalences are defined levelwise. It is thus enough to consider the two axioms on homotopy Kan extensions. We treat only the case of homotopy right Kan extensions. The other case follows by duality. Axiom (Der3) on the existence of homotopy Kan extension functors follows easily from the last lemma since one only has to consider the associated derived adjunctions at the level of homotopy categories. So it remains to establish Kan's formula. For this purpose, let $u\colon J\nach K$ be a functor and let $k\in K$ be an object. Consider the following diagram, which commutes up to natural isomorphism by the usual base change morphism from classical category theory:
$$\xymatrix{
\mathcal{M}^{J_{k/}}\ar[d]_{\lim}\ar@{}[dr]|{\cong}& \mathcal{M}^{J}\ar[l]_{\pr^\ast}\ar[d]^{u_\ast}\\
\mathcal{M}& \mathcal{M}^K\ar[l]_{k^\ast}
}$$
By the last lemma, the functors $\lim$ and $u_\ast$ are right Quillen functors with respect to the injective model structures. If we can show that also the functors $k^\ast$ and $\pr^\ast$ are right Quillen functors with respect to the injective model structures, then we are done. In fact, in that case the two compositions of derived right Quillen functors are canonically isomorphic and this in turn shows that the base change morphism is an isomorphism. So let us show that $k^\ast$ is a right Quillen functor. By definition of the injective model structures, $k^\ast$ preserves weak equivalences. Hence it is enough to show that $k^\ast$ preserves fibrations. Using the adjunction $(k_!,k^\ast)$ it is enough to show that $k_!\colon \mathcal{M}\nach \mathcal{M}^K$ preserves acyclic cofibrations. But an easy calculation with left Kan extension shows that we have $k_!(X)_l\cong \coprod_{\hhom_K(k,l)} X$. From this description it is immediate that $k_!$ preserves acyclic cofibrations. Finally, we will show in Lemma \ref{lemma_quillengroth} that also $\pr^\ast$ is a right Quillen functor with respect to the injective model structure.

The strongness is left to the reader. It can be deduced by some `mapping cylinder arguments' using the projective model structure on~$\mm^{[1]}.$
\end{proof}

To conclude the proof of Proposition \ref{prop_comb} we have to show that the functor $\pr^\ast\colon \mathcal{M}^J\nach \mathcal{M}^{J_{k/}}$ is a right Quillen functor with respect to the injective model structures. It is again immediate that~$\pr^\ast$ preserves injective weak equivalences. Hence it suffices to show that $\pr^\ast$ preserves injective fibrations. We will prove such a result for arbitrary Grothendieck opfibrations with discrete fibers which applies, in particular, to our situation.

\begin{lemma}\label{lemma_quillengroth}
Let $u\colon J\nach K$ be a Grothendieck opfibration with discrete fibers and let $\mathcal{M}$ be a combinatorial model category. Then the functor $u^\ast\colon \mathcal{M}^K\nach \mathcal{M}^J$ preserves injective fibrations.
\end{lemma}
\begin{proof}
By adjointness, it is enough to show that the left adjoint $u_!\colon \mathcal{M}^J\nach\mathcal{M}^K$ preserves acyclic injective cofibrations. For this purpose, let $X\in \mathcal{M}^J$ and let $k\in K.$ Then we make the following calculation:
$$u_!(X)_k\;\cong \;\colim _{J_{/k}}X\circ \pr \;
\cong \; \colim _{J_k}X\circ \pr \circ \; c\; \cong \; \coprod _{j\in J_k} X_j
$$
The first isomorphism is again Kan's formula for Kan extensions. The second isomorphism is given by the cofinality of right adjoints (Proposition \ref{prop_cofinality}) applied to the canonical functor $c\colon J_k\nach J_{/k}$. Finally, the last isomorphism uses the fact that the Grothendieck opfibration has discrete fibers. From this explicit description of $u_!$ the claim follows immediately.
\end{proof}

The proof of the above theorem actually shows a bit more. Given a cofibrantly generated model category $\mathcal{M}$, the prederivator $\DD_{\mathcal{M}}$ is a what could be called \emph{cocomplete prederivator} (with the obvious meaning). But by far more is true. There is the following more general result which is due to Cisinski \cite{cisinski}.

\begin{theorem}\label{thm_cisinski}
Let $\mathcal{M}$ be a model category and let $J$ be a small category. Denote by $W_J$ the class of levelwise weak equivalences in $\mathcal{M}^J.$ Then the assignment
$$\DD_{\mathcal{M}}\colon \Cat^{\op} \nach \CAT \colon J\auf \mathcal{M}^J[W_J^{-1}]$$
defines a derivator.
\end{theorem}
\noindent
The basic idea is to reduce the situation of an arbitrary diagram category using certain cofinality arguments to the situation where the indexing categories are so-called Reedy categories (\cite{hovey}). The proof can be found in \cite{cisinski}. From the proof it will, in particular, follow that the above localizations make sense (i.e., that no change of universe is necessary!) although, in general, there is no model structure on $\mm^J$ with $W_J$ as weak equivalences. For more comments about the relationship between model categories and derivators see Remark \ref{rem_equivalence}.

\begin{remark}
A combination of Theorem \ref{thm_cisinski} and Example \ref{example_bicompletecat} thus shows that derivators form quite a general framework. First, they subsume bicomplete categories. Moreover, they provide an abstract description of the calculus of homotopy Kan extensions at the level of the various homotopy categories associated to a model categories. Thus, this framework allows us to treat categorical limits and colimits and the homotopical variants on an equal footing. 

This is similar to what happens in the related theories. In the theory of $\infty$-categories, the notion of limits and colimits also subsumes both variants. In the case of nerves of categories, the notion reduces to the classical notion of (co)limits, while when applied to coherent nerves of (locally fibrant) simplicial model categories it coincides with the notion of homotopy (co)limits (cf.\ \cite{HTT} or \cite{groth_infinity}). 

Similarly, every bicomplete category can be endowed with the \emph{discrete model structure} where the weak equivalences are the isomorphisms and all morphisms are (co)fibrations. With respect to this model structure the theory of homotopy (co)limits reduces to the theory of categorical (co)limits.
\end{remark}

\section{The $2$-category of derivators}\label{sec_2cat}

\subsection{Morphisms and natural transformations}

Let $\DD$ and $\DD '$ be prederivators. A \emph{morphism of prederivators} $F\colon\DD\nach\DD'$ is a pseudo-natural transformation between the $2$-functors $\DD$ and $\DD'$ (cf.\ to Definition 7.5.2 of \cite{borceux1}). Spelling out this definition such a morphism is a pair $(F_\bullet, \gamma^F_\bullet)$ consisting of a collection of functors 
$$F_J\colon \DD(J)\nach \DD'(J),\qquad J\in\Cat,$$ 
and a family of natural isomorphisms $\gamma^F_u\colon u^\ast\circ  F_K\nach F_J\circ u^\ast,\; u\colon J\nach K$ as indicated in
$$
\xymatrix{
\DD(K)\ar[r]^{F_K}\ar[d]_{u^\ast}\ar@{}[dr]|{\cong}&\DD'(K)\ar[d]^{u^\ast}\\
\DD(J)\ar[r]_{F_J}&\DD'(J).
}
$$
This datum is subject to the following coherence properties. Given a pair of composable functors $J\stackrel{u}{\nach}K\stackrel{v}{\nach}L$ and a natural transformation $\alpha\colon u_1\nach u_2\colon J\nach K,$ we then have the following relation resp.\ commutative diagrams:
$$\xymatrix{
\gamma_{\id_J}=\id_{F_J}&&u^\ast v^\ast F\ar[r]^{\gamma _v}\ar@/_1.0pc/[dr]_{\gamma_{vu}}&u^\ast F v^\ast\ar[d]^{\gamma _u}&&u_1^\ast F\ar[r]^{\alpha^\ast}\ar[d]_{\gamma_{u_1}}&u_2^\ast F\ar[d]^{\gamma_{u_2}}\\
& &&  F u^\ast v^\ast && F u_1^\ast\ar[r]_{\alpha ^\ast}& Fu_2^\ast
}
$$
Here, we suppressed some indices (as we will frequently do in the sequel) to avoid awkward notation. Given such a morphism $F\colon\DD\nach\DD'$ the functor $F_e\colon\DD(e)\nach\DD'(e)$ is called the \emph{underlying functor}.

As usual the notion of a pseudo-natural transformation can be relaxed or can be strengthened. In the more relaxed situation there would be two versions of such morphisms, the lax ones and the colax ones, but these do not play an important role in this paper (they will briefly show up in the next subsection and they will be of some importance in the context of morphisms of two variables in \cite{groth_monder}). Strictly speaking, also in our situation there are two notions depending on the direction of the natural transformations $\gamma$. Since one can always pass to the inverse natural transformations these notions are equivalent. In what follows, we will be a bit sloppy in notation in that we will not distinguish notationally between the natural isomorphisms $\gamma$ belonging to such a morphism and their inverses $\gamma^{-1}.$
In the case of a $2$-natural transformation, i.e., if all natural transformations $\gamma$ are given by identities, we speak of a \emph{strict morphism}. The class of strict morphisms is too narrow in that many examples will only be pseudo-natural transformation but strict morphisms are conceptually easier. This becomes manifest, for example, in the $2$-categorical Yoneda lemma as opposed to the more general bicategorical Yoneda lemma. 

Finally, let $F,\:G\colon\DD\nach\DD'$ be morphisms of prederivators. A \emph{natural transformation} $\tau\colon F\nach G$ is a modification of the pseudo-natural transformations (see \cite[Definition 7.5.3]{borceux1}). Thus, such a $\tau$ is given by a family of natural transformations $\tau_J\colon F_J\nach G_J$ which are compatible with the coherence isomorphisms belonging to the functors $F$ and $G$ in the sense that for every functor $u\colon J\nach K$ the following diagram commutes:
$$
\xymatrix{
u^\ast F\ar[r]^\tau\ar[d]_\gamma& u^\ast G\ar[d]^\gamma\\
Fu^\ast\ar[r]_\tau&Gu^\ast
}
$$ 
Given two parallel morphisms $F$ and $G$ of prederivators let us denote by $\nat(F,G)$ the natural transformations from $F$ to $G.$ Thus, with prederivators as objects, morphisms as $1$-cells, and natural transformations as $2$-cells we obtain the $2$-category $\PDer$ of prederivators. In fact, this is just a special case of a $2$-category given by $2$-functors, pseudo-natural transformations, and modifications. The full sub-$2$-category spanned by the derivators is denoted by $\Der.$ Given two (pre)derivators $\DD$ and $\DD'$ let us denote the category of morphisms by $\Hom(\DD,\DD')$ while we will write $\Hom^{\strict}(\DD,\DD')$ for the full subcategory spanned by the strict morphisms. Correspondingly, we have two sub-$2$-categories $\PDer^{\strict}\nach\PDer$ and $\Der^{\strict}\nach\Der.$

With Lemma \ref{lemma_BCpasting} and later applications (e.g.\ Proposition \ref{prop_closureDer!}) in mind let us only mention that three of the above defining coherence conditions can be interpreted as equalities of certain pasting diagrams. Examples of morphisms will be given after the following two comments.

Let us only quickly mention that the $2$-categories $\PDer$ and $\Der$ admit finite $2$-products and are, in fact, Cartesian closed $2$-categories in a certain precise sense. This observation plays a central role in the development of monoidal aspects of the theory of derivators (cf.\ to \cite{groth_monder}) but will not be needed in this paper. The $2$-product and related notions will be studied in more detail in that reference.

In every $2$-category we have the notion of adjoint morphisms and equivalences. It turns out that the resulting notion of an equivalence of derivators can be simplified. In fact, an equivalence is already completely determined by giving a single morphism of derivators which is levelwise an equivalence of categories. However, the case of an adjunction is slightly more subtle and will be taken up again in the next subsection (cf.\ Proposition \ref{prop_adjoint}). Let us conclude this subsection by giving some examples.

\begin{example}\label{ex_morphisms}
{\rm i)} Let $\mc$ and $\md$ be categories and let us consider the associated represented prederivators $y\mc$ and $y\md$ respectively. A functor $F\colon \mathcal{C}\nach\mathcal{D}$ induces a strict morphism of prederivators $yF\colon y\mc\nach y\md$ in the obvious way.  This assignment is faithful and the morphisms in the image are precisely the strict ones, i.e., the $2$-natural transformations. Thus, we have a fully faithful $2$-functor 
$$y\colon\CAT\nach\PDer^{\strict}$$
whose restriction to bicomplete categories factors over $\Der^{\strict}.$ This can be seen as a special case of the $2$-categorical Yoneda lemma and from now on we will frequently drop the Yoneda embedding~$y$ from notation.\\
{\rm ii)} Let $\DD$ be a prederivator and let $v\colon L\nach M$ be a functor. The functors $(v\times\id_K)^\ast$ assemble into a strict \emph{precomposition morphism} of prederivators $v^\ast\colon\DD^M\nach\DD^L$ and similarly for natural transformations. Thus, every prederivator $\DD$ induces a $2$-functor $\DD^{(-)}\colon\Cat^{\op}\nach\PDer^{\strict}.$ In fact, this is a partial $2$-functor of the $2$-functor:
$$(-)^{(-)}\colon\Cat^{\op}\times\PDer\nach\PDer\colon(M,\DD)\mapsto\DD^M$$ 
{\rm iii)} Given a prederivator $\DD$ and a small category $M$ let us denote by $\DD(-)^M$ the prederivator which sends $K$ to $\DD(K)^M.$ The partial underlying diagram functors then assemble into a strict \emph{partial underlying diagram morphism} of prederivators $\dia_{M,-}\colon\DD^M\nach\DD(-)^M.$ Axiom (Der1) implies that $\dia_{\emptyset,-}$ and $\dia_{e\sqcup e,-}$ are equivalences in the case of a derivator.\\
{\rm iv)} Let $\mm$ be a combinatorial model category. Since the weak equivalences in the diagram categories~$\mm^K$ are defined levelwise, all the associated localization functors assemble into a strict morphism of derivators:
$$\gamma\colon\mm=y\mm\nach\DD_{\mm}$$
\end{example}

\subsection{Homotopy (co)limit preserving morphisms}\label{subsection_cocont}

Let $F\colon\DD\nach\DD'$ be a morphism of derivators and let $u\colon J\nach K$ be a functor. Let us recall that the natural transformation $\gamma^F_u$ and its inverse are $2$-cells as indicated in the following diagrams:
$$
\xymatrix{
\DD'(J)\xtwocell[1,1]{}\omit&\DD(J)\ar[l]_{F}&&\DD'(J)&\DD(J)\ar[l]_{F}\\
\DD'(K)\ar[u]^{u^\ast}&\DD(K)\ar[l]^{F}\ar[u]_{u^\ast}&&\DD'(K)\ar[u]^{u^\ast}&\DD(K)\ar[l]^{F}\ar[u]_{u^\ast}\xtwocell[-1,-1]{}\omit
}
$$
By passing to the corresponding Beck-Chevalley transformed $2$-cells we obtain natural transformations ${\gamma^F_u}_!$ respectively ${\gamma^F_u}_\ast$ as in:
$$
\xymatrix{
\DD'(J)\ar[d]_{u_!}&\DD(J)\ar[l]_{F}\ar[d]^{u_!}&&\DD'(J)\ar[d]_{u_\ast}&\DD(J)\ar[l]_{F}\ar[d]^{u_\ast}\xtwocell[1,-1]{}\omit\\
\DD'(K)\xtwocell[-1,1]{}\omit&\DD(K)\ar[l]^{F}&&\DD'(K)&\DD(K)\ar[l]^{F}
}
$$

\begin{definition}
Let $F\colon\DD\nach\DD'$ be a morphism of derivators and let $u\colon J\nach K$ be a functor. The morphism $F$ \emph{preserves homotopy left} respectively \emph{homotopy right Kan extensions along} $u$ if the natural transformation
$${\gamma^F_u}_!\colon u_!F\nach Fu_!\qquad\mbox{respectively}\qquad{\gamma^F_u}_\ast\colon Fu_\ast\nach u_\ast F$$
is an isomorphism.
\end{definition}

Similarly, we speak of a morphism of derivators which preserves homotopy left/right Kan extensions if the above property holds for all functors $u$ and also of a morphism which preserves homotopy (co)limits (of a particular shape or also in general). To motivate the definition let us quickly consider the following example.

\begin{example}
Let $F\colon \mc\nach\md$ be a functor between bicomplete categories and let us consider the induced strict morphism of represented derivators. We describe the above canonical morphisms in the absolute case, i.e., let $J$ be a category and $u=p_J\colon J\nach e$ be the unique functor to the terminal category. The above canonical morphisms then take the form:
$$
\xymatrix{
\colim _J F\ar[d]_\eta\ar[r]^\beta&F \colim _J & F \lim _J\ar[r]^\beta\ar[d]_\eta& \lim _J F\\
\colim _J Fu^\ast \colim _J\ar@{=}[r]& \colim _J u^\ast F \colim _J\ar[u]_\epsilon & \lim _J u^\ast F \lim _J\ar@{=}[r]& \lim _J F u^\ast \lim _J\ar[u]_\epsilon
}
$$
In the left diagram, $\beta$ evaluated at $X\in\mc^J$ is the canonical map from the colimit of $F\circ X$ to the image of $\colim X$ under $F$. Thus, we recover the usual notion of a colimit preserving functor, i.e., the functor $F\colon \mathcal{C}\nach\mathcal{D}$ preserves colimits if and only if the induced morphism of derivators $F\colon\mathcal{C}\nach\mathcal{D}$ preserves homotopy colimits. Dual comments apply to the diagram on the right.
\end{example}

We will see later that if a morphism $F\colon\DD\nach\DD'$ preserves certain homotopy Kan extensions then this is also the case for all induced morphisms $F^M\colon\DD^M\nach\DD'^M$ (cf.\ Corollary \ref{cor_minimalmorphism}). This will be a consequence of the fact that homotopy Kan extensions in derivators of the form $\DD^M$ are calculated pointwise. But let us first establish the following expected fact.

\begin{proposition}\label{prop_cocont}
A morphism $F\colon\DD\nach\DD'$ of derivators preserves homotopy left Kan extensions if and only if it preserves homotopy colimits.
\end{proposition}
\begin{proof}
Let us assume that $F$ preserves homotopy colimits and let us consider a functor $u\colon J\nach K.$ We obtain the following pasting diagram in which the natural transformation on the right is the one we want to show to be an isomorphism:
$$
\xymatrix{
\DD'(J_{/k})\ar[d]_{p_!}&\DD'(J)\ar[l]_{pr^\ast}\ar[d]_{u_!}&\DD(J)\ar[l]_F\ar[d]^{u_!}\\
\DD'(e)\xtwocell[-1,1]{}\omit&\DD'(K)\ar[l]^{k^\ast}\xtwocell[-1,1]{}\omit&\DD(K)\ar[l]^{F}
}
$$
Using axiom (Der4) and the fact that isomorphisms are detected pointwise it suffices to show that the pasting is an isomorphism. The compatibility of the formation of BC transformed $2$-cells with respect to pasting implies that we have to show that the transformed $2$-cell associated to the following left diagram is an isomorphism:
$$
\xymatrix{
\DD'(J_{/k})\xtwocell[1,1]{}\omit&\DD'(J)\ar[l]_{pr^\ast}\xtwocell[1,1]{}\omit&\DD(J)\ar[l]_F&&
\DD'(J_{/k})\xtwocell[1,1]{}\omit&\DD(J_{/k})\ar[l]_F\xtwocell[1,1]{}\omit&\DD(J)\ar[l]_{pr^\ast}\\
\DD'(e)\ar[u]^{p^\ast}&\DD'(K)\ar[u]^{u^\ast}\ar[l]^{k^\ast}&\DD(K)\ar[l]^{F}\ar[u]_{u^\ast}&&
\DD'(e)\ar[u]^{p^\ast}&\DD(e)\ar[u]^{p^\ast}\ar[l]^{F}&\DD(K)\ar[l]^{k^\ast}\ar[u]_{u^\ast}
}
$$
But, using the isomorphisms $\gamma^F_{pr}$ and $\gamma^F_k,$ this is equivalent to showing that the BC transformed $2$-cell associated to the right diagram is an isomorphism. This in turn follows from our assumption that $F$ preserves homotopy colimits and (Der4).
\end{proof}

For convenience let us collect the following important closure properties of homotopy Kan extensions preserving morphisms.

\begin{proposition}\label{prop_closureDer!}
Let $\DD,\;\DD',$ and $\DD''$ be derivators, let $u\colon I\nach J$ and $v\colon J\nach K$ be functors.\\
{\rm i)} The identity morphism $\id_{\DD}\colon\DD\nach\DD$ preserves homotopy left Kan extensions.\\
{\rm ii)} If $F\colon\DD\nach\DD'$ and $G\colon\DD'\nach\DD''$ preserve homotopy left Kan extensions along $u$ then so does the composition $G\circ F\colon\DD\nach\DD''.$\\
{\rm iii)} If $F\colon\DD\nach\DD'$ preserves homotopy left Kan extensions along $u$ and $v$ then it preserves homotopy left Kan extensions along $v\circ u.$\\
{\rm iv)} If $\tau\colon F\nach G$ is a natural isomorphism of morphisms of derivators $\DD\nach\DD'$ then $F$ preserves homotopy left Kan extensions along $u$ if and only if $G$ does.
\end{proposition}
\begin{proof}
The first claim follows immediately from the triangular identities of adjunctions. The other three claims can be deduced from the compatibility of the formation of BC transformed $2$-cells with respect to horizontal and vertical pasting respectively. In the proof of the last two claims it is convenient to rewrite some of the coherence conditions imposed on morphisms and natural transformations of derivators as equalities between certain pasting diagrams.
\end{proof}

Given two derivators $\DD$ and $\DD'$ let us denote by
$$\Hom_!(\DD,\DD')\qquad\mbox{respectively}\qquad\Hom_\ast(\DD,\DD')$$
the full subcategories of $\Hom(\DD,\DD')$ spanned by the morphisms which respect homotopy colimits and homotopy limits respectively. By the above proposition, these are replete subcategories to which the composition law can be restricted. Correspondingly, we obtain the following $2$-categories:
$$\Der_!\qquad\mbox{and}\qquad\Der_\ast$$

\begin{proposition}\label{prop_kanpointwise}
Let $\DD$ be a derivator and let $v\colon L\nach M$ be a functor. The morphism of derivators $v^\ast\colon\DD^M\nach\DD^L$ preserves homotopy Kan extensions. In particular, this is the case for the evaluation morphisms $m^\ast\colon\DD^M\nach\DD.$
\end{proposition}
\begin{proof}
We will treat the case of homotopy left Kan extensions. By Proposition \ref{prop_cocont} it is enough to show that $v^\ast\colon\DD^M\nach\DD^L$ preserves homotopy colimits. Thus, we have to show that for an arbitrary small category~$J$ the following square is $\DD$-exact:
$$
\xymatrix{
L\times J\ar[r]^-{v\times\id}\ar[d]_{\pr}\xtwocell[1,1]{}\omit&M\times J\ar[d]^{\pr}\\
L\ar[r]_-v&M
}
$$
But this follows from an application of Proposition \ref{prop_BCfibration}.
\end{proof}

Thus, this proposition tells us, in particular, that homotopy Kan extensions in the derivator $\DD^M$ are calculated pointwise. For a functor $u\colon J\nach K$ and an object $X\in\DD^M(J)$ we have natural isomorphisms:
$$\HoLKan_u(X_m)\stackrel{\cong}{\nach}(\HoLKan_uX)_m\qquad\mbox{and}\qquad(\HoRKan_uX)_m\stackrel{\cong}{\nach}\HoRKan_u(X_m)$$
Similarly, in the absolute case, i.e., in the case of $u=p_J\colon J\nach e,$ we obtain canonical isomorphisms:
$$\Hocolim_J(X_m)\stackrel{\cong}{\nach} (\Hocolim_JX)_m \qquad\mbox{and}\qquad (\Holim_JX)_m\stackrel{\cong}{\nach}\Holim_J(X_m)$$
These isomorphisms are well-behaved in the sense that the following diagram commutes. We give the compatibility in the case of homotopy left Kan extensions:
$$
\xymatrix{
 u_!m^\ast u^\ast \ar[d]_=\ar[r]^{{\gamma_u}_!}_\cong& m^\ast  u_! u^\ast\ar[d]^\epsilon\\
u_!u^\ast m^\ast\ar[r]_-\epsilon& m^\ast 
}
$$
In fact, this is immediate using the compatibility of base change with pasting and the equality:
$$
\xymatrix{
J\ar[r]^-{m\times\id}\ar[d]_u&M\times J\ar[d]^{\id\times u}\ar[r]^-{\id\times u}&M\times K\ar[d]^{\id}\ar@{}[rrd]|{=}&& 
J\ar[d]_u\ar[r]^u&K\ar[d]^{\id}\ar[r]^-{m\times\id}&M\times K\ar[d]^{\id}\\
K\ar[r]_-{m\times\id}&M\times K\ar[r]_-{\id}&M\times K&& 
K\ar[r]_-{\id}&K\ar[r]_-{m\times\id}&M\times K
}
$$
\noindent
This compatibility implies, in particular, that, for $X\in \DD^M(K)$, the counit $\epsilon\colon u_!u^\ast(X)\nach X$ is an isomorphism in $\DD^M(K)$ if and only if the counit $\epsilon\colon u_!u^\ast(X_m)\nach X_m$ is an isomorphism in $\DD(K)$ for all objects $m\in M.$ For later reference, we collect the following convenient consequence for the case of a fully faithful functor $u\colon J \nach K.$

\begin{corollary}\label{cor_essimpointwise}
Let $\DD$ be a derivator, $M$ a category, and let $u\colon J\nach K$ be a fully faithful functor. An object $X\in \DD^M(K)$ lies in the essential image of $u_!\colon \DD^M(J)\nach \DD ^M(K)$ if and only if $X_m$ lies in the essential image of $u_!\colon \DD(J)\nach\DD(K)$ for all $m\in M.$
\end{corollary}

The fact that homotopy Kan extensions in the derivator $\DD^M$ are calculated pointwise (Proposition~\ref{prop_kanpointwise}) can also be used to establish the following convenient result.

\begin{corollary}\label{cor_minimalmorphism}
Let $F\colon\DD\nach\DD'$ be a morphism of derivators and let $u\colon J\nach K$ be a functor. Then $F$ preserves homotopy left Kan extensions along $u$ if and only if $F^M\colon\DD^M\nach\DD'^M$ preserves homotopy left Kan extensions along $u$ for all small categories $M.$
\end{corollary}
\begin{proof}
We have to show that with ${\gamma^F_u}_!$ also the natural transformation ${\gamma^{F_M}_u}_!={\gamma^F_{\id_M\times u}}_!$ is a natural isomorphism. Since isomorphisms can be detected pointwise and since $m^\ast$ preserves homotopy left Kan extensions (by Proposition \ref{prop_kanpointwise}) this is equivalent to the fact that the pasting in the left diagram is a natural isomorphism:
$$
\xymatrix{
\DD'(J)\ar[d]_{u_!}&\DD'^M(J)\ar[l]_{m^\ast}\ar[d]_{u_!}&\DD^M(J)\ar[l]_{F^M}\ar[d]^{u_!}&&
\DD'(J)\ar[d]_{u_!}&\DD(J)\ar[l]_F\ar[d]_{u_!}&\DD^M(J)\ar[l]_{m^\ast}\ar[d]^{u_!}\\
\DD'(K)\xtwocell[-1,1]{}\omit&\DD'^M(K)\ar[l]^{m^\ast}\xtwocell[-1,1]{}\omit&\DD^M(K)\ar[l]^{F^M}&&
\DD'(K)\xtwocell[-1,1]{}\omit&\DD(K)\ar[l]^F\xtwocell[-1,1]{}\omit&\DD^M(K)\ar[l]^{m^\ast}
}
$$
By the natural isomorphism $m^\ast\circ F^M\cong F\circ m^\ast$ (and strictly speaking Proposition \ref{prop_closureDer!}) this is equivalent to the fact that the pasting in the right diagram is a natural isomorphism. But this follows from our assumption that $F$ preserves homotopy left Kan extensions along $u$ and the fact that $m^\ast$ lies in $\Hom_!(\DD^M,\DD).$
\end{proof}

Since the values of a derivator always admit initial objects and finite coproducts (Proposition~\ref{prop_(co)limits}) let us agree on establishing the following terminology.

\begin{definition}
A morphism of derivators \emph{preserves initial objects respectively finite coproducts} if the underlying functor has the respective property.
\end{definition}
\noindent
By the last corollary it follows immediately that for such a morphism the corresponding statement is also true for the functors at all levels.

Let us now again take up the notion of adjunctions between derivators. We include this slightly lengthy discussion since this will motivate how to define an adjunction of two variables for derivators. And this notion in turn will play an essential role in the theory of monoidal and enriched derivators (see \cite{groth_monder, groth_enriched}). 

An adjunction between two derivators~$\DD$ and~$\DD'$ consists of two morphisms~$L\colon\DD\nach\DD'$ and $R\colon\DD'\nach\DD$ and two natural transformations $\eta\colon\id\nach RL$ and $\epsilon\colon LR\nach\id$ which satisfy the usual triangular identities. One might wonder if less data would already determine such an adjunction. As a first step there is the following result.

\begin{lemma}\label{lemma_adjoint}
Let $L\colon\DD\nach\DD'$ be a morphism of prederivators such that $L_K\colon\DD(K)\nach\DD'(K)$ has a right adjoint $R_K$ for each $K\in\Cat.$ Then, there is a unique way to extend the $\{R_K\}$ to a \emph{lax} morphism of prederivators $R\colon\DD'\nach\DD$ such that the following diagram commutes for all functors $u\colon J\nach K,$ $X\in\DD(K),$ and $Y\in\DD'(K):$
$$
\xymatrix{
\hhom_{\DD'(K)}(LX,Y)\ar[r]\ar[d]_{u^\ast}&\hhom_{\DD(K)}(X,RY)\ar[d]^{u^\ast}\\
\hhom_{\DD'(J)}(u^\ast LX,u^\ast Y)\ar[d]_{\gamma^L}&\hhom_{\DD(J)}(u^\ast X,u^\ast RY)\ar[d]^{\gamma^R}\\
\hhom_{\DD'(J)}(Lu^\ast X,u^\ast Y)\ar[r]&\hhom_{\DD(J)}(u^\ast X,Ru^\ast Y)
}
$$
\end{lemma}
\begin{proof}
If we choose $X=RY$ and if we trace the adjunction counit $\epsilon\colon LRY\nach Y$ around the diagram we obtain the uniqueness of the natural transformations $\gamma^R.$ In fact, $\gamma^R$ as indicated in the right diagram is the Beck-Chevalley transformed $2$-cell of the transformation on the left:
$$
\xymatrix{
\DD'(J)&\DD'(K)\ar[l]_{u^\ast}&&\DD'(J)\ar[d]_R&\DD'(K)\ar[l]_{u^\ast}\ar[d]^R\xtwocell[1,-1]{}\omit\\
\DD(J)\ar[u]^L&\DD(K)\ar[u]_L\ar[l]^{u^\ast}\xtwocell[-1,-1]{}\omit&&\DD(J)&\DD(K)\ar[l]^{u^\ast}
}
$$ 
Thus, in formulas, we have $\gamma^R_u=(\gamma^L_u)^{-1}_\ast.$ It remains to check that this defines a lax morphism of prederivators $R\colon\DD\nach\DD'.$ We omit the details but mention that the respective formulas are implied by the triangular identities for adjunctions and by the behavior of the formation of BC transformed $2$-cells with respect to horizontal and vertical pasting (Lemma \ref{lemma_BCpasting}).
\end{proof}
\noindent
In general, we cannot deduce that the lax morphism $R\colon\DD'\nach\DD$ is an actual morphism, i.e., a \emph{pseudo}-natural transformation. However, in the context of derivators this can be reformulated using Lemma \ref{lemma_BCtwoisos} which guarantees that the following natural transformations are conjugate:
$$
{\gamma^L_u}_!\colon u_!\circ L\nach L\circ u_!\qquad\mbox{and}\qquad{\gamma^R_u}\colon u^\ast\circ R\nach R\circ u^\ast
$$
Thus, since these are conjugate transformations, one of them is an isomorphism if and only if this is the case for the other. From this we obtain the following result.

\begin{proposition}\label{prop_adjoint}
Let $L\colon\DD\nach\DD'$ be a morphism of derivators which admits levelwise right adjoints and let~$R\colon\DD'\nach\DD$ be a lax morphism as in Lemma \ref{lemma_adjoint}. The morphism~$L$ is a left adjoint morphism of derivators if and only if~$L$ preserves homotopy left Kan extensions if and only if~$R$ is a morphism of derivators. In particular, a morphism of derivators is an equivalence if and only if it is levelwise an equivalence of categories.
\end{proposition}
\begin{proof}
The equivalence of the statements about a left adjoint morphism follows immediately from the above. If the morphism~$L$ happens to be a levelwise equivalence of derivators then the lax morphism~$R$ is an actual morphism of derivators. In fact, the natural transformations $\gamma^R_u$ are compositions of three isomorphisms in this case.
\end{proof}

For later reference, let us collect the following important class of homotopy (co)limit preserving morphisms of derivators.

\begin{corollary}\label{cor_adjointcocont}
A left adjoint morphism of derivators preserves homotopy left Kan extensions.
\end{corollary}

Proposition \ref{prop_kanpointwise} and Proposition \ref{prop_adjoint} together give us immediately the following two classes of examples of adjunctions.

\begin{example}\label{ex_adjunctions}
{\rm i)} Let $\DD$ be a derivator and let $v\colon L\nach M$ be a functor, then we have two adjunctions of derivators $(v_!,v^\ast)\colon\DD^L\rightharpoonup\DD^M$ and $(v^\ast,v_\ast)\colon\DD^M \rightharpoonup\DD^L.$\\
{\rm ii)} Let $(F,U)\colon\mathcal{M}\nach\mathcal{N}$ be a Quillen adjunction between combinatorial model categories. Then the formation of derived Quillen functors gives us two (in general non-strict) morphisms of derivators $\mathbb{L}F\colon\DD_\mathcal{M}\nach\DD_{\mathcal{N}}$ and $\mathbb{R}U\colon\DD_\mathcal{N}\nach\DD_{\mathcal{M}}.$ These are part of an adjunction of derivators $(\mathbb{L}F,\mathbb{R}U)\colon\DD_{\mm}\rightharpoonup\DD_{\nn}$. In particular,~$\mathbb{L}F$ preserves homotopy left Kan extensions and~$\mathbb{R}U$ preserves homotopy right Kan extensions (Corollary \ref{cor_adjointcocont}). If we start with a Quillen equivalence then we obtain a derived equivalence of derivators. This already makes more precise the statement that a Quillen equivalence is not only a Quillen pair inducing an equivalence of homotopy categories but that it also respects the entire `homotopy theory'.\\
\end{example}

\begin{remark}\label{rem_equivalence}
$\bullet$ It can be shown that the above assignment $\mathcal{M}\auf \DD_\mathcal{M}$ suitably restricted defines a bi-equivalence of theories. Loosely speaking this says that nice model categories and nice derivators do the same job. More precisely, Renaudin has shown such a result in \cite{renaudin} by establishing the following two steps: Let $\ModQ$ denote the $2$-category of combinatorial model categories with Quillen adjunctions $(F,U)\colon\mm\rightharpoonup\nn$ as morphisms and natural transformations of left adjoints as $2$-morphisms. Renaudin shows that there is a pseudo-localization $\ModQ[W^{-1}]$ of the combinatorial model categories at the class $W$ of Quillen equivalences. Moreover, let $\Der^{\PPr}$ denote the $2$-category of derivators of small presentation together with adjunctions as morphisms. A derivator is said to be of small presentation if it can be obtained as a `nice' localization of the derivator associated to simplicial presheaves. The assignment $\DD_{(-)}\colon \ModQ\nach \Der^{\PPr}\colon \mathcal{M}\auf \DD_\mathcal{M}$ factors then up to natural isomorphism over the pseudo-localization $\ModQ[W^{-1}]$ as indicated in:
$$\xymatrix{
\ModQ\ar[r]^{\DD_{(-)}}\ar[d]_\gamma\ar@{}[dr]|{\cong}& \Der^{\PPr}\\
\ModQ[W^{-1}]\ar@/_1.3pc/[ru]_-{\DD_{(-)}}&
}
$$
Renaudin showed that the induced $2$-functor $\DD_{(-)}\colon \ModQ[W^{-1}]\nach \Der^{\PPr}$ is a biequivalence, i.e., a $2$-functor which is biessentially surjective and fully faithful in the sense that it induces \emph{equivalences} of morphism categories (for biequivalences cf.\ e.g.\ to \cite{street_categoricalstructures,lack_2catcompanion} and to \cite{lack_model2categories, lack_modelbicategories} for their more conceptual role).\\
$\bullet$ We want to include a remark on different approaches to a theory of $(\infty,1)-$categories. There are by now many different ways to axiomatize such a theory. Among these are the model categories, the $\infty-$categories, and the derivators. These theories are interrelated by various constructions. For a simplicial model category, one can use the coherent nerve construction of Cordier \cite{cordier} to obtain an underlying $\infty-$category. Moreover, given a bicomplete $\infty-$category or a model category, by forming systematically homotopy categories one obtains an associated derivator. These three theories are in fact all `equivalent in a certain sense' if one is willing to restrict to nice subclasses. These comparison results rely heavily on homotopical generalizations of the following `two-step hierarchy'. In classical category theory there are the presheaf categories which can be considered as universal cocompletions. More precisely, the fact that every contravariant set-valued functor on a small category is canonically a colimit of representable ones can be used to prove such a result. Nice localizations of these presheaf categories (the so-called accessible, reflective localizations) give us precisely the presentable categories (Representation Theorem \cite[Theorem 1.46]{adamekrosicky}). These two main steps, namely to establish the universal property of presheaf categories and to characterize presentable categories as nice localizations of presheaf categories, can be redone for all the different theories. To achieve this one has to replace presheaf categories by \emph{simplicial presheaf categories} which is fine with the basic philosophy of higher category theory. Moreover, the classical localization theory is replaced by a suitable Bousfield localization theory \cite{Bousfield, hirschhorn}. For model categories, this was done by Dugger in \cite{dugger_universal,dugger_combinatorial}, while the corresponding results for $\infty-$categories can be found in Lurie's \cite{HTT}. The characterization of presentable $\infty-$categories as being precisely the accessible, reflective localizations of simplicial presheaf categories is therein credited to \cite{Simpson}. For derivators, the free generation property of the derivator associated to simplicial presheaves can be found in \cite{cisinski_derivedkan}. Note however that the basic model used in the background is not the category of simplicial sets but the category of small categories. It can be shown that this way also all `homotopy types are modeled'. Finally, until now, the Representation Theorem for derivators of small presentation is only turned into a definition in \cite{renaudin}. The author plans to come back to this point in a later project. Having established these similar theories at all different levels one can then establish the comparison results if one restricts to the subclasses of (simplicial) combinatorial model categories, presentable $\infty-$categories, and derivators of small presentation.\\
\end{remark}

\section{Pointed derivators}\label{sec_pointed}

\subsection{Definition and basic examples}\label{subsec_pointed}

Since we are mainly interested in stable derivators, we turn immediately to the next richer structure, namely the pointed derivators. There are at least two ways to axiomatize a notion of a pointed derivator. From these two notions, we turn the `weaker one' into a definition. The `stronger one' will be referred to as a strongly pointed derivator, but we will show that these two notions actually coincide. 

\begin{definition}
A derivator $\DD$ is \emph{pointed} if the underlying category $\DD(e)$ of $\DD$ is pointed, i.e.,  admits a zero object $0\in \DD(e)$.
\end{definition}

Note that the pointedness is again only a property and not an additional structure. For a prederivator one would impose a slightly stronger condition: a prederivator is pointed if and only if all of its values and all restriction of diagram functors are pointed. In the case of a derivator these stronger properties follow immediately from the definition. 

\begin{proposition}\label{prop_productpointed}
Let $\DD$ be a pointed derivator and let $u\colon J\nach K$ a functor in $\Cat.$ Then $\DD(K)$ is also pointed and the functors $u_!,\;u^\ast,\;u_\ast$ are pointed. In particular, if a derivator $\DD$ is pointed then this is also the case for $\DD^M$ for every $M\in\Cat.$
\end{proposition}
\begin{proof}
The map from the initial to the final object in $\DD(K)$ is an isomorphism since this is pointwise the case. Moreover, each of the functors $u_!,\;u^\ast,\;u_\ast$ has an adjoint on at least one side and will hence preserve the zero objects.
\end{proof}

\begin{example}
{\rm i)} Let $\mc$ be a category. Then the represented prederivator $\mc$ is pointed if and only if the category $\mc$ is pointed.\\
{\rm ii)} The derivator $\DD_{\mathcal{M}}$ associated to a pointed combinatorial model category $\mathcal{M}$ is pointed.\\ 
{\rm iii)} A derivator $\DD$ is pointed if and only if its dual $\DD^{\op}$ is pointed.
\end{example}

We now want to give the stronger axiom as used by Maltsiniotis in \cite{maltsiniotis2}. 

\begin{definition}
A derivator $\DD$ is \emph{strongly pointed} if it has the following two properties:\\
{\rm i)} For every sieve $j\colon J\nach K,$ the homotopy right Kan extension functor $j_\ast$ has a right adjoint $j^!$:
$$(j_\ast,j^!)\colon \DD(J)\rightharpoonup \DD(K)$$
{\rm ii)} For every cosieve $i\colon J\nach K,$ the homotopy left Kan extension functor $i_!$ has a left adjoint $i^?$:
$$(i^?,i_!)\colon \DD(K)\rightharpoonup \DD(J)$$
\end{definition}

It is an immediate corollary of the definition that a strongly pointed derivator is pointed. In fact, one of the two additional properties is enough to ensure this.

\begin{corollary}
If $\DD$ is a strongly pointed derivator, then $\DD$ is pointed.
\end{corollary}
\begin{proof}
It is enough to consider the cosieve $\emptyset_e\colon \emptyset\nach e$. For an initial object ${\emptyset_e}_!(0)\in \DD(e)$ and an arbitrary $X\in \DD(e),$ we then deduce 
$\hhom_{\DD(e)}(X,{\emptyset_e}_!(0))\cong \hhom_{\DD(\emptyset)}(\emptyset_e^?X,0)=\ast,$ so that ${\emptyset_e}_!(0)$ is also terminal.
\end{proof}

The aim of this subsection is to prove that also the converse holds (cf.\ Corollary \ref{cor_pointed}). A further proof of that converse will be given in the stable situation, i.e., for stable derivators. That second proof is quite an indirect one. It relies on the fact that recollements of triangulated categories are overdetermined and will be given in Subsection \ref{subsection_recollement}. As a preparation, for the direct proof, we mention the following immediate consequence of Proposition \ref{prop_extzero_unpointed}. It states that homotopy left Kan extension along cosieves and homotopy right Kan extension along sieves are given by `extension by zero functors' and will be of constant use in the remainder of this paper.

\begin{proposition}\label{prop_extzero}
Let $\DD$ be a pointed derivator.\\
{\rm i)} Let $u\colon J\nach K$ be a cosieve, then the homotopy left Kan extension $u_!$ is fully faithful and $X\in \DD(K)$ lies in the essential image of $u_!$ if and only if $X_k\cong 0$ for all $k\in K-u(J).$\\
{\rm ii)} Let $u\colon J\nach K$ be a sieve, then the homotopy right Kan extension $u_\ast$ is fully faithful and $X\in \DD(K)$ lies in the essential image of $u_\ast$ if and only if $X_k\cong 0$ for all $k\in K-u(J).$
\end{proposition}

Following Heller \cite{heller_stable}, we introduce the following notation. Let $\DD$ be a pointed derivator and let $u\colon J\nach K$ be the inclusion of a full subcategory. Then we denote by $\DD(K,J)\subseteq \DD(K)$ the full, replete subcategory spanned by the objects $X$ which vanish on $J$, i.e., such that $u^\ast(X)=0.$ If $u$ is now a cosieve respectively a sieve, the above proposition guarantees that we have the following equivalences of categories:
$$(u_!,u^\ast)\colon \DD(J)\stackrel{\simeq}{\nach}\DD(K,K-J)\qquad\mbox{respectively}\qquad (u^\ast,u_\ast)\colon \DD(K,K-J)\stackrel{\simeq}{\nach}\DD(J)$$
This proposition, although easily proved, will be of central importance in all what follows. It will, in particular, be of constant use in the study of the important cone, fiber, suspension, and loop functors and also in the proof that the values of a stable derivator can be canonically endowed with the structure of a triangulated category. However, we first have to establish some properties of coCartesian and Cartesian squares and this will be done in the next Subsection \ref{subsection_squares}. 

To conclude this subsection we will now give the proof that pointed derivators are actually strongly pointed. The constructions involved in the proof are motivated by the paper of Rezk~\cite{rezk_naturalmodel} in which he gives a nice construction of the `natural model structure' on $\Cat.$ This model structure is due to Joyal and Tierney \cite{joyaltierney_stacks} and the adjective `natural' refers to the fact that the weak equivalences in that model structure are precisely the equivalences in the $2$-category $\Cat.$ We use a minor modification of the mapping cylinder categories used in \cite{rezk_naturalmodel}. Instead of forming the product with the groupoid generated by $[1]$ we use the category $[1]$ itself. This leads to two `differently oriented versions' of the mapping cylinder and both of them will be needed. 

\begin{lemma}
{\rm i)} Let $u\colon J\nach K$ be a cosieve. Then the subcategory $\DD(K,J)\subseteq \DD(K)$ is coreflective, i.e., the inclusion functor $\iota$ admits a right adjoint.\\
{\rm ii)} Let $u\colon J\nach K$ be a sieve. Then the subcategory $\DD(K,J)\subseteq \DD(K)$ is reflective, i.e., the inclusion functor $\iota$ admits a left adjoint. 
\end{lemma}
\begin{proof}
We will give the details for the proof of {\rm ii)} and mention the necessary modifications for~{\rm i)}. So, let $u\colon J\nach K$ be a sieve and let us construct the mapping cylinder category $\cyl(u).$ By definition,~$\cyl(u)$ is the full subcategory of $K\times [1]$ spanned by the objects $(u(j),1)$ and $(k,0).$ Thus, it is defined by the following pushout diagram where $i_0$ is the inclusion at $0$:
$$
\xymatrix{
J\ar[r]^u\ar[d]_{i_0}&K\ar[d]\\
J\times[1]\ar[r]&\cyl(u)
}
$$
There are the natural functors $i\colon J\nach \cyl(u)\colon j\auf (u(j),1)$ and $s\colon K\nach \cyl(u)\colon k\auf (k,0)$. Moreover, $\id\colon K\nach K$ and $J\times [1]\stackrel{\pr}{\nach} J\stackrel{u}{\nach} K$ induce a unique functor $q\colon \cyl(u)\nach K.$ These functors satisfy the relations $q\circ i=u,\;q\circ s=\id_K.$ 

Consider now an object $X\in\DD(\cyl(u),i(J))$ and let us calculate the value of $q_!(X)$ at some $u(j)\in K.$ For this purpose, we show that the following pasting is homotopy exact:
$$
\xymatrix{
e\ar[r]^-{i(j),1}\ar[d]_{\id}\xtwocell[1,1]{}\omit&i(J)_{/i(j)}\times{[1]}\ar[r]^\cong\ar[d]_p\xtwocell[1,1]{}\omit & \cyl(u)_{/u(j)}\ar[r]^{\pr}\ar[d]_p\xtwocell[1,1]{}\omit &\cyl(u)\ar[d]^q\\
e\ar[r]&e\ar[r]&e\ar[r]_{u(j)}&K
}
$$
Since $u$ is a sieve we have an isomorphism as depicted in the diagram. Moreover, $i(J)_{/i(j)}\times{[1]}$ has a terminal element so that the left two squares are homotopy exact by Proposition \ref{prop_cofinality}. Thus, we can conclude by (Der4) that the above pasting is homotopy exact and obtain $q_!(X)_{u(j)}\cong X_{i(j)}=0.$ The adjunction $(q_!,q^\ast)$ restricts to an adjunction $(q_!,q^\ast)\colon \DD(\cyl(u),i(J))\rightharpoonup \DD(K,u(J)).$ 

Moreover, since we defined the mapping cylinder forming the product with $[1]$ as opposed to the groupoid generated by it, $s\colon K\nach \cyl(u)$ is a sieve. Hence, by Proposition \ref{prop_extzero}, we have an induced equivalence $(s_{\ast},s^\ast)\colon\DD(K)\stackrel{\simeq}{\nach}\DD(\cyl(u),\cyl(u)-s(K))=\DD(\cyl(u),i(J)).$ 

Putting these two adjunctions together we obtain the adjunction 
$$(q_!\circ s_{\ast},s^\ast \circ q^\ast)\colon \DD(K)\rightharpoonup \DD(\cyl(u),i(J))\rightharpoonup \DD(K,u(J)).$$
The relation $q\circ s=\id$ implies that the right adjoint of this adjunction is the inclusion $\iota$ as intended and the reflection is given by $r=q_!\circ s_\ast.$ 

The proof of {\rm i)} is similar. Instead of using $\cyl(u)$ one uses this time the mapping cylinder category~$\cyl '(u)$ which is obtained by a similar pushout but using the inclusion $i_1$ instead of $i_0.$ Let us  denote the corresponding functors again by $i,\; q,$ and $s.$ Using a similar calculation of $q_\ast$ and the fact that~$s$ is now a cosieve, we can construct a coreflection $c.$
\end{proof}

\begin{corollary}\label{cor_pointed}
Let $\DD$ be a pointed derivator, then $\DD$ is also strongly pointed.
\end{corollary}
\begin{proof}
Given a sieve $u\colon J\nach K$ we have to show that $u_\ast$ has a right adjoint. The inclusion $v\colon K-u(J)\nach K$ of the complement is a cosieve. The above lemma applied to $v$ thus gives us a coreflection $c\colon\DD(K)\rightharpoonup \DD(K,K-u(J)).$ Putting this together with the equivalence induced by $u_\ast$ (guaranteed by Proposition \ref{prop_extzero}) we obtain the desired adjunction:
$$
(u_\ast,u^!)\colon
\xymatrix{
\DD(J)\ar@<0.5ex>[r]^-{u_\ast}&\DD(K,K-u(J))\ar@<0.5ex>[l]^-{u^\ast}\ar@<0.5ex>[r]^-\iota&\DD(K)\ar@<0.5ex>[l]^-c
}
$$
\noindent
The proof in the case of a cosieve is, of course, the dual one.
\end{proof}

The proofs of the last two results were constructive. So, for later reference, let us give precise formulas for these additional adjoint functors. Let $\DD$ be a pointed derivator and let $u\colon J\nach K$ be a cosieve. Let us denote by $v\colon J'=K-u(J)\nach K$ the sieve given by the complement. The adjunctions $(u^?,u_!)\colon\DD(K)\rightharpoonup\DD(J)$ and $(v_\ast,v^!)\colon\DD(J)\rightharpoonup\DD(K)$ are given by the following composite adjunctions respectively:
$$
u^?\colon
\xymatrix{
\DD(K)\ar@<0.6ex>[r]^-{s_\ast}&\DD(\cyl(v),i(J'))\ar@<0.6ex>[r]^-{q_!}\ar@<0.6ex>[l]^-{s^\ast}
&\DD(K,v(J'))\ar@<0.6ex>[r]^-{u^\ast}\ar@<0.6ex>[l]^-{q^\ast}&\DD(J)\ar@<0.6ex>[l]^-{u_!}
}
\colon u_!
$$
$$
v_\ast\colon
\xymatrix{
\DD(J')\ar@<0.6ex>[r]^-{v_\ast}&\DD(K,u(J))\ar@<0.6ex>[r]^-{q'^\ast}\ar@<0.6ex>[l]^-{v^\ast}
&\DD(\cyl'(u),i'(J))\ar@<0.6ex>[r]^-{s'^\ast}\ar@<0.6ex>[l]^-{q'_\ast}&\DD(K)\ar@<0.6ex>[l]^-{s'_!}
}
\colon v^!
$$
Here, $\cyl(v)$ is the mapping cylinder obtained from identifying the bottom $J'\times \{0\}$ of $J'\times [1]$ with the image of $v$, $i$ is the inclusion in the cylinder, $q$ is the projection and $s$ is the canonical section of~$q.$ The notation in the second decomposition is similar where the roles of $0$ and $1$ are interchanged.

\subsection{coCartesian and Cartesian squares}\label{subsection_squares}

In this subsection, we introduce coCartesian and Cartesian squares in a derivator and establish some facts about them which will be needed in Section \ref{sec_stable}. Some of these properties are well-known from classical category theory and will be reproved here for the context of derivators. The main results are the behavior of (co)Cartesian squares under cancellation and composition (Proposition \ref{prop_cancel}) and a `detection result' (due to Franke \cite{franke}) for (co)Cartesian squares (Proposition \ref{prop_detect}).

We denote the category $[1]\times [1]$ by $\square$, i.e., $\square$ is the following poset considered as a category where we draw the first coordinate horizontally:
$$\xymatrix{
(0,0)\ar[r]\ar[d] & (1,0)\ar[d] \\
(0,1)\ar[r] & (1,1)
}$$
\noindent
For the treatment of Cartesian and coCartesian squares, it is important to consider the following two inclusions of subcategories $\ipush\colon \push\nach\square$ resp.\ $\ipull\colon\pull\nach\square$ which are given by the subposets: 
$$\xymatrix{
(0,0)\ar[r]\ar[d]&(1,0)\ar@{}[rrrrd]|{\mbox{respectively}}&&&&&(1,0)\ar[d]\\
(0,1)& & & & & (0,1)\ar[r] & (1,1)
}
$$

\begin{definition}
Let $\DD$ be a derivator and let $X\in\DD(\square).$\\
i) The square $X$ is \emph{coCartesian} if it lies in the essential image of $\ipush_!\colon\DD(\push)\nach\DD(\square).$\\
ii) The square $X$ is \emph{Cartesian} if it lies in the essential image of $\ipull_\ast\colon\DD(\pull)\nach\DD(\square).$
\end{definition}
\noindent
It follows immediately from the fully faithfulness of homotopy Kan extensions along fully faithful functors (Proposition \ref{prop_honestext}) and Lemma \ref{lemma_essim} that such an $X\in \DD(\square)$ is coCartesian if and only if the canonical morphism $\epsilon _{(1,1)}\colon \ipush_! \ipush^\ast (X)_{(1,1)}\nach X_{(1,1)}$ is an isomorphism. Dually, the square $X$ is Cartesian if and only if the canonical morphism $\eta _{(0,0)}\colon X_{(0,0)}\nach \ipull_\ast \ipull^\ast (X)_{(0,0)}$ is an isomorphism.

Our first aim in this section is to establish a `detection result' for (co)Cartesian squares in larger diagrams  which will be used frequently later on. So, let us quickly give the notion of a square.

\begin{definition}
Let $J$ be a category. A \emph{square} in $J$ is a functor $i\colon \square \nach J$ which is injective on objects.
\end{definition}

Here is now the intended `detection result' for (co)Cartesian squares.

\begin{proposition}\label{prop_detect}
Let $i\colon \square \nach J$ be a square in $J$ and let $f\colon K\nach J$ be a functor.\\
{\rm i)} Assume that the induced functor $\push \stackrel {\tilde{i}}{\nach}{\big(J-i(1,1)\big)}_{/i(1,1)}$ has a left adjoint and that $i(1,1)$ does not lie in the image of $f$. Then for all $X=f_!(Y)\in \DD(J),\;Y\in\DD(K),$ the induced square $i^\ast (X)$ is coCartesian.\\
{\rm ii)} Assume that the induced  functor $\pull \stackrel {\tilde{i}}{\nach}{\big(J-i(0,0)\big)}_{i(0,0)/}$ has a right adjoint and that $i(0,0)$ does not lie in the image of $f$. Then for all $X=f_\ast(Y)\in \DD(J),\;Y\in\DD(K),$ the induced square $i^\ast (X)$ is Cartesian.
\end{proposition}
\begin{proof}
We give a proof of i). By assumption on $f$, $f$ factors as $K\stackrel{\bar{f}}{\nach}J-i(1,1)\stackrel{j}{\nach}J$ so that our setup can be summarized by:
$$
\xymatrix{
& \big(J-i(1,1)\big)_{/i(1,1)}\ar[d]^\pr&\\
\push \ar[ur]^{R=\tilde{i}}\ar[r]\ar[d]_{\ipush}&J-i(1,1)\ar[d]^j &K\ar[l]_{\bar{f}}\ar[dl]^f\\
\square\ar[r]_i&J &
}
$$
We want to show that the adjunctions counit $\epsilon\colon \ipush_!\ipush^\ast\nach\id$ is an isomorphism when applied to $i^\ast f_!(Y),\;Y\in\DD(K).$ But by Lemma \ref{lemma_essim} and Lemma \ref{lemma_BCpasting} this is equivalent to showing that the base change morphism associated to the pasting on the left-hand-side is an isomorphism when evaluated at $f_!(Y):$
$$
\xymatrix{
\push\cong\push_{/(1,1)}\ar[r]\ar[d]_p\xtwocell[1,1]{}\omit& \push\ar[r]^{\ipush}\ar[d]_{\ipush}\xtwocell[1,1]{}\omit& \square\ar[r]^i\ar[d]_{\id}\xtwocell[1,1]{}\omit& J\ar[d]^{\id}&
\push\ar[r]^-R\ar[d]_p\xtwocell[1,1]{}\omit& J-i(1,1)_{/i(1,1)}\ar[r]^{\pr}\ar[d]_p\xtwocell[1,1]{}\omit& J-i(1,1)\ar[r]^j\ar[d]_j\xtwocell[1,1]{}\omit& J\ar[d]^{\id}\\
e\ar[r]_{(1,1)}&\square\ar[r]_{\id}&\square\ar[r]_i&J&
e\ar[r]_{\id}&e\ar[r]_{i(1,1)}&J\ar[r]_{\id}&J
}
$$
Using Lemma \ref{lemma_BCpasting} again, this is equivalent to showing that the base change morphism associated to the pasting on the right gives an isomorphism when evaluated at $f_!(Y).$ But this is the case by~(Der4) and Proposition \ref{prop_cofinality}, since $f_!(Y)\cong j_!\bar{f}_!(Y)$ lies in the essential image of $j_!.$
\end{proof}

Typical applications of this proposition will be given when the categories under consideration are posets. Let $J$ and $K$ be posets considered as categories. Recall that a functor $u\colon J\nach K$ is the same as an order-preserving map. Moreover, an adjunction $(u,v)\colon J\rightharpoonup K$ is equivalently given by two order-preserving maps $u\colon J\nach K$ and $v\colon K\nach J$ such that $j\leq vu(j),\;j\in J,$ and $uv(k)\leq k,\; k\in K.$ In fact, in this case the triangular identities are automatically satisfied.

For $n \geq 0,$ we denote by $[n]$ the ordinal number $0<\ldots<n$ considered as a category. Moreover, let us denote the standard cosimplicial face resp.\ degeneracy maps by $d^i\colon [n-1]\nach [n],\;0\leq i \leq n,$ resp.\ $s^j\colon [n+1]\nach [n],\; 0\leq j \leq n.$ Here, $d^i$ is the unique monotone injection omitting $i$ while $s^j$ is the unique monotone surjection hitting $j$ twice. The images of these cosimplicial structure maps under a contravariant functor will, as usual, be written as $d_i$ resp.\ $s_j$.

\begin{lemma}\label{lemma_someadj}
For every $0\leq i\leq n-1$ we have an adjunction $(s^i,d^i)\colon[n]\rightharpoonup [n-1].$ In particular, we thus obtain the adjunctions $(s^0,d^0)\colon [2]\times [1]\rightharpoonup [1]\times [1]$ and $(s^1,d^1)\colon [2]\times [1]\rightharpoonup [1]\times [1].$
\end{lemma}

In the next proposition, we will consider squares $X\in\DD(\square)$ in a derivator and some of its associated sub-diagrams. To establish some short hand notation, let us denote by $d_v^i$ the face maps $\id \times d^i\colon [1]\nach [1]\times [1]=\square$ in the 'vertical direction' giving rise to `horizontal faces' and similarly in the other case. If we apply a contravariant functor to these morphisms we will interchange the indices and thus write $d^h_i$ and $d^v_i$ respectively.

\begin{proposition}\label{prop_cocartiso}
 Let $\DD$ be a derivator.\\
{\rm i)} An object of $\DD([1])$ is an isomorphism if and only if it lies in the essential image of the homotopy left Kan extension functor $0_!\colon\DD(e)\nach\DD([1]).$\\
{\rm ii)} Let $X\in \DD(\square)$ be a square such that $d_1^v (X)$ is an isomorphism, i.e., we have $X_{0,0}\stackrel{\cong}{\nach}X_{1,0}.$ The square $X$ is coCartesian if and only if also $d_0^v (X)$ is an isomorphism.
\end{proposition}
\begin{proof}
{\rm i)} This is a special case of (the dual of) Lemma \ref{lemma_terminal}.\\
{\rm ii)} By {\rm i)} our assumption on $X$ is equivalent to the adjunction counit $0_!0^\ast d_1^v (X)\nach d_1^v (X)$ being an isomorphism. Using Lemma \ref{lemma_essim} and (Der4), we can reformulate this by saying that the base change morphism associated to the following pasting is an isomorphism when evaluated on $X$:
$$
\xymatrix{
e\cong e_{/1}\ar[r]^{\pr}\ar[d]_p\xtwocell[1,1]{}\omit& e\ar[r]^0\ar[d]_0\xtwocell[1,1]{}\omit& [1]\ar[r]^{d^1_v}\ar[d]_{\id}\xtwocell[1,1]{}\omit& \square\ar[d]^{\id}\ar@{}[rrd]|{=}&&
e\ar[r]^{(0,0)}\ar[d]_{\id}\xtwocell[1,1]{}\omit&\square\ar[d]^{\id}\\
e\ar[r]_1&[1]\ar[r]_{\id}&\square\ar[r]_{d^1_v}&\square&&
e\ar[r]_{(1,0)}&\square
}
$$
We want to reformulate this in a way which is more convenient for this proof. For this purpose let us consider the following factorization of the \emph{horizontal} face map:
$$d_h^1=\ipush\circ j\colon[1]\stackrel{j}{\nach}\push\stackrel{\ipush}{\nach}\square$$
Now, our assumption that $d_1^v (X)$ is an isomorphism is equivalent to the counit $j_!j^\ast\ipush^\ast X\nach\ipush^\ast X$ being an isomorphism. In fact, using Lemma \ref{lemma_essim} and (Der4), the claim about the counit can be equivalently restated by saying that the base change of the following pasting is an isomorphism when evaluated at~$X:$ 
$$
\xymatrix{
e\cong [1]_{/(0,1)}\ar[r]^{\pr}\ar[d]_p\xtwocell[1,1]{}\omit& [1]\ar[r]^j\ar[d]_j\xtwocell[1,1]{}\omit& \push\ar[r]^{\ipush}\ar[d]_{\id}\xtwocell[1,1]{}\omit& \square\ar[d]^{\id}\ar@{}[rrd]|{=}&&
e\ar[r]^{(0,0)}\ar[d]_{\id}\xtwocell[1,1]{}\omit&\square\ar[d]^{\id}\\
e\ar[r]_{(0,1)}&\push\ar[r]_{\id}&\push\ar[r]_{\ipush}&\square&&
e\ar[r]_{(1,0)}&\square
}
$$
Thus, the claim follows from our previous reasoning. This in turn can be used to show that under our assumption the square $X$ is coCartesian if and only if the base change associated to
$$
\xymatrix{
[1]\ar[r]^j\ar[d]_j\xtwocell[1,1]{}\omit& \push\ar[r]^{\ipush}\ar[d]_{\id}\xtwocell[1,1]{}\omit& \square\ar[d]^{\id}\ar@{}[rrd]|{=}&&
[1]\ar[r]^{d^1_h}\ar[d]_{d^1_h}\xtwocell[1,1]{}\omit&\square\ar[d]^{\id}\\
\push\ar[r]_{\id}\ar[d]_{\ipush}\xtwocell[1,1]{}\omit& \push\ar[r]_{\ipush}\ar[d]_{\ipush}\xtwocell[1,1]{}\omit& \square\ar[d]^{\id}&&
\square\ar[r]_{\id}&\square\\
\square\ar[r]_{\id}&\square\ar[r]_{\id}&\square &&\\
}
$$
is an isomorphism at $X$. By Lemma \ref{lemma_essim} this is the case if and only if it is the case at $(1,1)$ which in turn is equivalent (by similar arguments as in the beginning of this proof) to the fact that $d_0^v(X)$ is an isomorphism.
\end{proof}

We now discuss the composition and cancellation property of (co)Cartesian squares. Recall from classical category theory that for a diagram in a category of the shape
$$\xymatrix{X_{0,0}\ar[r]\ar[d] & X_{1,0}\ar[d]\ar[r] & X_{2,0}\ar[d]\\
X_{0,1}\ar[r]& X_{1,1}\ar[r]&X_{2,1}}
$$
the following holds: if the square on the left is a pushout, then the square on the right is a pushout if and only if the composite square is. The corresponding result in the theory of derivators is the content of the next proposition. The methods are similar to the ones used in the proof of Proposition~\ref{prop_cocartiso} so we will be a bit more sketchy. Moreover, since we only use horizontal face maps this time we again drop the additional index.

\begin{proposition}\label{prop_cancel}
Let $\DD$ be a derivator and let $X \in \DD([2]\times [1]).$\\
{\rm i)} If $d_2(X)\in \DD(\square)$ is coCartesian, then $d_0(X)$ is coCartesian if and only if $d_1(X)$ is coCartesian.\\
{\rm ii)} If $d_0(X)\in \DD(\square)$ is Cartesian, then $d_2(X)$ is Cartesian if and only if $d_1(X)$ is Cartesian.
\end{proposition}
\begin{proof}
We give a proof of i). For this purpose, let $J_0$ resp.\ $J_1$ be the posets
$$\xymatrix{
(0,0)\ar[r]\ar[d]&(1,0)\ar[r]&(2,0)&\mbox{resp.}&(0,0)\ar[r]\ar[d]&(1,0)\ar[r]\ar[d]&(2,0)\\
(0,1)& &  && (0,1)\ar[r] & (1,1) &
}
$$
and denote the fully faithful inclusion functors by $i\colon J_0 \stackrel{i_0}{\nach} J_1 \stackrel{i_1}{\nach} J_2=[2]\times [1].$ Our assumption on $X$ and Lemma \ref{lemma_essim} guarantee that the base change morphism associated to the following square evaluated at $X$ is an isomorphism if and only if this is the case at $(2,1):$
$$
\xymatrix{
J_0\ar[r]^i\ar[d]_i\xtwocell[1,1]{}\omit&J_2\ar[d]^{\id}\\
J_2\ar[r]_{\id}&J_2
}
$$
We want to reformulate this property in two ways. First, using (Der4) and Proposition \ref{prop_cofinality} applied to $d^1\colon\push\nach J_0$ this can be seen to be equivalent to the claim that $d_1(X)$ is coCartesian. Second, using the factorization $i=i_1\circ i_0$ the property  can also be reformulated by saying that the base change morphism associated to the following pasting diagram is an isomorphism at $(2,1)$ when evaluated at $X:$
$$
\xymatrix{
J_0\ar[r]^{i_0}\ar[d]_{i_0}\xtwocell[1,1]{}\omit& J_1\ar[r]^{i_1}\ar[d]_{\id}\xtwocell[1,1]{}\omit& J_2\ar[d]^{\id}\\
J_1\ar[r]_{\id}\ar[d]_{i_1}\xtwocell[1,1]{}\omit& J_1\ar[r]_{i_1}\ar[d]_{i_1}\xtwocell[1,1]{}\omit& J_2\ar[d]^{\id}\\
J_2\ar[r]_{\id}&J_2\ar[r]_{\id}&J_2
}
$$
But under our assumption on $d_2(X)$ and using the cofinality of $d^0\colon[1]\nach[2]$ this can be seen to be equivalent to the fact that $d_0(X)$ is coCartesian which then concludes our proof.
\end{proof}

Now, that we have established the properties of (co)Cartesian squares necessary for our purposes, we will quickly define left exact, right exact, and exact morphisms of derivators. 

\begin{definition}
A morphism of derivators \emph{preserves coCartesian squares} if it preserves homotopy left Kan extensions along $\ipush\colon\push\nach\square.$ Similarly, a morphism of derivators \emph{preserves Cartesian squares} if it preserves homotopy right Kan extensions along $\ipull\colon\pull\nach\square.$
\end{definition}

As an immediate consequence of Corollary \ref{cor_minimalmorphism} and Corollary \ref{cor_essimpointwise} we have the following result.

\begin{corollary}\label{cor_coCartpointwise}
Let $F\colon \DD\nach\DD'$ be a morphism of derivators. Then $F$ preserves coCartesian squares if and only if $F\colon\DD^M\nach\DD'^M$ preserves coCartesian squares for all categories $M$. Moreover, an object $X\in \DD^M(\square)$ is coCartesian if and only if the squares $X_m\in\DD(\square)$ are coCartesian for all objects $m\in M.$
\end{corollary}

\begin{definition}
Let $F\colon \DD\nach \DD '$ be a morphism of derivators.\\
{\rm i)} The morphism $F$ is \emph{left exact} if it preserves Cartesian squares and final objects.\\
{\rm ii)} The morphism $F$ is \emph{right exact} if it preserves coCartesian squares and initial objects.\\
{\rm iii)} The morphism $F$ is \emph{exact} if it is left exact and right exact.
\end{definition}

It follows immediately from this definition that a left exact morphism preserves, in particular, finite products and dually for a right exact morphism.

\begin{example}\label{example_exact}
{\rm i)} Let $(F,U)\colon\mathcal{M}\nach\mathcal{N}$ be a Quillen adjunction between combinatorial model categories. The morphism $\mathbb{L}F\colon\DD_{\mathcal{M}}\nach\DD_{\mathcal{N}}$ is right exact and the morphism $\mathbb{R}U\colon\DD_{\mathcal{N}}\nach\DD_{\mathcal{M}}$ is left exact. This holds more generally for an arbitrary adjunction of derivators.\\
{\rm ii)} Let $\DD$ be a derivator and let $u\colon L\nach M$ be a functor. The induced strict morphism of derivators $u^\ast\colon\DD^M\nach\DD^L$ is exact.
\end{example}

\subsection{Suspensions, Loops, Cones, and Fibers}

Let $\DD$ be a pointed derivator and let $J$ be a category. In this subsection we want to construct the suspension and loop functors on $\DD(J)$ and the cone and fiber functors on $\DD(J\times [1]).$ By Proposition \ref{prop_productpointed}, we can assume $J=e.$ 

Let us begin with the suspension functor  $\Sigma$ and the loop functor $\Omega.$ The `extension by zero functors' as given by Proposition \ref{prop_extzero} will again be crucial. Let us consider the following sequences of functors:
$$
\xymatrix{
e\ar[r]^{(0,0)}& \push \ar[r]^\ipush& \square & e\ar[l]_{(1,1)}, && e\ar[r]^{(1,1)}& \pull \ar[r]^\ipull & \square & e\ar[l]_{(0,0)}.
}$$
Since $(0,0)\colon e\nach\push$ is a sieve the homotopy right Kan extension functor $(0,0)_\ast$ gives us an `extension by zero functor' by Proposition \ref{prop_extzero}, and similarly for the homotopy left Kan extension $(1,1)_!$ along the cosieve $(1,1)\colon e\nach\pull$.

\begin{definition}
Let $\DD$ be a pointed derivator.\\
i) The \emph{suspension functor} $\Sigma$ is given by $\xymatrix{
\Sigma\colon\DD(e)\ar[r]^{(0,0)_\ast}& \DD(\push)\ar[r]^{\ipush _!}& \DD(\square)\ar[r]^{(1,1)^\ast}&\DD(e).
}$\\
ii) The \emph{loop functor} $\Omega$ is given by $\xymatrix{
\Omega\colon\DD(e)\ar[r]^{(1,1)_!}& \DD(\pull)\ar[r]^{\ipull _\ast}& \DD(\square)\ar[r]^{(0,0)^\ast}&\DD(e).
}$
\end{definition}

The motivation for these definitions should be clear from topology. Recall that given a pointed topological space $X$, the suspension $\Sigma X$ is constructed by first taking two instances of the canonical inclusion into the (contractible!) cone $\mathsf{C}X$ and then forming the pushout:
$$\xymatrix{
X\ar[r]\ar[d]& \mathsf{C}X&&X\ar[r]\ar[d]&\mathsf{C}X\ar[d]\\
\mathsf{C}X&&&\mathsf{C}X\ar[r]&\Sigma X
}
$$ 
We can consider this diagram as a homotopy pushout. The above definition abstracts precisely this construction. 

Of course, we want to show that these functors define an adjoint pair $(\Sigma, \Omega)\colon \DD(e)\rightharpoonup \DD(e).$ For this purpose, let us denote by $\mathcal{M}\subset \DD(\square),\;\mathcal{M}^{\push}\subset \DD(\push),$ and $\mathcal{M}^{\pull}\subset\DD(\pull)$ the respective full subcategories spanned by the objects $X$ with $X_{1,0}\cong 0 \cong X_{0,1}.$ 

\begin{proposition}
If $\DD$ is a pointed derivator, then we have an adjunction
$(\Sigma,\Omega)\colon \DD(e)\rightharpoonup \DD(e).$
\end{proposition}
\begin{proof}
With the notation established above, the suspension and the loop functor can be factored as follows:
$$
\xymatrix{
\Sigma\colon & \DD(e)\ar@{=}[d]\ar[r]^{(0,0)_\ast}_{\simeq} & \mathcal{M}^{\push}\ar@{=}[d]\ar[r]^{\ipush_!}& \mathcal{M}\ar@{=}[d]\ar[r]^{\ipull^\ast} & \mathcal{M}^{\pull}\ar@{=}[d]\ar[r]^{(1,1)^\ast}_{\simeq}& \DD(e)\ar@{=}[d]& \\
& \DD(e) & \mathcal{M}^{\push}\ar[l]_-{(0,0)^\ast}^\simeq& \mathcal{M}\ar[l]_-{\ipush^\ast} & \mathcal{M}^{\pull}\ar[l]_{\ipull_\ast}& \DD(e)\ar[l]^{\simeq}_{(1,1)_!}& \colon \Omega
}
$$
The existence of the factorization is clear and the fact that the functors $(0,0)_\ast$ and $(1,1)_!$ restricted this way are equivalences follows from their fully faithfulness and Proposition \ref{prop_extzero}. From this description, one sees immediately that we have an adjunction $(\Sigma,\Omega)$ which is, in fact, given as a composite adjunction of four adjunctions among which two are equivalences.
\end{proof}

Using similar constructions, one can introduce \emph{cone} and \emph{fiber functors} for pointed derivators. Again, the definition is easily motivated from topology. If we consider a map of pointed spaces $f\colon X\nach Y$ then the mapping cone $\mathsf{C}f$ of $f$ is constructed in two steps by forming a pushout as indicated in the next diagram:
$$
\xymatrix{
X\ar[r]^f\ar[d]&Y&&X\ar[r]^f\ar[d]&Y\ar[d]\\
\mathsf{C}X&&&\mathsf{C}X\ar[r]&\mathsf{C}f
}$$
To axiomatize this in the context of a pointed derivator, let us consider the following morphisms of posets:
$$\xymatrix{
[1]\ar[r]^i& \push \ar[r]^{i_{\push}}& \square & \pull\ar[l]_{i_{\pull}} & [1]\ar[l]_j
}$$
Here, $i$ is the sieve classifying the horizontal arrow while $j$ is the cosieve classifying the vertical arrow. In particular, by Proposition \ref{prop_extzero}, we have again extension by zero functors $i_\ast$ and $j_!.$

\begin{definition}
Let $\DD$ be a pointed derivator.\\
{\rm i)} The \emph{cone functor} $\cone \colon \DD([1])\nach \DD([1])$ is defined as the composition: $$\cone\colon \DD([1])\stackrel{i_\ast}{\nach}\DD(\push)\stackrel{{i_{\push}}_!}{\nach}\DD(\square)\stackrel{j^\ast}{\nach}\DD([1])$$\\
{\rm ii)} The \emph{fiber functor} $\fibre \colon \DD([1])\nach \DD([1])$ is defined as the composition: $$\fibre\colon \DD([1])\stackrel{j_!}{\nach}\DD(\pull)\stackrel{{i_{\pull}}_\ast}{\nach}\DD(\square)\stackrel{i^\ast}{\nach}\DD([1])$$
Moreover, let $\mathsf{C}\colon \DD([1])\nach \DD(e)$ be the functor obtained from the cone functor by evaluation at 1, and similarly let $\mathsf{F}\colon \DD([1])\nach \DD(e)$ be the functor obtained from the fiber functor by evaluation 0.
\end{definition}

Proposition \ref{prop_cocartiso} shows that the cone $\mathsf{C}f$ of an isomorphism $f$ is the zero object $0.$ In general, the converse is only true in the stable situation (cf.\ Proposition \ref{prop_isocone}). There is the following counterexample to the converse in the unstable situation. 

\begin{counterexample}
Let $\mathcal{E}$ be an exact category in the sense of Quillen (cf.\ \cite{quillen_ktheory}). Moreover, let us assume $\mathcal{E}$ to have enough injectives but also that $\mathcal{E}$ is not Frobenius, i.e., the classes of injectives and projectives do not coincide. The stable category~$\underline{\mathcal{E}}$ which is obtained from $\mathcal{E}$ by dividing out the maps factoring over injectives is a `suspended category' in the sense of \cite{kellervossieck}. Let now $X$ be an object of $\mathcal{E}$ of injective dimension $1$ and let $0\nach X\nach I^0=I\nach I^1=\Sigma X\nach 0$ be an injective resolution of $X$. By definition of the suspended structure on $\underline{\mathcal{E}}$ (cf.\ \cite{kellervossieck} or \cite[Chapter I]{happel_triangulated}) the diagram
$$
\xymatrix{
X\ar[r]^u\ar[d]_{\id}&I\ar[d]^{\id}\ar[r]^v&\Sigma X\ar[d]^{\id}\\
X\ar[r]&I\ar[r]&\Sigma X
}
$$
gives rise to the distinguished triangle $X\stackrel{u}{\nach} I\stackrel{v}{\nach}\Sigma X\stackrel{\id}{\nach}\Sigma X.$ Since $\Sigma X$ is trivial in the stable category $\underline{\mathcal{E}}$ the morphism $u$ is an example of a morphism which is not an isomorphism but still has a vanishing cone. In the stable situation, i.e., in the Frobenius case, this counterexample cannot exist. In fact, the above resolution of $X$ would split because $\Sigma X$ is by assumption injective, hence projective, showing that the injective dimension of $X$ is zero. This example can be made into an example about pointed derivators by using \cite{keller_exact}.
\end{counterexample}

As a preparation for the next proof, let us denote by $\mathcal{N}\subset \DD(\square),\:\mathcal{N}^{\push}\subset \DD(\push),$ and $\mathcal{N}^{\pull}\subset\DD(\pull)$ the respective full subcategories spanned by the objects $X$ with $X_{0,1}\cong 0.$

\begin{proposition}
Let $\DD$ be a pointed derivator, then we have an adjunction: 
$$(\cone,\fibre)\colon \DD([1])\rightharpoonup \DD([1])$$
\end{proposition}
\begin{proof}
There are the following factorizations of the cone and fiber functors:
$$\xymatrix{
\cone\colon & \DD([1])\ar@{=}[d]\ar[r]^{i_\ast}_{\simeq}& \mathcal{N}^{\push}\ar@{=}[d]\ar[r]^{\ipush_!}&\mathcal{N}\ar@{=}[d]\ar[r]^{\ipull^\ast}&\mathcal{N}^{\pull}\ar@{=}[d]\ar[r]^{j^\ast}_{\simeq}&\DD([1])\ar@{=}[d]&\\
& \DD([1])&\mathcal{N}^{\push}\ar[l]^{\simeq}_{i^\ast}&\mathcal{N}\ar[l]_{\ipush^\ast}&
\mathcal{N}^{\pull}\ar[l]_{\ipull_\ast}&\DD([1])\ar[l]_{j_!}^{\simeq}&:\fibre
}$$
The existence of these factorizations is again obvious and the fact that the outer functors are equivalences follows again from Proposition \ref{prop_extzero}. Thus, this shows that the pair $(\cone,\fibre)$ is the composition of four adjunctions among which two are equivalences.
\end{proof}

In the literature, there is also an alternative description of some of the functors we just introduced. This alternative description is, for example, helpful in the understanding of morphisms in~$\DD([1])$ which induce zero morphisms on underlying diagrams. Using our explicit construction of the (co)exceptional inverse image functors at the end of Subsection~\ref{subsec_pointed}, let us quickly show these two approaches to be equivalent. For this purpose, let $\DD$ again be a pointed derivator and let us consider the category $[1]$ together with the cosieve $1\colon e\nach[1]$ and the sieve $0\colon e\nach [1].$ Corollary~\ref{cor_pointed} implies that we have adjunctions:
$$
(1^?,1_!)\colon\DD([1])\rightharpoonup\DD(e)\qquad\mbox{and}
\qquad(0_\ast,0^!)\colon\DD(e)\rightharpoonup\DD([1])
$$
The formulas via the mapping cylinder constructions can be made very explicit in this case so that we have the following descriptions of the additional adjoints $1^?$ and $0^!:$
$$
\xymatrix{
1^?\colon\DD([1])\ar[r]^{j_\ast}_{\simeq}&\DD(\push,(0,1))\ar[r]^-{{\pr_1}_!}&\DD([1],0)\ar[r]^-{1^\ast}_-\simeq&\DD(e)\\
0^!\colon\DD([1])\ar[r]^{j_!}_\simeq&\DD(\pull,(1,0))\ar[r]^-{{\pr_1}_\ast}&\DD([1],1)\ar[r]^-{0^\ast}_-\simeq&\DD(e)
}
$$
In both formulas, $j$ denotes the functor classifying the horizontal arrow and the functors $\pr_1$ are suitable restrictions of the projection on the first component $\square\nach [1].$ It follows from Lemma~\ref{lemma_terminal} that in both cases the composition of the last two functors is naturally isomorphic to the homotopy colimit and homotopy limit functor respectively. A final application of (Der4) then implies the following result.

\begin{proposition}\label{prop_altdesc}
Let $\DD$ be a pointed derivator then we have the following natural isomorphisms: 
$$\mathsf{C}\cong 1^?,\qquad \Sigma\cong 1^?\circ 0_\ast,\qquad\mathsf{F}\cong 0^!,\qquad\mbox{and}\qquad\Omega\cong 0^!\circ 1_!$$
In particular, we have adjunctions $(\mathsf{C},1_!)\colon\DD([1])\rightharpoonup\DD(e)$ and $(0_\ast,\mathsf{F})\colon\DD(e)\rightharpoonup\DD([1]).$
\end{proposition}

The above definitions can easily be extended (using Example \ref{ex_morphisms}) to morphisms at the level of derivators. Thus, given a pointed derivator $\DD$ we obtain, in particular, adjunctions of derivators 
$$(\Sigma,\Omega)\colon\DD\rightharpoonup\DD\qquad\mbox{and}
\qquad(\cone,\fibre)\colon\DD^{[1]}\rightharpoonup\DD^{[1]}.$$ 

Since the construction of the above functors is based only on certain extension by zero functors and the formation of some (co)Cartesian squares the following proposition is immediate. It applies, in particular, to the precomposition morphisms $v^\ast\colon\DD^M\nach\DD^L$ for a pointed derivator $\DD.$

\begin{proposition}
Let $G\colon\DD\nach\DD'$ be a morphism of pointed derivators.\\
{\rm i)} If $G$ is left exact then we have canonical isomorphisms 
$$G\circ\Omega \nach \Omega\circ G \qquad\mbox{and}\qquad G\circ\fibre\nach \fibre\circ\: G.$$
{\rm ii)} If $G$ is right exact then we have canonical isomorphisms 
$$\Sigma\circ G\nach G\circ \Sigma\qquad\mbox{and}\qquad\cone\circ\: G\nach G\circ \cone.$$
\end{proposition}

\section{Stable derivators}\label{sec_stable}

\subsection{The additivity of stable derivators}

In this subsection, we come to the central notion of a stable derivator. Similarly to the situation of a stable model category or a stable $\infty$-category, one adds a `linearity condition' to the pointed situation. This will ensure, in particular, that the suspension and the loop functor define a pair of inverse equivalences
$$(\Sigma, \Omega)\colon \DD(e)\stackrel{\simeq}{\nach} \DD(e).$$ This notion was introduced by Maltsiniotis in \cite{maltsiniotis2} by forming a combination of the axioms of Grothendiecks derivators \cite{grothendieck} and Franke's systems of triangulated diagram categories \cite{franke}. More details on the history can be found in the paper \cite{cisinskineeman} by Cisinski and Neeman. 

\begin{definition}
A strong derivator $\DD$ is \emph{stable} if it is pointed and if an object of $\DD(\square)$ is coCartesian if and only if it is Cartesian.
\end{definition}

The strongness property will be crucial in two situations in the construction of the canonical triangulated structures. Let us call a square \emph{biCartesian} if it satisfies the equivalent conditions of being Cartesian or coCartesian. 

\begin{example}
\rm{i)} Let $\mathcal{M}$ be a stable combinatorial model category then the associated derivator~$\DD_{\mathcal{M}}$ is stable. Thus, we have, in particular, the stable derivator associated to unbounded chain complexes, modules over a differential graded algebra, spectra based on simplicial sets and module spectra over a given symmetric ring spectrum. These derivators can be endowed with some additional structure:  they are examples of monoidal derivators resp.\ derivators tensored over a monoidal derivator as discussed in \cite{groth_monder, groth_enriched}.\\
\rm{ii)} A derivator $\DD$ is stable if and only if the dual derivator $\DD^{\op}$ is stable.
\end{example}

Let us begin by the following convenient result. 

\begin{proposition}\label{prop_againtriang}
Let $\DD$ be a stable derivator and let $M$ be a category. Then $\DD^M$ is again stable.
\end{proposition}
\begin{proof}
It is immediate that a derivator $\DD$ is strong if and only if $\DD^M$ is strong for all categories $M.$ Moreover, we know that $\DD^M$ is pointed by Proposition \ref{prop_productpointed}. Thus, let us consider the (co)Cartesian squares. For an object $X\in\DD^M(\square)$, using Corollary \ref{cor_coCartpointwise}, we have that $X$ is coCartesian if and only if $X_m\in\DD(\square)$ is coCartesian for all $m\in M.$ Using the stability of $\DD$ and the corresponding result for Cartesian squares in $\DD^M(\square)$ we are done.
\end{proof}

We give immediately the expected result on the suspension and loop functors in this stable situation. Recall the definition of the categories $\mathcal{M},\:\mathcal{M}^{\pull},\:\mathcal{M}^{\push},$ and the factorization of $(\Sigma,\Omega)$ in the case of a pointed derivator. Let us denote, in addition, by $\mathcal{M}^\Sigma\subset\mathcal{M}$ (resp.\ $\mathcal{M}^\Omega\subset \mathcal{M}$) the full subcategory spanned by the coCartesian (resp.\ Cartesian) squares. With this notation, in the case of a \emph{pointed} derivator, there is the following additional factorization of $(\Sigma,\Omega):$
$$
\xymatrix{
\Sigma\colon & \DD(e)\ar@{=}[d]\ar[r]^{(0,0)_\ast}_{\simeq} & \mathcal{M}^{\push}\ar@{=}[d]\ar[r]^{\ipush_!}_\simeq& \mathcal{M}^\Sigma\ar[r]^{\ipull^\ast} & \mathcal{M}^{\pull}\ar@{=}[d]\ar[r]^{(1,1)^\ast}_{\simeq}& \DD(e)\ar@{=}[d]& \\
& \DD(e) & \mathcal{M}^{\push}\ar[l]_-{(0,0)^\ast}^\simeq& \mathcal{M}^\Omega\ar[l]_-{\ipush^\ast} & \mathcal{M}^{\pull}\ar[l]_{\ipull_\ast}^\simeq& \DD(e)\ar[l]^{\simeq}_{(1,1)_!}& \colon \Omega
}
$$
In this diagram, all but possibly the two restriction functors in the middle are equivalences. 
In the case of a \emph{stable} derivator, we have $\mathcal{M}^\Sigma=\mathcal{M}^\Omega$ and these two restriction functors are also equivalences:
$$
\xymatrix{
\Sigma\colon & \DD(e)\ar@{=}[d]\ar[r]^{(0,0)_\ast}_{\simeq} & \mathcal{M}^{\push}\ar@{=}[d]\ar[r]^{\ipush_!}_\simeq& \mathcal{M}^\Sigma\ar@{=}[d]\ar[r]^{\ipull^\ast}_\simeq & \mathcal{M}^{\pull}\ar@{=}[d]\ar[r]^{(1,1)^\ast}_{\simeq}& \DD(e)\ar@{=}[d]& \\
& \DD(e) & \mathcal{M}^{\push}\ar[l]_-{(0,0)^\ast}^\simeq& \mathcal{M}^\Omega\ar[l]_-{\ipush^\ast}^\simeq & \mathcal{M}^{\pull}\ar[l]_{\ipull_\ast}^\simeq& \DD(e)\ar[l]^{\simeq}_{(1,1)_!}& \colon \Omega
}
$$
\noindent
This proves the first half of the next result. The second half can be proved in a similar way.

\begin{proposition}\label{prop_stablesigma}
Let $\DD$ be a stable derivator, then we have equivalences of derivators
$$(\Sigma,\Omega)\colon\DD\stackrel{\simeq}{\nach}\DD\qquad\mbox{and}
\qquad (\cone,\fibre)\colon\DD^{[1]}\stackrel{\simeq}{\nach}\DD^{[1]}.$$ 
\end{proposition}

Let us mention the following result which shows that in the stable situation isomorphisms can be characterized by the vanishing of the cone. We use the same notation as in Proposition \ref{prop_cocartiso}.

\begin{proposition}\label{prop_isocone}
Let $\DD$ be a stable derivator and let $X\in\DD(\square).$ If two of the three following statements hold for the square $X$ then so does the third one:\\
{\rm i)} The square $X$ is coCartesian,\\
{\rm ii)} The arrow $d_0^v X$ is an isomorphism,\\
{\rm iii)} The arrow $d_1^v X$ is an isomorphism.\\
In particular, an object $f\in\DD([1])$ is an isomorphism if and only if the cone $\mathsf{C}f$ is zero. 
\end{proposition}
\begin{proof}
For the first part we can apply Proposition \ref{prop_cocartiso} to see that we only have to show that {\rm i)} and {\rm ii)} imply {\rm iii)}. But this statement follows from the dual of Proposition \ref{prop_cocartiso} which can be applied because every coCartesian square is also Cartesian in the stable situation. Finally, the second part follows from the first part when applied to the special case of the defining square of the cone.
\end{proof}

The next aim is to show that, in the stable case, finite coproducts and finite products in $\DD(J)$ are canonically isomorphic. By Proposition \ref{prop_againtriang}, we can assume that $J=e.$ But let us first mention the following result which is immediate from Proposition \ref{prop_cancel} on the composition and the cancellation properties of (co)Cartesian squares. That result is crucial in order to establish the semi-additivity.

\begin{proposition}\label{prop_2outof3}
Let $\DD$ be a stable derivator and let $X\in \DD([2]\times [1]).$ If two of the squares $d_0(X),d_1(X),$ and $d_2(X)$ are biCartesian, then so is the third one.
\end{proposition}

We now give the result on the semi-additivity of the values of a stable derivator, i.e., we want to show that the values then admit finite biproducts. We know already from Proposition \ref{prop_(co)limits} that the values of an arbitrary derivator admit finite coproducts and finite products.

\begin{proposition}\label{prop_preadditive}
Let $\DD$ be a stable derivator and consider a functor $u\colon J\nach K.$ Then finite coproducts and finite products in $\DD(J)$ are canonically isomorphic. Moreover, these are preserved by $u^\ast, u_!,$ and $u_\ast.$
\end{proposition}
\begin{proof}
For the first part, it is again enough to show the result for the case $J=e.$ Let us consider the inclusion $j_2\colon L_2\nach L_3$ of the left poset $L_2$ in the right poset $L_3$:
$$\xymatrix{
 &(1,0)\ar[r]&(2,0) &&& (0,0)\ar[d]\ar[r] & (1,0)\ar[r] & (2,0)\\
 (0,1)\ar[d] && &&& (0,1)\ar[d] &&\\
 (0,2) && &&& (0,2) &&
}
$$
Moreover, let $j_1\colon e \sqcup e \nach L_2$ be the map $(1,0)\sqcup (0,1)$ and let $j_3\colon L_3\nach [2]\times [2]=L$ be the obvious inclusion. Since $j_1$ is a sieve the homotopy Kan extension functor~${j_1}_\ast$ is an `extension by zero functors' by Proposition~\ref{prop_extzero}, and similarly for the homotopy Kan extension functor~${j_2}_!$ associated to the cosieve~$j_2$. Let us consider the functor:
$$\DD(e)\times\DD(e)\simeq \DD(e\sqcup e)\stackrel{{j_1}_\ast}{\nach}\DD(L_2)\stackrel{{j_2}_!}{\nach}\DD(L_3)\stackrel{{j_3}_!}{\nach}\DD(L)$$
The image $Q\in \DD(L)$ of a pair $(X,Y)\in\DD(e)\times\DD(e)$ under this functor has as underlying diagram:
$$
\xymatrix{
&0\ar[r]\ar[d]& X\ar[r]\ar[d] & 0\ar[d]\\
\dia_L (Q):&Y\ar[r]\ar[d] & B \ar[r]\ar[d] & Y'\ar[d]\\
&0\ar[r]& X' \ar[r]& Z
}$$
Let us denote the four inclusions of the smaller squares in $L$ by $i_k,\;k=1,\ldots ,4,$ i.e., let us set
$$i_1=d^2\times d^2,\qquad i_2=d^0\times d^2,\qquad i_3=d^2\times d^0,\qquad\mbox{and}\qquad i_4=d^0\times d^0.$$
An application of Proposition \ref{prop_detect} to these inclusions $i_k\colon\square\nach L,\;k=1,\ldots ,4,$ and $f=j_3$ allows us to deduce that all squares are biCartesian. In fact, in all four cases, $i_k(1,1)\notin \im(j_3)$ and we only have to check that the induced functors $\tilde{i}_k\colon \push\nach {L-i_k(1,1)}_{/i_k(1,1)}$ are right adjoints. For $k=1$, this functor is an isomorphism while in the other three cases Lemma \ref{lemma_someadj} applies. By~Proposition \ref{prop_2outof3}, also the composite squares $(d_2\times d_1)(Q)$ and $(d_1\times d_2)(Q)$ are biCartesian. Hence, Proposition~\ref{prop_cocartiso} ensures that we have isomorphisms $X\cong X'$ and $Y\cong Y'$. Similarly, the square $(d_1\times d_1)(Q)$ is biCartesian and we obtain an isomorphism $Z\cong 0$. Thus, we see that $B$ is simultaneously a coproduct of $X$ and $Y$ and a product of $X'\cong X$ and $Y'\cong Y$.\\
The fact that these biproducts are preserved by $u^\ast, u_!,$ and $u_\ast$ follows immediately since each of the three functors has an adjoint functor on at least one side.
\end{proof}

\begin{corollary}
Let $\DD$ be a stable derivator and let $J$ be a category. Every object of $\DD(J)$ is canonically a commutative monoid object and a cocommutative comonoid object. In particular, the morphism set $\hhom_{\DD(J)}(X,Y),\;X,Y\in \DD(J),$ carries canonically the structure of an abelian monoid.
\end{corollary}
\begin{proof}
For $X\in\DD(J),$ the diagonal map $\Delta_X\colon X\nach X\times X\cong X\sqcup X$ is counital, coassociative and cocommutative. Dually, the codiagonal $\nabla_X\colon X\times X\cong X\sqcup X\nach X$ is unital, associative and commutative. These can be used to define the sum of two morphism $f,g\colon X\nach Y$ using the usual convolution or cup product, i.e., as
$$f+g\colon X\stackrel{\Delta _X}{\nach}X\times X\stackrel{f\times g}{\nach}Y\times Y\cong Y\sqcup Y \stackrel{\nabla_Y}{\nach}Y$$
and it is immediate that this defines the structure of an abelian monoid on $\hhom_{\DD(J)}(X,Y).$
\end{proof}

We will from now on use the standard notation $\oplus$ for the biproduct. The next aim is to show that objects of the form $\Omega X$ (resp.\ $\Sigma X$) are even abelian \emph{group} (resp.\ \emph{cogroup}) objects. We give the proof in the case of $\Omega X$ in which case the constructions can be motivated by the process of concatenation of loops in topology. Let us begin with some preparations. Since the aim is to `model categorically' the concatenation and inversion of loops we have to consider finite direct sums of `loop objects'. For the construction of the finite sums of loop objects there is the following conceptual approach which admits an obvious dualization.  Let $\pull_n$ be the poset with objects $e_0,\ldots ,e_n$ and $t$ and with ordering generated by $e_i\leq t,\;i=0,\ldots ,n$. The pictures of $\pull_n$ for $n=1$ and $n=2$ are:
$$
\xymatrix{
& e_1\ar[d] && e_1\ar[dr] & e_2\ar[d]\\
e_0\ar[r]& t&&e_0\ar[r]&t
}
$$
Let $\Fin$ denote the category of the finite sets $\la n\ra=\{0,\ldots ,n\}$ with all set-theoretic maps as morphisms between them. The assignment $\la n\ra\auf \pull_n$
can be extended to a functor $\Fin \nach \Cat$ if we send $f\colon \la k \ra \nach \la n \ra$ to $\pull_f\colon \pull _k\nach \pull_n$ with $\pull_f(e_i)=e_{f(i)}$ and $\pull_f(t)=t.$ Since $t\colon e \nach \pull _n$ is a cosieve, $t_!\colon \DD(e)\nach \DD(\pull_n)$ gives us an `extension by zero functor'. Define $P_n$ as 
$$
\xymatrix{
P_n\colon \DD(e)\ar[rr]^{t_!}&&\DD(\pull _n)\ar[rr]^{\Holim_{\pull_n}}&& \DD(e)  
}
$$
\noindent
and note that we have a canonical isomorphism $P_1X\cong\Omega X.$ This construction can be extended to a functor as described in the following lemma.

\begin{lemma}\label{lemma_gamma}
Let $\DD$ be a stable derivator. The above construction defines a bifunctor: $$P\colon \Fin^{\op}\times\DD(e) \nach \DD(e)\colon (\la n \ra ,X)\auf P_nX$$
\end{lemma}
\begin{proof}
The functoriality of $P$ in the second variable is obvious so let us assume we are given a morphism $f\colon\la k\ra\nach\la n\ra.$ From such a morphism $f$ we obtain the following diagram given on the left-hand-side:
$$
\xymatrix{
e\ar[d]_{t}\ar[r]&e\ar[d]^{t}&&\DD(e)\ar[d]_{t_!}&\DD(e)\ar[d]^{t_!}\ar[l]\\
\pull_k\ar[r]^{\pull_f}\ar[d]_p&\pull_n\ar[d]^p&&\DD(\pull_k)\ar[d]_{\Holim}\xtwocell[-1,1]{}\omit&\DD(\pull_n)\ar[d]^{\Holim}\ar[l]_{\pull_f^\ast}\xtwocell[1,-1]{}\omit\\
e\ar[r]&e&&\DD(e)&\DD(e)\ar[l]
}
$$
The formation of the corresponding base change morphisms gives rise to the pasting diagram on the right (note that we had to use both variants here). Using the fact that isomorphisms can be detected pointwise and (Der4) it is easy to check that the upper $2$-cell is invertible. Thus we can define $P_f$ as the following composition:
$$P_f\colon\quad P_n=\Holim_{\pull_n}\circ\:t_!\nach \Holim_{\pull_k}\circ\pull_f^\ast\circ\:t_!\nach \Holim_{\pull_k}\circ\:t_!=P_k$$
The functoriality of this construction follows from the nice behavior of pasting with respect to both the inversion of natural transformations and base change.
\end{proof}

Let us fix notation for some morphisms in $\Fin$. Given a $(k+1)$-tuple $(i_0,i_1,\ldots ,i_k)$ of elements of $\la n \ra$ let us denote by $(i_0i_1\ldots i_k)$ the corresponding morphism $\la k\ra \nach \la n\ra$ which sends $j$ to $i_j$. For~$n\geq 1$ and $1\leq k \leq n,$ we have thus the morphism $(k-1,k)\colon \la 1\ra \nach \la n \ra.$ So, for a stable derivator~$\DD$ and an object~$X\in\DD(e),$ we obtain by the last lemma induced maps: 
$$(k-1,k)^\ast=P((k-1,k),\id_X)\colon P_nX\nach P_1X\cong \Omega X$$ 
These maps taken together define the following \emph{Segal maps} and satisfy the `usual' Segal condition~(\cite{segal_categories}).

\begin{lemma}\label{lemma_Segal}
Let $\DD$ be a stable derivator and let $X\in\DD(e)$. For $n\geq 1$ and $1\leq k\leq n,$ the~$(k-1,k)^\ast$ together define a natural isomorphism in $\DD(e):$
$$s=s_n\colon P_nX\stackrel{\cong}{\nach} \prod_{k=1}^nP_1(X)\cong\bigoplus_{k=1}^n\Omega X$$
\end{lemma}
\begin{proof}
By induction on $n$ and by the functoriality of $P_\bullet X,$ it is enough to check this for $n=2$. Let $J$ be the poset obtained from $\pull_2$ by adding two new elements $\omega_0$ and $\omega_1$ such that $\omega_0\leq e_0,e_1$ and $\omega _1\leq e_1,e_2.$ Moreover, let us denote the resulting inclusion by $j\colon \pull_2\nach J.$ Under the obvious isomorphism $J\cong [1]\times \pull,$ we can consider the adjunction $(d^1\times \id, s^0\times \id)\colon \pull \rightharpoonup [1]\times\pull$ as an adjunction $(L,R)\colon \pull\rightharpoonup J.$ By Proposition \ref{prop_cofinality} we have a natural isomorphism between $P_2$ and
$$
\xymatrix{
\DD(e) \ar[r]^{t_!}& \DD(\pull_2)\ar[r]^{j_\ast}& \DD(J)\ar[r]^{L^\ast}& \DD(\pull)\ar[r]^{\Holim}&\DD(e).
}
$$
But it is easy to see that the composition of the first three functors evaluated on $X$ yields a diagram which vanishes at $t$ and is isomorphic to $\Omega X$ at the two remaining arguments. It thus follows that we have an isomorphism $P_2(X)\cong\Omega X\oplus \Omega X$ induced by the Segal map.
\end{proof}

Having the functorial construction of finite direct sums of \emph{loop objects} at our disposal, we want to show now that $\Omega X$ is always canonically an abelian group object. As an intermediate step, let us construct a pairing $\star\colon \Omega X\oplus \Omega X\nach \Omega X$ which will be called the \emph{concatenation map}. By the last lemma we can invert the Segal maps and hence define the pairing by the following composition:
$$
\xymatrix{
\star \colon \Omega X\oplus \Omega X & P_2(X)\ar[l]^-\cong_-s\ar[r]^-{(02)^\ast}&\Omega X
}
$$

\begin{lemma}\label{lemma_concatenation}
Let $\DD$ be a stable derivator and let $X$ be an object of $\DD(e).$ The concatenation map $\star\colon \Omega X\oplus\Omega X\nach \Omega X$ is an associative pairing on $\Omega X$.
\end{lemma}
\begin{proof}
Let $U$ be a further object of $\DD(e)$ and consider three morphisms $f,g,h\colon U\nach \Omega X$ in $\DD(e)$. In the following diagram, all maps labeled by $s$ are Segal maps:
$$
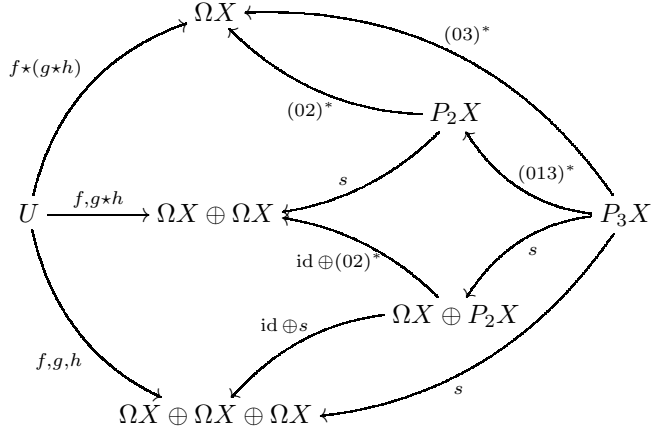
\xymatrix{
& \Omega X&& \\
&& P_2X\ar@/^1.0pc/[lu]^{(02)^\ast}\ar@/^1.0pc/[ld]_s & \\
U\ar@/_1.4pc/[ddr]_{f,g,h}\ar[r]^-{f,g\star h}\ar@/^1.4pc/[ruu]^{f \star (g\star h)} & \Omega X\oplus \Omega X& & P_3X\ar@/^1.0pc/[lu]_{(013)^\ast}\ar@/_2.2pc/[lluu]_{(03)^\ast}\ar@/_1.0pc/[ld]^s \ar@/^2.2pc/[lldd]^s\\
&& \Omega X\oplus P_2X\ar@/_1.0pc/[lu]^-{\id \oplus (02)^\ast}\ar@/_1.2pc/[dl]_-{\id \oplus s} &\\
& \Omega X\oplus \Omega X\oplus \Omega X & & 
}$$
The two quadrilaterals on the left commute by definition of the concatenation and the right one commutes by functoriality of $P_\bullet$. We can thus deduce the relation $f\star (g\star h)=(03)^\ast m(f,g,h)$ where $m(f,g,h)\colon U\nach P_3 X$ is the unique map such that $s\circ m(f,g,h)=(f,g,h).$ This `associative description' of $f\star (g\star h)$ together with the Yoneda lemma implies the associativity of the concatenation map.
\end{proof}

Heading for the additive inverse of the identity on loop objects, let us consider the only non-trivial automorphism $\sigma\colon \la 1\ra\nach \la 1\ra$ in $\Fin$. Then $\pull_\sigma\colon \pull\nach\pull$ is the isomorphism interchanging the vertices $(1,0)$ and $(0,1)$. There is thus an induced automorphism $\sigma^\ast=(10)^\ast\colon \Omega X\nach \Omega X$ which we call the \emph{inversion of loops}.

\begin{proposition}\label{prop_sign}
Let $\DD$ be a stable derivator and let $X\in \DD(e)$. The inversion of loops map $\sigma^\ast\colon\Omega X\nach \Omega X$
is an additive inverse to $\id_{\Omega X}.$ In particular, $\Omega X\in \DD(e)$ is an abelian \emph{group} object. 
\end{proposition}
\begin{proof}
By functoriality of the construction $P_\bullet X$, there is a right action of the symmetric group on three letters on $P_2X$. We want to describe the corresponding action on $\Omega X\oplus \Omega X$ obtained by conjugation with the Segal map $s$. The strategy of the proof is then to use this action in order to relate the concatenation product and the addition of morphisms. 

For different elements $i,j\in \la 2 \ra$ let us denote by $\sigma _{ij}$ the associated transposition. One checks that the following diagram commutes
$$\xymatrix{
P_2X\ar[d]_s\ar[rr]^{\sigma_{02}^\ast}&&P_2X\ar[d]^s\\
\Omega X\oplus \Omega X\ar[rr]_{\left( \begin{array}{cc} 0&\sigma^\ast \\ \sigma ^\ast & 0\end{array}\right)}&&\Omega X\oplus \Omega X
}$$
where the arrows labeled by $s$ are again Segal maps. From the equality of the maps $$\sigma_{01}\circ (01)=(01)\circ \sigma\colon \la 1\ra \nach \la 2\ra$$ we conclude that the endomorphism of $\Omega X\oplus \Omega X$ corresponding to $\sigma _{01}$ is a lower triangular matrix
$$s\circ \sigma _{01}^\ast \circ s^{-1}=\left(\begin{array}{cc}\sigma^\ast & 0\\ \alpha & \beta\end{array}\right)\colon\Omega X\oplus \Omega X\nach \Omega X\oplus \Omega X$$
for some maps $\alpha,\beta \colon \Omega X\nach \Omega X.$ The fact that $\sigma_{01}$ is an involution implies the relations: 
$$\alpha \sigma ^\ast + \beta\alpha = 0 \qquad\mbox{and}\qquad \beta^2=\id$$
The aim is now to show that both maps $\alpha$ and $\beta$ are identities which would in particular imply that $\sigma ^\ast$ is an additive inverse of $\id _{\Omega X}$. 

From the relation $(02)=\sigma_{01}\circ (12)$ we immediately get $(02)^\ast=(12)^\ast\circ \sigma_{01}^\ast\colon P_2X\nach \Omega X.$ Using the matrix description of the map induced by $\sigma _{01}$ we see that for two maps $f,g\colon U\nach \Omega X$ there is the following formula for the concatenation product:
$$f\star g=\alpha f + \beta g\colon U\nach \Omega X$$ 
By Lemma \ref{lemma_concatenation} we know that the concatenation pairing is associative. If we take $U=\Omega X$ and compare the two expressions for $(0\star 0)\star \id_{\Omega X}$ and $0\star(0\star \id_{\Omega X})$ we already obtain the first intended relation $\beta= \id_{\Omega X}.$ 

Instead of using $(02)=\sigma_{01}\circ (12),$ we can also use the relation $(02)=\sigma _{12}\circ (01)\colon \la 1\ra \nach \la 2 \ra$ to obtain a further description of the concatenation product. First, since 
$$\sigma _{12}=\sigma _{02}\circ \sigma_{01}\circ \sigma _{02}\colon \la 2 \ra\nach \la 2\ra$$
we obtain that the endomorphism on $\Omega X\oplus \Omega X$ induced by $\sigma _{12}^\ast$ has the following matrix description:
$$s\circ \sigma_{12}^\ast \circ s^{-1}=\left(\begin{array}{cc}\sigma^\ast \beta \sigma ^\ast & \sigma ^\ast \alpha \sigma ^\ast \\ 0& \sigma^\ast\end{array}\right)\colon\Omega X\oplus \Omega X\nach \Omega X\oplus \Omega X$$
From this and the formula $(02)^\ast =(01)^\ast \circ \sigma_{12}^\ast$ we see that the concatenation product can also be written as: 
$$f\star g= \sigma^\ast \beta \sigma^\ast f + \sigma ^\ast \alpha \sigma ^\ast g\colon U\nach \Omega X$$ 
A comparison of these two descriptions concludes the proof since we obtain $\alpha=\sigma ^\ast \beta \sigma ^\ast =\id_{\Omega X}.$
\end{proof}

\begin{remark}
Although we will not make use of this remark we want to emphasize the following. The proof of the last proposition shows that the addition on mapping spaces into loop objects coincides with the pairing induced by the concatenation of loops. Similarly, additive inverses are given by the inversion of loops. Thus for maps $f,g\colon U\nach \Omega X$ we have: 
$$f+g=f\star g\qquad \mbox{and}\qquad -f\stackrel{def}{=}\sigma^\ast f$$
\end{remark}

A combination of this proposition, the result on the semi-additivity of $\DD(J)$ (Proposition \ref{prop_preadditive}), and the fact that $(\Sigma,\Omega)$ is a pair of inverse equivalences in the stable situation gives us immediately the following corollary.

\begin{corollary}\label{cor_additive}
If $\DD$ is a stable derivator then $\DD(J)$ is an additive category for an arbitrary $J.$ Moreover, for an arbitrary functor $u\colon J\nach K$, the induced functors $u^\ast,u_!,$ and $u_\ast$ are additive.
\end{corollary}

\subsection{The canonical triangulated structures}

We can now attack the main result of this section, namely, that given a stable derivator $\DD$ then the categories $\DD(J)$ are canonically triangulated categories. Using Proposition \ref{prop_againtriang}, we can again assume without loss of generality that we are in the case $J=e$. The suspension functor of the triangulated structure will be the suspension functor $\Sigma \colon \DD(e)\nach \DD(e)$ we constructed already. 

Thus, let us construct the class of distinguished triangles. For this purpose, let $K$ denote the poset:
$$
\xymatrix{
(0,0)\ar[d]\ar[r]&(1,0)\ar[r]&(2,0)\\
(0,1)&&
}
$$
Moreover, let $i_0\colon [1]\nach K$ be the map classifying the left horizontal arrow and let $i_1\colon K\nach [2]\times [1]$ be the inclusion. Let us denote the composition by $i\colon [1]\stackrel{i_0}{\nach}K\stackrel{i_1}{\nach}[2]\times [1].$ Again, since~$i_0$ is a sieve, ${i_0}_\ast$ gives us an extension by zero functor. Let us consider the functor: 
$$T\colon \DD([1])\stackrel{{i_0}_\ast}{\nach}\DD(K)\stackrel{{i_1}_!}{\nach}\DD([2]\times [1])$$
We claim that the squares $d_0T(f),\:d_1T(f),$ and $d_2T(f)\in \DD(\square)$ are then biCartesian for an arbitrary $f\in\DD([1]).$ Moreover, if the underlying diagram of $f$ is $X\nach Y$ then we have canonical isomorphisms $T(f)_{2,1}\cong \Sigma X$ and $T(f)_{1,1}\cong \mathsf{C}(f).$
In fact, by Proposition \ref{prop_2outof3}, it is enough to show the biCartesianness of $d_0T(f)$ and $d_2T(f).$ This can be done by two applications of the detection result Proposition \ref{prop_detect} to $i_1\colon K\nach J=[2]\times [1].$ It is easy to check (using Lemma \ref{lemma_someadj} in one of the cases) that the assumptions of that proposition are satisfied. Since $i_0$ is a sieve, the underlying diagram of $d_1T(f)$ and $d_2T(f)$ respectively look like:
$$
\xymatrix{
X\ar[r]\ar[d]&0\ar[d]&&X\ar[r]\ar[d]&Y\ar[d]\\
0\ar[r] & T(f)_{2,1}&&0\ar[r]&T(f)_{1,1}
}
$$
Moreover, by the proof of Proposition \ref{prop_stablesigma}, $d_1T(f)$ lies in the essential image of 
$$\xymatrix{
\DD(e)\ar[r]^{(0,0)_\ast}&\DD(\push)\ar[r]^{\ipush_!}&\DD(\square).
}
$$
Hence, we have a canonical isomorphism $T(f)_{2,1}\cong\Sigma X.$
Similarly, if we let $j\colon [1]\nach \push$ denote the functor classifying the upper horizontal morphism $d_2T(f)$ then lies in the essential image of 
$$
\xymatrix{
\DD([1])\ar[r]^{j_\ast}&\DD(\push)\ar[r]^{\ipush_!}&\DD(\square).
}
$$
Hence, we also have a canonical isomorphism $T(f)_{1,1}\cong \mathsf{C}(f)$ as intended.

Thus, for $f\in \DD([1]),$ by first restricting $T(f)$ to $[3]$ in the expected way and then forming the underlying diagram in $\DD(e)$, we obtain a triangle  $(T_f)$ in $\DD(e)$ which is of the following form:
$$
(T_f):\quad X\nach Y\nach \mathsf{C}(f)\nach \Sigma X
$$
Call a triangle in $\DD(e)$ \emph{distinguished} if it is isomorphic to $(T_f)$ for some $f\in\DD([1]).$ We are now in the position to state the following important theorem.

\begin{theorem}\label{thm_triang}
Let $\DD$ be a stable derivator and let $J$ be a category. Endowed with the suspension functor $\Sigma\colon \DD(J)\nach\DD(J)$ and the above class of distinguished triangles, $\DD(J)$ is a triangulated category.
\end{theorem}

\noindent
The fact that this triangulated structure is compatible with the restriction and homotopy Kan extension functors will be discussed in Corollary \ref{cor_canonicallyexact}. For easier reference to the axioms of a triangulated category we include a definition. For more background on this theory cf.\ for example~\cite{neeman} or to \cite{schwede_bookproject}. The form of the octahedron axiom given here is sufficient in order to obtain the usual form of the octahedron axiom. This observation was made in \cite{kellervossieck} (for a proof of it see \cite{schwede_bookproject}).

\begin{definition}
Let $\mathcal{T}$ be an additive category with a self-equivalence $\Sigma\colon \mathcal{T}\nach\mathcal{T}$ and a class of so-called distinguished triangles $X\nach Y\nach Z\nach \Sigma X.$ The pair consisting of $\Sigma$ and the class of distinguished triangles determines a \emph{triangulated structure} on $\mathcal{T}$ if the following four axioms are satisfied. In this case, the triple consisting of the category, the endofunctor, and the class of distinguished triangles is called a \emph{triangulated category}.\\
{\rm (T1)} For every $X\in\mathcal{T},$ the triangle $X\stackrel{\id}{\nach}X\nach 0\nach \Sigma X$ is distinguished. Every morphism in $\mathcal{T}$ occurs as the first morphism in a distinguished triangle and the class of distinguished triangles is replete, i.e., is closed under isomorphisms.\\
{\rm (T2)} A triangle $X\stackrel{f}{\nach} Y\stackrel{g}{\nach} Z\stackrel{h}{\nach} \Sigma X$ is distinguished if and only if the rotated triangle $Y\stackrel{g}{\nach} Z\stackrel{h}{\nach} \Sigma X\stackrel{-f}{\nach}\Sigma Y$ is.\\
{\rm (T3)} Given two distinguished triangles and a commutative solid arrow diagram
$$\xymatrix{
X\ar[r]\ar[d]^u& Y\ar[r]\ar[d]^v & Z\ar[r]\ar@{-->}[d]^w & \Sigma X\ar[d]^{\Sigma u}\\
X'\ar[r] & Y'\ar[r] & Z'\ar[r] & \Sigma X'
}
$$
there exists a dashed arrow $w\colon Z\nach Z'$ as indicated such that the extended diagram commutes.\\
{\rm (T4)} For every pair of composable arrows $f_3\colon X\stackrel{f_1}{\nach} Y\stackrel{f_2}{\nach} Z$ there is a commutative diagram
$$\xymatrix{
X\ar[r]^{f_1}\ar@{=}[d]& Y \ar[r]^{g_1}\ar[d]_{f_2}& C_1\ar[r]^{h_1}\ar[d]& \Sigma X\ar@{=}[d]\\
X\ar[r]_{f_3}& Z\ar[d]_{g_2}\ar[r]_{g_3}& C_3\ar[r]_{h_3}\ar[d]& \Sigma X\ar[d]^{\Sigma f_1}\\
&C_2\ar[d]_{h_2}\ar@{=}[r]&C_2\ar[d]^{\Sigma g_1\circ h_2}\ar[r]_{h_2}&\Sigma Y\\
&\Sigma Y\ar[r]_{\Sigma g_1}& \Sigma C_1& 
}
$$
in which the rows and columns are distinguished triangles.
\end{definition}

We will now give the proof of the theorem.

\begin{proof} (of Theorem \ref{thm_triang})\\
It suffices to do this for the case $J=e$. The additivity of $\DD(e)$ is already given by Corollary \ref{cor_additive}. Moreover, in this stable setting, the suspension functor $\Sigma$ is an equivalence.\\
\indent
(T1):  The first part of axiom (T1) is settled by Proposition \ref{prop_isocone} and the second part is  settled using the assumed strongness. The last part of (T1) holds by definition of the class of distinguished triangles.\\ 
\indent
(T3): Axiom (T3) is settled similarly by reducing first to the situation of triangles of the form~$(T_f)$ for $f\in \DD([1])$ and then applying the strongness again.\\
\indent
(T2): Before we give the actual proof of axiom (T2) we recall that the axioms of a triangulated category as given here are not in a minimal form. In fact, if one has already established axioms (T1) and (T3) it suffices to give a proof of one half of the rotation axiom as indicated in the next claim (cf.\ again to \cite{schwede_bookproject} for this fact).\\
\textit{Claim:} let $X\stackrel{f}{\nach} Y\stackrel{g}{\nach} Z\stackrel{h}{\nach} \Sigma X$ be a distinguished triangle in $\DD(e)$, then also the rotated triangle $Y\stackrel{g}{\nach} Z\stackrel{h}{\nach} \Sigma X\stackrel{-\Sigma f}{\nach} \Sigma Y$ is distinguished.\\
We can again reduce to the case where the given distinguished triangle is $(T_f)$ for some $f\in \DD([1]).$ Let us consider the category $J$ given by the following full subposet of $[2]\times[2]$
$$\xymatrix{
(0,0)\ar[r]\ar[d]& (1,0)\ar[r]\ar[dd]& (2,0)\\
(0,1)\ar@/_1.0pc/[dr]& & \\
& (1,2) &
}
$$
and let $i\colon [1]\nach J$ be the functor classifying the upper left horizontal morphism. Then $i$ is a sieve and $i_\ast$ gives us thus an extension by zero functor. Moreover, let us denote by $j$ the canonical inclusion of $J$ in $K=[2]\times [2]-\{(0,2)\}.$ For a given $f\in \DD([1])$ let us consider $j_!i_\ast(f)$. Again, by a repeated application of Proposition \ref{prop_detect} all squares in $j_!i_\ast(f)$ are biCartesian. If the diagram of $f$ is $f\colon X\nach Y$ then the underlying diagram of  $j_!i_\ast(f)$ looks like:
$$\xymatrix{
X\ar[r]^f\ar[d]& Y\ar[r]\ar[d]^g& 0\ar[d]\\
0\ar[r]& \mathsf{C}f\ar[r]^h\ar[d]& \Sigma X\ar[d]\\
& 0\ar[r]& \Sigma Y
}
$$
In fact, the inclusion $(d^1\times d^2)\colon \push\nach K$ allows us to identify the value at $(2,1)$ with $\Sigma X$ while the inclusion $(d^0\times d^1)\colon\push\nach K$ gives us an identification of the lower right corner with $\Sigma Y$. However, this last inclusion differs from the usual one by the automorphism $\sigma\colon\push\nach\push.$ By Proposition \ref{prop_sign}, the induced map $\sigma^\ast\colon \Sigma Y\nach \Sigma Y$ is $-\id_{\Sigma Y}$. Hence, using moreover the unique natural transformation of the two inclusions $(d^0\times d^1)\nach (d^1\times d^2)\colon\push\nach K,$ we can identify the morphism $\Sigma X\nach \Sigma Y$ as $-\Sigma f$ and this shows that the triangle $(T_g)$ is as stated in the claim.\\
(T4): It remains to give a proof of the octahedron axiom. The proof of this will be split into two parts.\\
{\rm i)} In the first part, given an object $F\in \DD([2]),$ we construct an associated octahedron diagram in~$\DD(e).$ The pattern of this part of the proof is by now quite familiar. Consider the category $J$ given by the following full subposet of $[4]\times [2]$
$$\xymatrix{
(0,0)\ar[r]\ar[d]& (1,0)\ar[r]\ar[dd]& (2,0)\ar[r]& (3,0)\ar@/^1.0pc/[dr]& \\
(0,1)\ar[rrrr]\ar@/_1.0pc/[dr]&&&&(4,1)\\
& (1,2) &&&
}
$$
and let $i\colon [2]\nach J$ classify the two composable upper left morphisms. Moreover, let 
$$j\colon J\nach K=[4]\times [2] -\{(4,0),(0,2)\}$$
be the canonical inclusion. Since $i$ is a sieve, the homotopy right Kan extension functor $i_\ast$ is an extension by zero functor. For $F\in \DD([2])$ let us consider $D=j_!i_\ast (F) \in \DD(K).$ If the underlying diagram of $F$ is $X\stackrel{f_1}{\nach}Y\stackrel{f_2}{\nach}Z$ then the underlying diagram of $D$ is
$$\xymatrix{
X\ar[r]^{f_1}\ar[d]& Y\ar[r]^{f_2}\ar[d] & Z\ar[d]\ar[r]& 0\ar[d]&\\
0\ar[r]& \widehat{C_1}\ar[r]\ar[d]& \widehat{C_3}\ar[r]\ar[d]& SX\ar[r]\ar[d]& 0\ar[d]\\
&0\ar[r]& \widehat{C_2}\ar[r]& SY\ar[r]& S\widehat{C_1}
}
$$
A repeated application of Proposition \ref{prop_detect} guarantees that all squares in $D$ are biCartesian. Hence the same is also true for all compound squares one can find in $D$. This allows us to find canonical isomorphisms $\widehat{C_k}\cong \mathsf{C}(f_k)$ if we set $f_3=f_2\circ f_1.$ More precisely, the cone functor $\mathsf{C}$ has of course to be applied to $f_1=d_2(F),\:f_2=d_0(F),$ and $f_3=d_1(F)\in \DD([1]).$ Similarly, we obtain isomorphisms $SX\cong \Sigma X,\:SY\cong \Sigma Y,$ and $S\widehat{C_1}\cong\Sigma \widehat{C_1}.$ Thus, one can extract an octahedron diagram in $\DD(e)$ from the object $D.$\\
{\rm ii)} In this part, we show that every `first half of an octahedron diagram' comes up to isomorphism from an object $F\in \DD([2]).$ Let us restrict attention to the upper left square 
$$\xymatrix{
X\ar[r]^{f_1}\ar@{=}[d]& Y\ar[d]^{f_2}\\
X\ar[r]_{f_3}& Z
}
$$
of such a diagram. The strongness of $\DD$ guarantees that there is an object $F_1\in \DD([1])$ and an isomorphism $\dia F_1\cong (f_1\colon X\nach Y).$ Moreover, let us consider $p^\ast Z\in \DD([1]),$ where $p\colon [1]\nach e$ is the unique functor. Then, we obtain a morphism $\phi\colon F_1\nach p^\ast Z$ as the image of $f_2$ under the two natural isomorphisms (we applied Lemma \ref{lemma_terminal} to obtain the second one): 
$$
\hhom _{\DD(e)}(Y,Z)\quad\cong\quad \hhom _{\DD([1])}(F_1,1_\ast Z)\quad\cong\quad \hhom _{\DD([1])}(F_1,p^\ast Z)
$$
Considering this map $\phi\colon F_1\nach p^\ast Z$ as an object of $\DD([1])^{[1]},$ a further application of the strongness guarantees the existence of an object $Q\in \DD(\square)$ such that $\dia _{[1],[1]}Q\cong (\phi\colon F_1\nach p^\ast Z)$:
$$\xymatrix{
\dia Q:&X\ar[d]_{f_1}\ar[r]^{\phi_0}&Z\ar[d]\\
&Y\ar[r]_{\phi_1}&Z
}
$$
If $i\colon [2]\nach \square$ classifies the non-degenerate pair of composable arrows passing through the lower left corner $(0,1)$ then let us set $F=i^\ast Q\in \DD([2]).$ This $F$ does the job.
\end{proof}

From now on, whenever we consider the values of a stable derivator as triangulated categories we will always mean the triangulated structure of Theorem \ref{thm_triang}. The next aim is to show that the functors belonging to a stable derivator can be canonically made into exact functors with respect to these structures. In the stable setting, Corollary \ref{cor_additive} induces immediately the following one.

\begin{corollary}\label{cor_exactadd}
Let $F\colon \DD\nach \DD'$ be a morphism of stable derivators, then:
$$F \mbox{ is left exact}\quad\iff\quad F  \mbox{ is exact}\quad\iff\quad F \mbox{ is right exact}$$
In particular, the components $F_J\colon \DD(J)\nach\DD'(J)$ of an exact morphism are additive functors.
\end{corollary}

Exact morphisms are the `correct' morphisms for stable derivators. Some evidence for this is given by the next result.

\begin{proposition}\label{prop_triang}
Let $F\colon \DD\nach \DD'$ be an exact morphism of stable derivators and let $J$ be a category. The functor $F_J\colon \DD(J)\nach \DD'(J)$ can be canonically endowed with the structure of an exact functor of triangulated categories.
\end{proposition}
\begin{proof}
By Proposition \ref{prop_againtriang}, we can assume without loss of generality that $J=e.$ Moreover, by definition, $F$ preserves zero objects and coCartesian squares. In particular, coCartesian squares such that the two off-diagonal entries vanish are preserved by $F.$ This gives us the canonical isomorphism $F\circ \Sigma \cong \Sigma \circ F.$ Similarly, $F$ preserves composites of two coCartesian squares. In particular, among the composites those which vanish at $(2,0)$ and $(0,1)$ are preserved. These were used to define the class of distinguished triangles in the canonical triangulated structures from where it follows that~$F$ together with the canonical isomorphism $F\circ\Sigma\cong\Sigma\circ F$ is exact.
\end{proof}

This result can now be applied to Example \ref{example_exact}. In particular, we can deduce that the functors belonging to a stable derivator respect the canonical triangulated structures we just constructed.

\begin{corollary}\label{cor_canonicallyexact}
Let $\DD$ be a stable derivator and let $u\colon J\nach K$ be a functor. The induced functors $u^\ast\colon \DD(K)\nach \DD(J)$ and $u_!,\:u_\ast\colon\DD(J)\nach\DD(K)$ can be canonically endowed with the structure of exact functors.
\end{corollary}
\begin{proof}
Since we have adjunctions $(u_!,u^\ast)$ and $(u^\ast,u_\ast)$, it suffices to show that $u^\ast$ can be canonically endowed with the structure of an exact functor (cf.\ \cite[p.463]{margolis}). But this functor $u^\ast$ can be considered as $u^\ast\colon \DD^K(e)\nach\DD^J(e)$ and hence the result follows by a combination of the last proposition and Example \ref{example_exact}.
\end{proof}

\begin{remark}\label{rem_advantages}
Theorem \ref{thm_triang} and Proposition \ref{prop_triang} reveal certain advantages of the language of stable derivators over the language of triangulated categories. A triangulated category $\mathcal{T}$ is, by the very definition, a triple consisting of a category $\mathcal{T}$ together with a functor $\Sigma\colon \mathcal{T}\nach\mathcal{T}$ and a class of distinguished triangle as additionally \emph{specified structure}. These are then subject to a list of axioms. One advantage of the stable derivators is that this structure does not have to be specified but instead is canonically available. Once the derivator is stable, i.e., has some easily motivated \emph{properties}, triangulated structures can be canonically constructed. In particular, the octahedron axiom does not have to be made explicit. 

Similarly, the fact that a morphism $F$ of triangulated categories is exact means, by the very definition, that the functor is endowed with an \emph{additional structure} given by a natural isomorphism $\sigma\colon F\circ \Sigma\nach \Sigma \circ F$ which behaves nicely with respect to the two chosen classes of distinguished triangles. But, in fact, the exactness of such a morphism should only be a property and not a structure. In most applications, the exact functors under consideration are `derived functors' of functors defined `on certain models in the background'. And in this situation, the exactness then reflects the fact that this functor preserves (certain) finite homotopy (co)limits. In the setting of stable derivators this is precisely the notion of an exact morphism. In particular, the exactness of a morphism is again a property and not the specification of an additional structure.

These same advantages are also shared by stable $\infty-$categories as studied in detail in Lurie's book \cite{HA}. A short introduction to that theory can be found in \cite[Section 5]{groth_infinity}.
\end{remark}

We are now basically done with the development of the theory of (stable) derivators. So let us analyze what conditions on a $2$-subcategory $\Dia\subseteq \Cat$ have to be imposed in order to be able to also deduce the same results for (stable) derivators of type $\Dia.$ By the very definition of a derivator, we need that the empty category and the terminal category belong to $\Dia$. Moreover, it has to be closed under finite coproducts to give sense to axiom (Der1). Furthermore, we frequently reduced situations to the case of the underlying category by using the passage from $\DD$ to $\DD^M$. Thus, $\Dia$ has also to be closed under products. We also used various finite posets as admissible shapes in the proofs of this section so we should ask axiomatically for a sufficient supply of them. Finally, $\Dia$ has to be closed under the slice construction since we impose axiomatically Kan's formula. There is the following definition of a diagram category which we cite from \cite{cisinskineeman}. In particular, this notion has the closure properties we used in the development of the theory.

\begin{definition}\label{def_diagram}
A full $2$-subcategory $\Dia \subseteq \Cat$ is called a \emph{diagram category} if it satisfies the following axioms:\\
$\bullet$ All finite posets considered as categories belong to $\Dia$.\\
$\bullet$ For every $J\in\Dia$ and every $j\in J$, the slice constructions $J_{j/}$ and $J_{/j}$ belong to $\Dia$.\\
$\bullet$ If $J\in \Dia$ then also $J^{\op}\in\Dia$.\\
$\bullet$ For every Grothendieck fibration $u\colon J\nach K,$ if all fibers $J_k,\:k\in K,$ and the base $K$ belong to $\Dia$ then also $J$ lies in $\Dia$.
\end{definition}
\noindent
With this notion one can now define prederivators and (pointed, stable) derivators of type $\Dia$ as $2$-functors $\Dia^{\op}\nach \CAT$ satisfying the corresponding axioms. We leave it to the reader to check that all results we established so far can also be proved in that more general situation.

\begin{example}
The full $2$-subcategory of finite posets is the smallest diagram category, $\Cat$ itself is the largest one. Further examples are given by the full $2$-subcategories spanned by the finite categories or the finite-dimensional categories. Moreover, the intersection of a family of diagram categories is again a diagram category.
\end{example}

\subsection{Recollements of triangulated categories}\label{subsection_recollement}

In this short subsection, we mainly mention that sieves and cosieves give rise to recollements of triangulated categories in the context of a stable derivator. This can be used to reprove (in the stable case) that the (co)exceptional inverse image functors show up for free. 

We begin with a very short recap of the theory of recollements of triangulated categories. For classical examples of recollements in algebraic geometry cf.\ \cite{beilinson}, for a very nice modern treatment cf.\ also to the thesis of Heider \cite{heider}. Recollements capture axiomatically the situation in which we are given three triangulated categories $\mathcal{T}',\:\mathcal{T},$ and $\mathcal{T}''$ such that every object of $\mathcal{T}$ can be obtained as an extension of an object of $\mathcal{T}''$ by an object of $\mathcal{T}'$ and vice-versa. More precisely, there is the following definition.

\begin{definition}
A \emph{recollement of triangulated categories} is a diagram of triangulated categories and exact functors\\
$$
\xymatrix{
\mathcal{T}'\ar[rr]^{i_!}&& \mathcal{T}\ar[rr]^{j^\ast} \ar@/^1.2pc/[ll]^{i^\ast} \ar@/_1.2pc/[ll]_{i^?}&& \mathcal{T}''\ar@/^1.2pc/[ll]^{j_\ast} \ar@/_1.2pc/[ll]_{j_!}
}
$$
such that the following properties hold:\\
$\bullet$ the pairs $(i^?,i_!),\:(i_!, i^\ast),\:(j_!,j^\ast),$ and $(j^\ast,j_\ast)$ are adjunctions\\
$\bullet$ $j^\ast i_!=0$\\
$\bullet$ the functors $i_!,\: j_!,$ and $j_\ast$ are fully faithful and\\
$\bullet$ every object $X\in\mathcal{T}$ sits in two distinguished triangles of the form
$$
\xymatrix{
i_! i^\ast X\ar[r]& X\ar[r] & j_\ast j^\ast X\ar[r]& \Sigma i_! i^\ast X,&
j_!j^\ast X\ar[r]& X\ar[r] & i_! i^? X \ar[r]& \Sigma j_!j^\ast X
}
$$
where in both triangles the first two arrows are the respective adjunction morphisms.
\end{definition}

One can show that in this situation $\mathcal{T}'=\ker j^\ast$ and that $\mathcal{T}''$ is the Verdier quotient $\mathcal{T}/\mathcal{T}'$ (\cite{heider}). The latter follows immediately from the first since by definition a recollement gives us a reflective localization and a coreflective colocalization (\cite{krause}). Let us remark further that this definition is not given in a minimal form but is overdetermined. Recall from classical category theory that if a functor admits an adjoint on both sides then if one of the adjoints is fully faithful then this is also the case for the other one (\cite[Prop. 3.4,2]{borceux1}). And, even more interesting for us, it suffices to only have the right half of a recollement. More precisely, there is the following result (\cite[Proposition~1.14]{heider}).

\begin{proposition}
Consider a diagram of triangulated categories and exact functors
$$\xymatrix{
\mathcal{T}\ar[rr]^{j^\ast} && \mathcal{T}''\ar@/^1.2pc/[ll]^{j_\ast} \ar@/_1.2pc/[ll]_{j_!}
}
$$
\noindent 
such that $(j_!,j^\ast)$ and $(j^\ast,j_\ast)$ are adjunctions and one of the two functors $j_!,\:j_\ast$ is fully faithful. If we denote by $\mathcal{T}'$ the kernel of $j^\ast$ and by $i_! \colon \mathcal{T}'\nach \mathcal{T}$ the inclusion then the above diagram can be extended to a recollement:
$$
\xymatrix{
\mathcal{T}'\ar[rr]^{i_!}&& \mathcal{T}\ar[rr]^{j^\ast} \ar@/^1.2pc/[ll]^{i^\ast} \ar@/_1.2pc/[ll]_{i^?}&& \mathcal{T}''\ar@/^1.2pc/[ll]^{j_\ast} \ar@/_1.2pc/[ll]_{j_!}
}
$$
\end{proposition}

In the context of a stable derivator, there is the following class of examples.

\begin{example}
Let $\DD$ be a stable derivator and consider a sieve $j\colon U\nach X$. Moreover, let $Z$ be the full subcategory of $X$ spanned by the objects which are not in the image of $j$. Then the inclusion $i\colon Z\nach X$ is a cosieve. Moreover, by the fully faithfulness of homotopy Kan extensions along fully faithful functors and by Proposition \ref{prop_extzero}, the last proposition gives us the following recollements:
$$
\xymatrix{
\DD(U)\ar[rr]^{j_\ast}&& \DD(X)\ar[rr]^{i^\ast} \ar@/^1.2pc/[ll]^{j^!} \ar@/_1.2pc/[ll]_{j^\ast}&& \DD(Z)\ar@/^1.2pc/[ll]^{i_\ast} \ar@/_1.2pc/[ll]_{i_!}&
\DD(Z)\ar[rr]^{i_!}&& \DD(X)\ar[rr]^{j^\ast} \ar@/^1.2pc/[ll]^{i^\ast} \ar@/_1.2pc/[ll]_{i^?}&& \DD(U)\ar@/^1.2pc/[ll]^{j_\ast} \ar@/_1.2pc/[ll]_{j_!}
}
$$
\end{example}

This example shows that for a sieve $j\colon U\nach X$ (resp.\ for a cosieve $i\colon Z\nach X$) the additional adjoint functor $j^!\colon \DD(X)\nach \DD(U)$ (resp.\ $i^?\colon \DD(X)\nach \DD(Z)$) shows up for free in the above recollements. Thus, this reproves, in the stable case, that a pointed derivator admits (co)exceptional inverse image functors.

\bibliographystyle{alpha}
\bibliography{derivators_1}

\def\cprime{$'$}
\begin{thebibliography}{Dug01b}

\bibitem[AR94]{adamekrosicky}
Ji{\v{r}}{\'{\i}} Ad{\'a}mek and Ji{\v{r}}{\'{\i}} Rosick{\'y}.
\newblock {\em Locally presentable and accessible categories}, volume 189 of
  {\em London Mathematical Society Lecture Note Series}.
\newblock Cambridge University Press, Cambridge, 1994.

\bibitem[Ayo07a]{ayoub1}
Joseph Ayoub.
\newblock Les six op\'erations de {G}rothendieck et le formalisme des cycles
  \'evanescents dans le monde motivique. {I}.
\newblock {\em Ast\'erisque}, (314):x+466 pp. (2008), 2007.

\bibitem[Ayo07b]{ayoub2}
Joseph Ayoub.
\newblock Les six op\'erations de {G}rothendieck et le formalisme des cycles
  \'evanescents dans le monde motivique. {II}.
\newblock {\em Ast\'erisque}, (315):vi+364 pp. (2008), 2007.

\bibitem[BBD82]{beilinson}
Alexander Be{\u\i}linson, Joseph Bernstein, and Pierre Deligne.
\newblock Faisceaux pervers.
\newblock In {\em Analysis and topology on singular spaces, {I} ({L}uminy,
  1981)}, volume 100 of {\em Ast\'erisque}, pages 5--171. Soc. Math. France,
  Paris, 1982.

\bibitem[Bek00]{beke}
Tibor Beke.
\newblock Sheafifiable homotopy model categories.
\newblock {\em Math. Proc. Cambridge Philos. Soc.}, 129(3):447--475, 2000.

\bibitem[BK72]{bousfieldkan}
Aldridge~Knight Bousfield and Daniel~Marinus Kan.
\newblock {\em Homotopy limits, completions and localizations}.
\newblock Lecture Notes in Mathematics, Vol. 304. Springer-Verlag, Berlin,
  1972.

\bibitem[Bor94a]{borceux1}
Francis Borceux.
\newblock {\em Handbook of categorical algebra. 1}, volume~50 of {\em
  Encyclopedia of Mathematics and its Applications}.
\newblock Cambridge University Press, Cambridge, 1994.
\newblock Basic category theory.

\bibitem[Bor94b]{borceux2}
Francis Borceux.
\newblock {\em Handbook of categorical algebra. 2}, volume~51 of {\em
  Encyclopedia of Mathematics and its Applications}.
\newblock Cambridge University Press, Cambridge, 1994.
\newblock Categories and structures.

\bibitem[Bou75]{Bousfield}
Aldridge~Knight Bousfield.
\newblock The localization of spaces with respect to homology.
\newblock {\em Topology}, 14:133--150, 1975.

\bibitem[BV73]{BoardVogt}
Michael Boardman and Rainer Vogt.
\newblock {\em Homotopy invariant algebraic structures on topological spaces}.
\newblock Lecture Notes in Mathematics, Vol. 347. Springer-Verlag, Berlin,
  1973.

\bibitem[Cis03]{cisinski}
Denis-Charles Cisinski.
\newblock Images directes cohomologiques dans les cat\'egories de mod\`eles.
\newblock {\em Ann. Math. Blaise Pascal}, 10(2):195--244, 2003.

\bibitem[Cis08]{cisinski_derivedkan}
Denis-Charles Cisinski.
\newblock Propri\'et\'es universelles et extensions de {K}an d\'eriv\'ees.
\newblock {\em Theory Appl. Categ.}, 20:No. 17, 605--649, 2008.

\bibitem[CN08]{cisinskineeman}
Denis-Charles Cisinski and Amnon Neeman.
\newblock Additivity for derivator {$K$}-theory.
\newblock {\em Adv. Math.}, 217(4):1381--1475, 2008.

\bibitem[Cor82]{cordier}
Jean-Marc Cordier.
\newblock Sur la notion de diagramme homotopiquement coh\'erent.
\newblock {\em Cahiers Topologie G\'eom. Diff\'erentielle}, 23(1):93--112,
  1982.
\newblock Third Colloquium on Categories, Part VI (Amiens, 1980).

\bibitem[DS95]{DwyerSpalinski}
William~Gerard Dwyer and Jan Spali{\'n}ski.
\newblock Homotopy theories and model categories.
\newblock In {\em Handbook of algebraic topology}, pages 73--126.
  North-Holland, Amsterdam, 1995.

\bibitem[Dug01a]{dugger_combinatorial}
Daniel Dugger.
\newblock Combinatorial model categories have presentations.
\newblock {\em Adv. Math.}, 164(1):177--201, 2001.

\bibitem[Dug01b]{dugger_universal}
Daniel Dugger.
\newblock Universal homotopy theories.
\newblock {\em Adv. Math.}, 164(1):144--176, 2001.

\bibitem[Ehr63]{ehresmann}
Charles Ehresmann.
\newblock Cat\'egories structur\'ees.
\newblock {\em Ann. Sci. \'Ecole Norm. Sup.}, 80:349--426, 1963.

\bibitem[FR08]{rosicky_convenient}
Lisbeth Fajstrup and Ji{\v{r}}{\'{\i}} Rosick{\'y}.
\newblock A convenient category for directed homotopy.
\newblock {\em Theory Appl. Categ.}, 21:No. 1, 7--20, 2008.

\bibitem[Fra96]{franke}
Jens Franke.
\newblock Uniqueness theorems for certain triangulated categories with an
  {A}dams spectral sequence, 1996.
\newblock Preprint.

\bibitem[Gro]{grothendieck}
Alexander Grothendieck.
\newblock Les d\'erivateurs.
\newblock \url{http://www.math.jussieu.fr/~maltsin/groth/Derivateurs.html}.
\newblock Manuscript.

\bibitem[Gro10]{groth_infinity}
Moritz Groth.
\newblock A short course on $\infty$-categories.
\newblock \url{http://arxiv.org/abs/1007.2925}, 2010.
\newblock Preprint.

\bibitem[Gro11]{groth_monder}
Moritz Groth.
\newblock Monoidal derivators.
\newblock \url{http://www.math.ru.nl/~mgroth}, 2011.
\newblock Preprint.

\bibitem[Gro12a]{groth_enriched}
Moritz Groth.
\newblock Enriched derivators, 2012.
\newblock In preparation.

\bibitem[Gro12b]{groth_scderivator}
Moritz Groth.
\newblock A short course on derivators, 2012.
\newblock In preparation.

\bibitem[GU71]{gabrielulmer}
Peter Gabriel and Friedrich Ulmer.
\newblock {\em Lokal pr\"asentierbare {K}ategorien}.
\newblock Lecture Notes in Mathematics, Vol. 221. Springer-Verlag, Berlin,
  1971.

\bibitem[Hap88]{happel_triangulated}
Dieter Happel.
\newblock {\em Triangulated categories in the representation theory of
  finite-dimensional algebras}, volume 119 of {\em London Mathematical Society
  Lecture Note Series}.
\newblock Cambridge University Press, Cambridge, 1988.

\bibitem[Hei07]{heider}
Andreas Heider.
\newblock Two results from {M}orita theory of stable model categories.
\newblock \url{http://arxiv.org/abs/0707.0707}, 2007.
\newblock Preprint.

\bibitem[Hel88]{heller}
Alex Heller.
\newblock Homotopy theories.
\newblock {\em Mem. Amer. Math. Soc.}, 71(383):vi+78, 1988.

\bibitem[Hel97]{heller_stable}
Alex Heller.
\newblock Stable homotopy theories and stabilization.
\newblock {\em J. Pure Appl. Algebra}, 115(2):113--130, 1997.

\bibitem[Hir03]{hirschhorn}
Philip~Steven Hirschhorn.
\newblock {\em Model categories and their localizations}, volume~99 of {\em
  Mathematical Surveys and Monographs}.
\newblock American Mathematical Society, Providence, RI, 2003.

\bibitem[Hov99]{hovey}
Mark Hovey.
\newblock {\em Model categories}, volume~63 of {\em Mathematical Surveys and
  Monographs}.
\newblock American Mathematical Society, Providence, RI, 1999.

\bibitem[Joy]{joyal4}
Andr{\'e} Joyal.
\newblock Notes on quasi-categories.
\newblock Preprint.

\bibitem[Joy08a]{joyal3}
Andr{\'e} Joyal.
\newblock The theory of quasi-categories and its applications, 2008.
\newblock Lectures at CRM Barcelona, www.crm.cat/HigherCategories/hc2.pdf.

\bibitem[Joy08b]{joyal1}
Andr{\'e} Joyal.
\newblock The theory of quasi-categories {I}, to appear, 2008.
\newblock Preprint.

\bibitem[Joy08c]{joyal2}
Andr{\'e} Joyal.
\newblock The theory of quasi-categories {II}, to appear, 2008.
\newblock Preprint.

\bibitem[JT91]{joyaltierney_stacks}
Andr{\'e} Joyal and Myles Tierney.
\newblock Strong stacks and classifying spaces.
\newblock In {\em Category theory ({C}omo, 1990)}, volume 1488 of {\em Lecture
  Notes in Math.}, pages 213--236. Springer, Berlin, 1991.

\bibitem[Kan58]{kan_adjoint}
Daniel~Marinus Kan.
\newblock Adjoint functors.
\newblock {\em Trans. Amer. Math. Soc.}, 87:294--329, 1958.

\bibitem[Kel91]{keller_universal}
Bernhard Keller.
\newblock Derived categories and universal problems.
\newblock {\em Comm. Algebra}, 19:699--747, 1991.

\bibitem[Kel07]{keller_exact}
Bernhard Keller.
\newblock Appendice: {L}e d\'erivateur triangul\'e associ\'e \`a une
  cat\'egorie exacte.
\newblock In {\em Categories in algebra, geometry and mathematical physics},
  volume 431 of {\em Contemp. Math.}, pages 369--373. Amer. Math. Soc.,
  Providence, RI, 2007.

\bibitem[Kra10]{krause}
Henning Krause.
\newblock Localization theory for triangulated categories.
\newblock In {\em Triangulated categories}, volume 375 of {\em London Math.
  Soc. Lecture Note Ser.}, pages 161--235. Cambridge Univ. Press, Cambridge,
  2010.

\bibitem[KS05]{kelly_2cat}
Gregory~Maxwell Kelly and Ross Street.
\newblock Review of the elements of 2-categories.
\newblock {\em Repr. Theory Appl. Categ.}, pages vi+137 pp. (electronic), 2005.
\newblock Reprint of the 1982 original [Cambridge Univ. Press, Cambridge;
  MR0651714].

\bibitem[KV87]{kellervossieck}
Bernhard Keller and Dieter Vossieck.
\newblock Sous les cat\'egories d\'eriv\'ees.
\newblock {\em C. R. Acad. Sci. Paris S\'er. I Math.}, 305(6):225--228, 1987.

\bibitem[Lac02]{lack_model2categories}
Stephen Lack.
\newblock A {Q}uillen model structure for 2-categories.
\newblock {\em $K$-Theory}, 26(2):171--205, 2002.

\bibitem[Lac04]{lack_modelbicategories}
Stephen Lack.
\newblock A {Q}uillen model structure for bicategories.
\newblock {\em $K$-Theory}, 33(3):185--197, 2004.

\bibitem[Lac10]{lack_2catcompanion}
Stephen Lack.
\newblock A 2-categories companion.
\newblock In {\em Towards higher categories}, volume 152 of {\em IMA Vol. Math.
  Appl.}, pages 105--191. Springer, New York, 2010.

\bibitem[Lur09]{HTT}
Jacob Lurie.
\newblock {\em Higher topos theory}, volume 170 of {\em Annals of Mathematics
  Studies}.
\newblock Princeton University Press, Princeton, NJ, 2009.

\bibitem[Lur11]{HA}
Jacob Lurie.
\newblock Higher algebra.
\newblock \url{http://www.math.harvard.edu/~lurie/}, 2011.
\newblock Preprint.

\bibitem[Mal01]{maltsiniotis1}
Georges Maltsiniotis.
\newblock Introduction \`a la th\'eorie des d\'erivateurs (d'apr\`es
  {G}rothendieck).
\newblock \url{http://people.math.jussieu.fr/~maltsin/textes.html}, 2001.
\newblock Preprint.

\bibitem[Mal07]{maltsiniotis2}
Georges Maltsiniotis.
\newblock La {$K$}-th\'eorie d'un d\'erivateur triangul\'e.
\newblock In {\em Categories in algebra, geometry and mathematical physics},
  volume 431 of {\em Contemp. Math.}, pages 341--368. Amer. Math. Soc.,
  Providence, RI, 2007.

\bibitem[Mal11]{maltsiniotis_exact}
Georges Maltsiniotis.
\newblock Carr\'es exacts homotopiques, et d\'erivateurs.
\newblock \url{http://arxiv.org/pdf/1101.4144v1}, 2011.
\newblock Preprint.

\bibitem[Mar83]{margolis}
Harvey~Robert Margolis.
\newblock {\em Spectra and the {S}teenrod algebra}, volume~29 of {\em
  North-Holland Mathematical Library}.
\newblock North-Holland Publishing Co., Amsterdam, 1983.
\newblock Modules over the Steenrod algebra and the stable homotopy category.

\bibitem[ML98]{maclane}
Saunders Mac~Lane.
\newblock {\em Categories for the working mathematician}, volume~5 of {\em
  Graduate Texts in Mathematics}.
\newblock Springer-Verlag, New York, second edition, 1998.

\bibitem[Nee01]{neeman}
Amnon Neeman.
\newblock {\em Triangulated categories}, volume 148 of {\em Annals of
  Mathematics Studies}.
\newblock Princeton University Press, Princeton, NJ, 2001.

\bibitem[Qui67]{quillen}
Daniel~Gray Quillen.
\newblock {\em Homotopical algebra}.
\newblock Lecture Notes in Mathematics, No. 43. Springer-Verlag, Berlin, 1967.

\bibitem[Qui73]{quillen_ktheory}
Daniel~Gray Quillen.
\newblock Higher algebraic {$K$}-theory. {I}.
\newblock In {\em Algebraic {$K$}-theory, {I}: {H}igher {$K$}-theories ({P}roc.
  {C}onf., {B}attelle {M}emorial {I}nst., {S}eattle, {W}ash., 1972)}, pages
  85--147. Lecture Notes in Math., Vol. 341. Springer, Berlin, 1973.

\bibitem[Ren09]{renaudin}
Olivier Renaudin.
\newblock Plongement de certaines th\'eories homotopiques de {Q}uillen dans les
  d\'erivateurs.
\newblock {\em J. Pure Appl. Algebra}, 213(10):1916--1935, 2009.

\bibitem[Rez]{rezk_naturalmodel}
Charles Rezk.
\newblock A model category for categories.
\newblock \url{http://www.math.uiuc.edu/~rezk/papers.html}.
\newblock Preprint.

\bibitem[Ros09]{rosicky_comb}
Ji{\v{r}}{\'{\i}} Rosick{\'y}.
\newblock On combinatorial model categories.
\newblock {\em Appl. Categ. Structures}, 17(3):303--316, 2009.

\bibitem[Sch07]{schwede_bookproject}
Stefan Schwede.
\newblock An untitled book project about symmetric spectra.
\newblock \url{http://www.math.uni-bonn.de/people/schwede}, 2007.
\newblock Preprint.

\bibitem[Seg74]{segal_categories}
Graeme Segal.
\newblock Categories and cohomology theories.
\newblock {\em Topology}, 13:293--312, 1974.

\bibitem[Sim07]{Simpson}
Carlos Simpson.
\newblock A {G}iraud-type characterization of the simplicial categories
  associated to closed model categories as $\infty$-pretopoi.
\newblock \url{math.AT/9903167}, 2007.
\newblock Preprint.

\bibitem[Str96]{street_categoricalstructures}
Ross Street.
\newblock Categorical structures.
\newblock In {\em Handbook of algebra, {V}ol.\ 1}, pages 529--577.
  North-Holland, Amsterdam, 1996.

\bibitem[Vis05]{vistoli}
Angelo Vistoli.
\newblock Grothendieck topologies, fibered categories and descent theory.
\newblock In {\em Fundamental algebraic geometry}, volume 123 of {\em Math.
  Surveys Monogr.}, pages 1--104. Amer. Math. Soc., Providence, RI, 2005.

\end{thebibliography}

\end{document}